\documentclass{article}
\usepackage{latexsym,amsfonts,amsthm,amsmath,amscd,amssymb}
\usepackage[dvips]{graphicx}

\newtheorem{lemma}{Lemma}[section]
\newtheorem{thm}[lemma]{Theorem}
\newtheorem{rem}[lemma]{Remark}
\newtheorem{prop}[lemma]{Proposition}
\newtheorem{cor}[lemma]{Corollary}

\newtheorem{example}[lemma]{Example}
\newtheorem{defn}[lemma]{Definition}

\newcommand{\cl}{C \kern -0.1em \ell}

\newcommand\matR{{\mathbb{R}}}

\renewcommand{\hbar}{{\overline{h}}}

\newfont{\Got}{eufm10 scaled 1200}

\newcommand\calD{{\mathcal D}}

\begin{document}

\title{Pseudo-bundles of exterior algebras as diffeological Clifford modules}

\author{Ekaterina~{\textsc Pervova}}

\maketitle

\begin{abstract}
\noindent We consider the diffeological pseudo-bundles of exterior algebras, and the Clifford action of the corresponding Clifford algebras, associated to a given finite-dimensional and locally trivial
diffeological vector pseudo-bundle, as well as the behavior of the former three constructions (exterior algebra, Clifford action, Clifford algebra) under the diffeological gluing of pseudo-bundles. Despite these
being our main object of interest, we dedicate significant attention to the issues of compatibility of pseudo-metrics, and the gluing-dual commutativity condition, that is, the condition ensuring that the dual of the
result of gluing together two pseudo-bundles can equivalently be obtained by gluing together their duals, which is not automatic in the diffeological context. We show that, assuming that the dual of the gluing
map, which itself does not have to be a diffeomorphism, on the total space is one, the commutativity condition is satisfied, via a natural map, which in addition turns out to be an isometry for the
natural pseudo-metrics on the pseudo-bundles involved.

\end{abstract}

\section*{Introduction}

This work is intended as a supplement to \cite{clifford-alg}, dealing with some issues regarding pseudo-bundles of exterior algebras associated to finite-dimensional diffeological vector pseudo-bundles,
viewed as pseudo-bundles of diffeological Clifford modules, so endowed with the action of the corresponding diffeological pseudo-bundles of Clifford algebras. Let us explain as briefly as possible what all
these objects are; the precise definitions are to be found in a dedicated section, or else in the references given therein.

First of all, the \emph{diffeology}. The notion is due to J.M. Souriau \cite{So1}, \cite{So2} and is a categorical extension of the notion of a smooth structure; in essence, or maybe as an example, it is a way to
consider a topological space which is in no way a manifold, as if it were one. In and of itself, a diffeology on a set is a collection of maps into this set which are declared to be smooth; this collection must satisfy
certain conditions. The set is then a diffeological space, and all the basic constructions follow; there is a notion of smooth maps between two diffeological spaces, that of the underlying topology, the so-called
\emph{D-topology} (introduced in \cite{iglFibre}, see also \cite{CSW_Dtopology} for a recent treatment), and so on. The notion builds on existing ones, such as those of \emph{differentiable spaces},
\emph{V-manifolds}, and so on (see, for instance, \cite{satake1957}, \cite{chen1977}, to name a few). A particularly important aspect of diffeology is that all the usual topological constructions, notably subsets
and quotients, have an inherited diffeological structure (unlike the case of smooth manifolds, where subsets and quotients are quite rarely smooth manifolds themselves).

The basic object for us, though, is not just a diffeological space, but a \emph{diffeological vector pseudo-bundle} (\cite{iglFibre}, \cite{vincent}, \cite{iglesiasBook}, \cite{pseudobundles}). The difference with
respect to the standard notion is not only in the fact that the smooth structure is replaced by a diffeological one (under some respects this would be a minor difference), but also in that it does not have to be
locally trivial, although in many contexts we do add this assumption, see the discussion of  \emph{pseudo-metrics}, which replace the usual Riemannian metrics --- diffeological pseudo-bundles frequently
do not carry the latter.

On the other hand, as explained in the references listed above, the usual operations on vector bundles have their diffeological counterparts; in particular, direct sums, tensor products, and taking duals all
apply. From this, obtaining pseudo-bundles of tensor algebras, those of exterior algebras, or defining, in the abstract, pseudo-bundles of Clifford modules is automatic.

The point of view that we take in this paper has to do with studying the behavior of these concepts under the operation of \emph{diffeological gluing}. This procedure is one of the many possible extensions of
the concept of an atlas on a smooth manifold, and the resulting spaces are among the more obvious extensions of smooth manifolds and include some well-known singular spaces; for instance, a manifold
with a conical singularity can be seen as a result of gluing of a usual smooth manifold to a single-point space.

For the basic operations on diffeological vector pseudo-bundles, the behavior under diffeological gluing was considered in \cite{pseudobundles} (see also \cite{pseudometric-pseudobundle} for some details);
for tensor algebras and pairs of given Clifford modules, in \cite{clifford-alg}. What is lacking is a study of gluing of pseudo-bundles of the exterior algebras, in particular, the covariant version. This paper aims to
fill this void (another motivation for it is to provide some necessary building blocks for defining the notion of a diffeological Dirac operator and studying its behavior under gluing, see \cite{dirac}).

\paragraph{The content} Section 1 goes over the main definitions used and introduces notation. In Section 2 we consider the compatibility of pseudo-metrics in terms of the assumptions on the gluing map(s);
in Section 3 we relate this to the gluing-dual commutativity condition, and in Section 4 we show that the compatibility of dual pseudo-metrics implies that the commutativity condition must be satisfied. In
Section 5 we show that the gluing-dual commutativity diffeomorphism is an isometry. All these allow us to consider, in Sections 6-8, Clifford algebras (the covariant case), the exterior algebras, and the
corresponding Clifford actions; in particular, in Section 8 we establish, where appropriate, several equivalences showing that, again under the gluing-dual commutativity assumption, everything reduces to two
cases, the contravariant case and the covariant one. Section 9 contains a couple of simple (but necessarily lengthy) examples.

\paragraph{Acknowledgments} This work benefitted from some assistance, for which I would like to thank\\ Prof.~Riccardo Zucchi, even if he always says that he is being overvalued, some of his doctorate
students (Martina and Leonardo), although they do not expect this at all, and also Prof.~Mario Petrini (he most definitely will be surprised).

\section{Main definitions and known facts}

We now go, as briefly as possible, over the main definitions that appear in what follows (for terms whose use is not as frequent, we will provide definitions as we go along).

\subsection{Diffeology and diffeological vector spaces}

Let $X$ be a set. A \textbf{diffeology} on $X$ (see \cite{So1}, \cite{So2}) is a set $\calD=\{p:U\to X\}$ of maps into $X$, each defined on a domain of some $\matR^n$ (with varying $n$), that satisfies the
following conditions: 1) it includes all constant maps, \emph{i.e.}, maps of form $U\to\{x_0\}$, for all open (possibly disconnected) sets $U\subseteq\matR^n$ and for all points $x_0\in X$; 2) for any
$\calD\ni p:U\to X$ and for any usual smooth map $g:V\to U$ (again defined on some domain $V\subseteq\matR^m$) we have $p\circ g\in\calD$; and 3) if a set map $p:U\to X$ is such that its domain of 
definition $U\subseteq\matR^n$ has an open cover $U=\cup_{i\in I}U_i$ for which $p|_{U_i}\in\calD$ then $p\in\calD$.

The maps composing $\calD$ are called \textbf{plots}. A standard example of diffeology/ diffeological space is a usual smooth manifold $M$, with diffeology composed of all usual smooth maps into $M$. On
the other hand, any set (usually at a least a topological space) admits plenty of non-standard diffeologies, obtained via the concept of a \emph{generated diffeology}.

\paragraph{Generated diffeologies} Given a fixed set $X$, various diffeologies on it can be compared with respect to the inclusion;\footnote{In fact, each diffeology is just a set of maps, which may wholly
contain another diffeology, or be contained in one; in fact, there is a complete lattice on them on any $X$.} for two diffeologies $\calD$ and $\calD'$ such that $\calD\subset\calD'$ one says that $\calD$ is
\textbf{finer} than $\calD'$, whereas $\calD'$ is said to be \textbf{coarser}. Frequently, for a given property $P$ which a diffeology might possess, there is the finest and/or the coarsest diffeology with the
property $P$ (see \cite{iglesiasBook}, Sect. 1.25); this fact is often used in describing concrete diffeologies, or defining a class of them. A specific example of the former is the \textbf{generated diffeology}: for a
set $X$ and a set $A=\{p:U\to X\}$ of maps into $X$, the diffeology \textbf{generated by $A$} is the smallest diffeology on $X$ that contains $A$. Notice that $A$ can be \emph{any} set; it might include
non-differentiable maps, discontinuous ones, and so on.

\paragraph{Smooth maps} A map $f:X\to Y$ between two diffeological spaces $X$ and $Y$ is considered \textbf{smooth} if for every plot $p$ of $X$ the composition $f\circ p$ is a plot of $Y$. Note that it
might easily happen that all $f\circ p$ are plots of $Y$, but that \emph{vice versa} is not true: $Y$ may have plots that do not have form $f\circ p$, whatever the plot $p$ of $X$. On the other hand, if for every
plot $q:U\to Y$ of $Y$ and for every point $u\in U$ there is a plot $p_u$ of $X$ such that in a neighborhood of $u$ we have $q=f\circ p_u$ then we say that the diffeology of $Y$ is the \textbf{pushforward} of the
diffeology of $X$ via $f$, and conversely, the diffeology of $X$ is the \textbf{pullback} of the diffeology of $Y$ by $f$.

\paragraph{Subset, quotient, product, and disjoint union diffeologies} All typical topological constructions admit diffeological counterparts. If $X'$ is any subset of a diffeological space $X$, it carries the
\textbf{subset diffeology} that consists of all plots of $X$ whose range is contained in $X'$; and if $X/\sim$ is any\footnote{We stress that this can indeed be any quotient, with no restrictions on the equivalence
relation $\sim$.} quotient of $X$, with $\pi:X\to X/\sim$ being the natural projection, then the standard choice of diffeology on $X/\sim$ is the \textbf{quotient diffeology}, defined as the pushforward of the
diffeology of $X$ by $\pi$. As we said above, this means that locally each plot of $X/\sim$ has form $\pi\circ p$, where $p$ is a plot of $X$.

Let us now have several (a finite number of, although the definition can be stated more broadly) diffeological spaces $X_1,\ldots,X_n$. Their usual direct product also has its standard diffeology, the
\textbf{product diffeology}, defined as the coarsest diffeology such that the projection on each term is smooth. Locally any plot of this diffeology is just an $n$-tuple $(p_1,\ldots,p_n)$, where each $p_i$ is a plot
of the corresponding $X_i$. Finally, the disjoint union $\sqcup_{i=1}^nX_i$ of these spaces has the \textbf{disjoint union diffeology}, this being the finest diffeology such that the inclusion of each term into the
disjoint union is smooth. Locally, any plot of this diffeology is a plot of precisely one of terms $X_i$.

\paragraph{Functional diffeology} The space $C^{\infty}(X,Y)$ of all smooth (in the diffeological sense) maps between two diffeological spaces $X$ and $Y$ also has its standard diffeology, called the
\textbf{functional diffeology}. It consists of all possible maps $q:U\to C^{\infty}(X,Y)$ such that for every plot $p:U'\to X$ of $X$ the natural evaluation map $U\times U'\ni(u,u')\mapsto q(u)(p(u'))\in Y$ is smooth
(with respect to the diffeology of $Y$; the product $U\times U'$ is still a domain, therefore asking for the evaluation map to be smooth is equivalent to asking it to be a plot of $Y$).

\paragraph{Diffeological vector spaces} This is one specific instance of a diffeological space endowed also with an algebraic structure whose operations are smooth for the diffeology involved.\footnote{There
are analogous notions of a diffeological group, diffeological algebra, and so on.} A (real) \textbf{diffeological vector space} is a vector space $V$ endowed with a diffeology such that the addition map
$V\times V\to V$ and the scalar multiplication map $\matR\times V\to V$ are smooth (for the product diffeology on $V\times V$ and $\matR\times V$ respectively). For a fixed $V$ there can be many such
diffeologies; any one of them is called a \textbf{vector space diffeology} (on $V$). If $V$ is finite-dimensional, and so as just a vector space is isomorphic to some $\matR^n$, then the finest of all vector space
diffeologies is the one consisting of all usual smooth maps into it.\footnote{In contrast with non-vector space diffeologies, the finest of which is always the discrete one, consisting of constant maps only.} This
diffeology is called the \textbf{standard diffeology}; endowed with it, $V$ is called a \textbf{standard space}.

All usual operations on vector spaces (taking subspaces, quotients, direct sums, tensor products, and duals) admit their natural diffeological counterparts (\cite{vincent}, \cite{wu}; see also \cite{multilinear}), via
the more general diffeological constructions described above. Thus, any vector subspace of a diffeological vector space $V$ is automatically endowed with the subset diffeology; every quotient space carries
the quotient diffeology; the direct sum carries the product diffeology (relative to the diffeologies of its terms); and the tensor product has the quotient diffeology of the finest vector space diffeology on the free
product of the factors that contains the product diffeology on their direct product. Finally, the diffeological dual $V^*$ of $V$ is defined as the space of all diffeologically smooth linear maps $L^{\infty}(V,\matR)$
(where $\matR$ is standard) endowed with the functional diffeology. Notice that, unless a finite-dimensional $V$ is a standard space, we have $\dim(V^*)<\dim(V)$ (and in general, the space of smooth linear
maps between two diffeological vector spaces is strictly smaller than the space of all linear maps).

\paragraph{Pseudo-metrics and characteristic subspaces} A finite-dimensional diffeological vector space $V$ does not admit a smooth scalar product, unless it is standard (see \cite{iglesiasBook}). The best
possible substitute for it is any smooth symmetric semi-definite positive bilinear form of rank $\dim(V^*)$; (at least one) such a form exists on any finite-dimensional $V$ and is called a \textbf{pseudo-metric}.

A pseudo-metric $g$ on a finite-dimensional diffeological vector space $V$ allows to identify in $V$ the unique vector subspace $V_0$ which is maximal, with respect to inclusion, for the following two
properties: the subset diffeology of $V_0$ is that of a standard space, and $V_0$ splits off smoothly in $V$, which means there is a usual vector space direct sum decomposition $V=V_0\oplus V_1$ such that 
the corresponding direct sum diffeology on $V$ relative to the subset diffeologies on $V_0$ and $V_1$ coincides with the initial diffeology of $V$.\footnote{ \emph{A priori}, this
direct-sum-diffeology-from-subset-diffeologies is finer, so not all usual direct sum decompositions of diffeological vector spaces are smooth; see an example in \cite{pseudometric}.} This subspace can be
described as the subspace generated by all the eigenvectors of $g$ that are relative to the positive eigenvalues; however, it does not actually depend on the specific choice of a pseudo-metric and is instead
an invariant of $V$ itself. It is called the \textbf{characteristic subspace} of $V$.

\subsection{Diffeological vector pseudo-bundles and pseudo-metrics on them}

A \emph{diffeological vector pseudo-bundle} is a diffeological counterpart of a usual smooth vector bundle.\footnote{The choice of the term \emph{diffeological vector pseudo-bundle} is ours; the same
object is called just a \emph{diffeological fibre bundle} in \cite{iglFibre}, a \emph{regular vector bundle} in \cite{vincent}, and a \emph{diffeological vector space over $X$} in \cite{CWtangent}. The choice of the
term pseudo-bundle is meant to distinguish these objects from the numerous other versions of bundles that have appeared so far.} Apart from the diffeological smoothness replacing the usual concept of
smooth maps, they lack an atlas of local trivializations, although in many contexts, and in most of what follows, we do add this assumption.

\paragraph{Diffeological vector pseudo-bundles} Let $V$ and $X$ be diffeological spaces, and let $\pi:V\to X$ be a smooth surjective map. The map $\pi$, or the total space $V$, is called a \textbf{diffeological 
vector pseudo-bundle} if for each $x\in X$ the pre-image $\pi^{-1}(x)$ carries a vector space structure such that the following three maps are smooth: the addition map $V\times_X V\to V$ (where $V\times_X V$ 
is endowed with the subset diffeology as a subset of $V\times V$), the scalar multiplication map $\matR\times V\to V$, and the zero section $X\to V$. An example of a diffeological vector pseudo-bundle which 
is not locally trivial, can be found in \cite{CWtangent} (see Example 4.3).

\paragraph{The fibrewise operations and fibrewise diffeologies} Since each fibre of a diffeological vector pseudo-bundle is a diffeological vector space, all the usual operations on vector bundles (direct sums,
tensor products, dual bundles) can be performed on/with pseudo-bundles (see \cite{vincent}, and also \cite{pseudobundles} for some details), although not in an entirely similar way (the lack of local
trivializations prevents that). Instead, these operations are performed by first carrying out the operation in question to each fibre, defining the total space of the new pseudo-bundle as the union of the resulting
diffeological vector spaces (with the obvious fibering over the base), and finally defining the diffeology of this total space as the finest that induces on each fibre its existing diffeology.\footnote{This obviously
poses the question of the existence of such diffeology; this was considered in \cite{vincent}; see also \cite{CWtangent}, Proposition 4.6 for a relevant methodology.}

As an example, and also because this instance will be particularly important for us, let us consider dual pseudo-bundles. Let $\pi:V\to X$ be a diffeological vector pseudo-bundle with finite-dimensional
fibres. The dual pseudo-bundle of $V$ is
$$V^*=\cup_{x\in X}(\pi^{-1}(x))^*,$$ where $(\pi^{-1}(x))^*=L^{\infty}(\pi^{-1}(x),\matR)$ is the diffeological dual of the diffeological vector space $\pi^{-1}(x)$, the pseudo-bundle projection $\pi^*$ is given by
$\pi^*((\pi^{-1}(x))^*)=\{x\}$ for all $x\in X$, and the diffeology on $V^*$ is characterized as follows: a map $q:\matR^l\supseteq U'\to V^*$ is a plot of $V^*$ if and only if for every plot
$p:\matR^m\supseteq U\to V$ of $V$ the evaluation map $(u',u)\mapsto q(u')(p(u))\in\matR$ is smooth for the subset diffeology on its domain of definition $\{(u',u)\,|\,\pi^*(q(u'))=\pi(p(u))\}\subseteq\matR^{l+m}$
and the standard diffeology on $\matR$. The collection of all possible maps $q$ satisfying this property does form a diffeology, equipped with which, $V^*$ becomes a difffeological vector pseudo-bundle, and
furthermore, the corresponding subset diffeology on each fibre $(\pi^*)^{-1}(x)$ coincides with the usual (functional) diffeology on $(\pi^{-1}(x))^*$.

Finally, another useful observation (and maybe a peculiarity of diffeology) is that any collection of vector subspaces, one per fibre, in a diffeological vector pseudo-bundle is again a diffeological vector
pseudo-bundle (called a \textbf{diffeological sub-bundle}), for the subset diffeology. Likewise, any collection of quotients, one of each fibre, is a diffeological vector pseudo-bundle for the quotient diffeology.
We call these facts a peculiarity since they go well beyond what happens for the usual smooth vector bundles.

\paragraph{Pseudo-metrics on diffeological vector pseudo-bundles} Let $\pi:V\to X$ be a diffeological vector pseudo-bundle with finite-dimensional fibres. A \textbf{pseudo-metric} on it is a smooth section of
the pseudo-bundle $V^*\otimes V^*$ such that for each $x\in X$ the bilinear form $g(x)$ is a pseudo-metric, in the sense of diffeological vector spaces, on the fibre $\pi^{-1}(x)$. Not all pseudo-bundles admit
a pseudo-metric (see \cite{pseudobundles}; it is not quite clear yet under which conditions a pseudo-bundle admits a pseudo-metric), although if a pseudo-bundle is locally trivial with a finite atlas of local
trivializations, the reasoning similar to that in the case of usual smooth vector bundles would allow to conclude its existence on any pseudo-bundle with the above two properties.

\subsection{Diffeological gluing}

On the level of the underlying topological\footnote{This means the  topology underlying the diffeological structure, the so-called D-topology, see \cite{iglFibre}. In most significant examples, however, the
diffeology is put on a space already carrying a topological structure, and in way such that the D-topology coincides with it.} spaces, diffeological gluing (introduced in \cite{pseudobundles}) is just the usual
topological gluing. The result is endowed with a canonical diffeology, called the gluing diffeology. It is the finest diffeology for several properties, and is usually finer than other natural diffeologies on the
same space.

\subsubsection{Gluing of spaces, maps, and pseudo-bundles}

The basic ingredient in the definition of the diffeological gluing procedure is the operation of gluing of two diffeological spaces, where we must essentially specify which diffeology is assigned to the space 
obtained by the usual topological gluing. This basic construction is then extended to gluing of smooth maps between diffeological spaces, a particularly important instance of which is the gluing of diffeological 
vector pseudo-bundles.

\paragraph{Diffeological spaces} Let $X_1$ and $X_2$ be two diffeological spaces, and let $f:X_1\supseteq Y\to X_2$ be a map defined on a subset $Y$ of $X_1$ and smooth for the subset diffeology on $Y$.
Let
$$X_1\cup_f X_2:=\left(X_1\sqcup X_2\right)/_{\sim},$$ where the equivalence relation $\sim$ is given by, $X_1\sqcup X_2\ni x_1\sim x_2\in X_1\sqcup X_2$ if and only if either $x_1=x_2$ or $x_1\in Y$ and 
$x_2=f(x_1)$. Denote by $\pi:X_1\sqcup X_2\to X_1\cup_f X_2$ the quotient projection, and define the \textbf{gluing diffeology} on $X_1\cup_f X_2$ to be the pushforward of the disjoint union diffeology on 
$X_1\sqcup X_2$ by $\pi$.\footnote{Notice that in the case when the gluing map $f$ is invertible as map from its domain to its range, and only in this case, $X_1\cup_f X_2$ is a span of the spaces $X_1$ 
and $X_2$. However, \emph{a priori} the gluing construction is more general.}

This construction is particularly adapted to endowing piecewise-linear objects with a diffeology (among others). An easiest example is a wedge of two lines, which can be identified with the union of the two 
coordinate axes in $\matR^2$. It is interesting to notice that the gluing diffeology on this union is finer than the subset diffeology relative to its inclusion into $\matR^2$, as demonstrated by an example due 
to Watts, see Example 2.67 in \cite{watts}. (From this, it is easy to extrapolate the existence of similar examples in other dimensions). Indeed, relative to the gluing diffeology the natural inclusions
\footnote{They are defined by composition of the obvious inclusions into $X_1\sqcup X_2$, with the quotient projection $\pi$.}
$$i_1:X_1\setminus Y\hookrightarrow X_1\cup_f X_2\mbox{ and }i_2:X_2\hookrightarrow X_1\cup_f X_2$$ are smooth, and their ranges form a disjoint cover of $X_1\cup_f X_2$. Notice in particular that
$X_1$ does not in general inject into $X_1\cup_f X_2$, while $X_2$ always does; in other words, \emph{a priori} the operation of gluing is not symmetric.\footnote{Which is a natural consequence of the 
construction and also its merit, as it allows to treat, for instance, conical singularities as the results of gluing to a one-point space. For this reason, although in most cases we deal with gluings along 
invertible maps, hence symmetric ones, we treat them as if they were not, to keep the discussion as general as possible.}

This asymmetry is demonstrated, for instance, by the following (useful in practice) description of plots of the gluing diffeology. Since the latter is a pushforward diffeology, any plot of it lifts to a plot
of the covering space $X_1\sqcup X_2$. By the properties of the disjoint union diffeology, this means that if $p:U\to X_1\cup_f X_2$ is a plot and $U$ is connected then it either lifts to a plot $p_1$
of $X_1$ or a plot $p_2$ of $X_2$. By construction of $X_1\cup_f X_2$, we obtain that in the former case
$$p(u)=\left\{\begin{array}{ll} i_1(p_1(u)) & \mbox{if }p_1(u)\in X_1\setminus Y,\\
i_2(f(p_1(u))) & \mbox{if }p_1(u)\in Y, \end{array}\right.$$ while in the latter case we simply have $p=i_2\circ p_2$.

\paragraph{Gluing of maps} The operation of gluing of diffeological spaces, when performed on the domains and possibly the ranges of some given smooth maps, defines a gluing of these maps, provided the
maps themselves satisfy a natural compatibility condition.\footnote{According to a personal preference, one could use the language of category theory and describe the gluing operation as the pushforward 
of the pair of maps $\mbox{Id}:X_1\to X_1$ and $f:X_1\supset Y\to X_2$ to the category of smooth maps between diffeologies, that enjoys the bifunctoriality property.} More precisely, suppose first that 
$X_1,X_2,Z$ are diffeological spaces and $\varphi_i:X_i\to Z$ for $i=1,2$ are smooth maps. Let
$f:X_1\supseteq Y\to X_2$ be a smooth map, and suppose that $\varphi_2(f(y))=\varphi_1(y)$ for all $y\in Y$ (the maps $\varphi_1$ and $\varphi_2$ are then said to be \textbf{$f$-compatible}). Then the map
$\varphi_1\cup_f\varphi_2:X_1\cup_f X_2\to Z$ given by
$$(\varphi_1\cup_f\varphi_2)(x)=\left\{\begin{array}{ll}
\varphi_1(i_1^{-1}(x)) & \mbox{if }x\in\mbox{Range}(i_1),\\
\varphi_2(i_2^{-1}(x)) & \mbox{if }x\in\mbox{Range}(i_2)
\end{array}\right.$$ is well-defined and smooth. In fact, assigning the map $\varphi_1\cup_f\varphi_2\in C^{\infty}(X_1\cup_f X_2,Z)$ to each pair $(\varphi_1,\varphi_2)$, with $\varphi_i\in C^{\infty}(X_i,Z)$ for
$i=1,2$, of $f$-compatible maps yields a map
$$C^{\infty}(X_1,Z)\times_{comp}C^{\infty}(X_2,Z)\to C^{\infty}(X_1\cup_f X_2,Z)$$ that is smooth for the functional diffeology on $C^{\infty}(X_1\cup_f X_2,Z)$ and the subset
diffeology\footnote{Relative to the product diffeology on $C^{\infty}(X_1,Z)\times C^{\infty}(X_2,Z)$ which in turn comes from the functional diffeologies on $C^{\infty}(X_1,Z)$ and $C^{\infty}(X_2,Z)$.} on the set
of $f$-compatible pairs $C^{\infty}(X_1,Z)\times_{comp}C^{\infty}(X_2,Z)$ (see \cite{pseudometric-pseudobundle} for details).

Finally, all of this extends to the case of two maps with distinct ranges, that is, $\varphi_1:X_1\to Z_1$ and $\varphi_2:X_2\to Z_2$, with appropriate gluings of $X_1$ to $X_2$ and $Z_1$ to $Z_2$.
Specifically, let again $f:X_1\supseteq Y\to X_2$ be smooth, but consider also a smooth $g:Z_1\supseteq\varphi_1(Y)\to Z_2$; assume that $\varphi_2(f(y))=g(\varphi_1(y))$ for all $y\in Y$ (the maps
$\varphi_1$ and $\varphi_2$ are said to be \textbf{$(f,g)$-compatible}). Notice that the counterparts of $i_1$ and $i_2$ for the space $Z_1\cup_g Z_2$ are the natural inclusion
maps $j_1:Z_1\setminus\varphi_1(Y)\hookrightarrow Z_1\cup_g Z_2$ and $j_2:Z_2\to Z_1\cup_g Z_2$. We define the map $\varphi_1\cup_{(f,g)}\varphi_2$ by setting
$$(\varphi_1\cup_{(f,g)}\varphi_2)(x)=\left\{\begin{array}{ll}
j_1(\varphi_1(i_1^{-1}(x))) & \mbox{if }x\in\mbox{Range}(i_1),\\
j_2(\varphi_2(i_2^{-1}(x))) & \mbox{if }x\in\mbox{Range}(i_2).
\end{array}\right.$$ All the analogous statements, in particular, the smoothness of the thus-defined map $C^{\infty}(X_1,Z_1)\times_{comp}C^{\infty}(X_2,Z_2)\to C^{\infty}(X_1\cup_f X_2,Z_1\cup_g Z_2)$,
continue to hold (see \cite{pseudometric-pseudobundle}).

\paragraph{Pseudo-bundles} Let us now turn to gluing of two pseudo-bundles; this is of course a specific instance of gluing of $(f,g)$-compatible maps. We specifically indicate it in order to to fix notation and
standard terminology, and also to give conditions under which the result of gluing is again a diffeological vector pseudo-bundle.

Let $\pi_1:V_1\to X_1$ and $\pi_2:V_2\to X_2$ be two such pseudo-bundles, let $f:X_1\supseteq Y\to X_2$ be a smooth map, and let $\tilde{f}:V_1\supseteq\pi_1^{-1}(Y)\to V_2$ be any smooth lift of $f$
that is linear on each fibre (where it is defined). The gluing of these pseudo-bundles consists in the already-defined operations of gluing $V_1$ to $V_2$ along $\tilde{f}$, gluing $X_1$ to $X_2$ along
$f$, and $\pi_1$ to $\pi_2$ along $(\tilde{f},f)$. It is easy to check (see \cite{pseudobundles}) that the resulting map $\pi_1\cup_{(\tilde{f},f)}\pi_2:V_1\cup_{\tilde{f}}V_2\to X_1\cup_f X_2$ is a diffeological
vector pseudo-bundle for the gluing diffeologies on $V_1\cup_{\tilde{f}}V_2$ and $X_1\cup_f X_2$; in particular, the vector space structure on its fibres is inherited from either $V_1$ or $V_2$ (more precisely,
it is inherited from $V_1$ on fibres over the points in $i_1(X_1\setminus Y)$, and from $V_2$ on fibres over the points in $i_2(X_2)$).

\paragraph{Standard notation for gluing of pseudo-bundles} We now fix some standard notation that applies to pseudo-bundles specifically. We have already described the standard inclusions
$$i_1:X_1\setminus Y\hookrightarrow X_1\cup_f X_2,\,\,\,i_2:X_2\hookrightarrow X_1\cup_f X_2,$$
$$j_1:V_1\setminus\pi_1^{-1}(Y)\hookrightarrow V_1\cup_{\tilde{f}}V_2,\,\,\,j_2:V_2\hookrightarrow V_1\cup_{\tilde{f}}V_2.$$ When dealing with more than one gluing at a time, we will needed a more 
complicated notation, which is as follows. Let $\chi_1\cup_{(\tilde{h},h)}\chi_2:W_1\cup_{\tilde{h}}W_2\to Z_1\cup_h Z_2$ be any pseudo-bundle obtained by gluing; denote by $Y'\subset Z_1$ the domain of 
definition of $h$. Then the counterparts of $i_1$ and $i_2$ will be denoted by\footnote{The rule-of-thumb is that $i_1^{smth}$ and $i_2^{smth}$ are used respectively for the left-hand and the right-hand factor 
in the base space of the pseudo-bundle under consideration, while $j_1^{smth}$ and $j_2^{smth}$ refer to the left-hand and the right-hand factor of the total space.}
$$i_1^{Z_1}:Z_1\setminus Y'\hookrightarrow Z_1\cup_h Z_2\mbox{ and }i_2^{Z_2}:Z_2\hookrightarrow Z_1\cup_h Z_2\mbox{ for }Z_1\cup_h Z_2$$
$$j_1^{W_1}:W_1\setminus\chi_1^{-1}(Y')\hookrightarrow W_1\cup_{\tilde{h}}W_2\mbox{ and }j_2^{W_2}:W_2\hookrightarrow W_1\cup_{\tilde{h}}W_2\mbox{ for }W_1\cup_{\tilde{h}}W_2.$$
Obviously, $i_1^{W_1}$ and $j_1^{W_1}$ would mean the same thing, and the same goes for $i_2^{W_2}$ and $j_2^{W_2}$; however, we use the different letters so that there always be a clear distinction
between the base space and the total space. Finally, since the base space will be the same for all our pseudo-bundles, and so we will usually use the abbreviated notation $i_1,i_2$ for it.

\paragraph{The switch map} As we mentioned many times already, the operation of gluing for diffeological spaces is asymmetric. However, if we assume that gluing map $f$ is a diffeomorphism with its image
then obviously, we can use its inverse to perform the gluing in the reverse order, with the two results, $X_1\cup_f X_2$ and $X_2\cup_{f^{-1}}X_1$, being canonically diffeomorphic via the
so-called \textbf{switch map}
$$\varphi_{X_1\leftrightarrow X_2}:X_1\cup_f X_2\to X_2\cup_{f^{-1}}X_1.$$ Using the notation just introduced, this map can be described by
$$\left\{\begin{array}{ll}
\varphi_{X_1\leftrightarrow X_2}(i_1^{X_1}(x))=i_2^{X_1}(x) & \mbox{for }x\in X_1\setminus Y,\\
\varphi_{X_1\leftrightarrow X_2}(i_2^{X_2}(f(x)))=i_2^{X_1}(x) & \mbox{for }x\in Y,\\
\varphi_{X_1\leftrightarrow X_2}(i_2^{X_2}(x))=i_1^{X_2}(x) & \mbox{for }x\in X_2\setminus f(Y).\end{array}\right.$$ This is well-defined, not only because the maps $i_1^{X_1}$ and $i_2^{X_2}$
are injective with disjoint ranges covering $X_1\cup_f X_2$, but also because $f$ is a diffeomorphism with its image.

\subsubsection{Gluing and operations}

Diffeological gluing of pseudo-bundles is relatively well-behaved with respect to the usual operations on vector bundles. More precisely, it commutes with the direct sum and the tensor product,
while he situation is somewhat more complicated for the dual pseudo-bundles, see \cite{pseudobundles} (the facts needed are recalled below).

\paragraph{Direct sum} Gluing of diffeological vector pseudo-bundles commutes with the direct sum in the following sense. Given a gluing along $(\tilde{f},f)$ of a pseudo-bundle $\pi_1:V_1\to X_1$ to a
pseudo-bundle $\pi_2:V_2\to X_2$, as well as a gluing along $(\tilde{f}',f)$ of a pseudo-bundle $\pi_1':V_1'\to X_1$ to a pseudo-bundle $\pi_2':V_2'\to X_2$, there are two natural
pseudo-bundles that can be formed from these by applying the operations of gluing and direct sum. These are the pseudo-bundles
$$(\pi_1\cup_{(\tilde{f},f)}\pi_2)\oplus(\pi_1'\cup_{(\tilde{f}',f)}\pi_2'):(V_1\cup_{\tilde{f}}V_2)\oplus(V_1'\cup_{\tilde{f}'}V_2')\to X_1\cup_f X_2\mbox{ and }$$
$$(\pi_1\oplus\pi_1')\cup_{(\tilde{f}\oplus\tilde{f}',f)}(\pi_2\oplus\pi_2'):(V_1\oplus V_1')\cup_{\tilde{f}\oplus\tilde{f}'}(V_2\oplus V_2')\to X_1\cup_f X_2;$$ they are diffeomorphic as pseudo-bundles, that is, there
exists a fibrewise linear diffeomorphism
$$\Phi_{\cup,\oplus}:(V_1\cup_{\tilde{f}}V_2)\oplus(V_1'\cup_{\tilde{f}'}V_2')\to(V_1\oplus V_1')\cup_{\tilde{f}\oplus\tilde{f}'}(V_2\oplus V_2')$$ (see below) that covers the identity map on the base
$X_1\cup_f X_2$.

\paragraph{Tensor product} What has just been said about the direct sum, applies equally well to the tensor product. Specifically, the two possible pseudo-bundles are
$$(\pi_1\cup_{(\tilde{f},f)}\pi_2)\otimes(\pi_1'\cup_{(\tilde{f}',f)}\pi_2'):(V_1\cup_{\tilde{f}}V_2)\otimes(V_1'\cup_{\tilde{f}'}V_2')\to X_1\cup_f X_2\mbox{ and }$$
$$(\pi_1\otimes\pi_1')\cup_{(\tilde{f}\otimes\tilde{f}',f)}(\pi_2\otimes\pi_2'):(V_1\otimes V_1')\cup_{\tilde{f}\otimes\tilde{f}'}(V_2\otimes V_2')\to X_1\cup_f X_2,$$ and again, they are
pseudo-bundle-diffeomorphic via
$$\Phi_{\cup,\otimes}:(V_1\cup_{\tilde{f}}V_2)\otimes(V_1'\cup_{\tilde{f}'}V_2')\to(V_1\otimes V_1')\cup_{\tilde{f}\otimes\tilde{f}'}(V_2\otimes V_2')$$ covering the identity on $X_1\cup_f X_2$.

\paragraph{The dual pseudo-bundle} The case of dual pseudo-bundles is substantially different. For one thing, to even make sense of the commutativity question, we must assume that $f$ is invertible (which
in general it does not have to be). Moreover, even with this assumption, in general the operation of gluing does \emph{not} commute with that of taking duals. The reason of this is easy to
explain. Let $\pi_1:V_1\to X_1$ and $\pi_2:V_2\to X_2$ be two diffeological vector pseudo-bundles, and let $(\tilde{f},f)$ be a gluing between them; consider the pseudo-bundle
$\pi_1\cup_{(\tilde{f},f)}\pi_2:V_1\cup_{\tilde{f}}V_2\to X_1\cup_f X_2$ and the corresponding dual pseudo-bundle
$$(\pi_1\cup_{(\tilde{f},f)}\pi_2)^*:(V_1\cup_{\tilde{f}}V_2)^*\to X_1\cup_f X_2;$$ compare it with the result of the induced gluing (performed along the pair $(\tilde{f}^*,f)$) of $\pi_2^*:V_2^*\to X_2$ to
$\pi_1^*:V_1^*\to X_1$, that is, the pseudo-bundle
$$\pi_2^*\cup_{(\tilde{f}^*,f)}\pi_1^*:V_2^*\cup_{\tilde{f}^*}V_1^*\to X_2\cup_{f^{-1}}X_1.$$ It then follows from the construction itself that for any $y\in Y$ (the domain of gluing) we have
$$((\pi_1\cup_{(\tilde{f},f)}\pi_2)^*)^{-1}(i_2^{X_2}(f(y)))\cong(\pi_2^{-1}(f(y)))^*\mbox{ and }(\pi_2^*\cup_{(\tilde{f}^*,f)}\pi_1^*)^{-1}(i_2^{X_1}(y))\cong(\pi_1^{-1}(y))^*;$$
since $i_2^{X_2}(f(y))$ and $i_2^{X_1}(y)$ are related by the switch map, for the two pseudo-bundles to be diffeomorphic in a natural way\footnote{For us this means, via a diffeomorphism covering the
switch map.} the two vector spaces $(\pi_2^{-1}(f(y)))^*$ and $(\pi_1^{-1}(y))^*$ must be diffeomorphic, and \emph{a priori} they are not.\footnote{They may have different dimensions.}

Thus, we obtain the necessary condition (for the pseudo-bundles $(V_1\cup_{\tilde{f}}V_2)^*$ and $V_2^*\cup_{\tilde{f}^*}V_1^*$ to be diffeomorphic), which is that
$(\pi_2^{-1}(f(y)))^*\cong(\pi_1^{-1}(y))^*$ for all $y\in Y$. We do note right away that this condition may not be sufficient, in the sense that two pseudo-bundles over the same base may have all
respective fibres diffeomorphic without being diffeomorphic themselves (this can be illustrated by the standard example of open annulus and open M\"obius strip, both of which, equipped with the
standard diffeology\footnote{That is, one determined by their usual smooth structure.} can be seen as pseudo-bundles over the circle). We will leave it at that for now, returning to the issue of the
gluing-dual commutativity later in the paper.

\subsubsection{The commutativity diffeomorphisms}

We now say more about the commutativity diffeomorphisms mentioned in the previous section.\footnote{If one wishes, these could be described as universal factorization properties in the category of 
diffeological pseudo-bundles}

\paragraph{The diffeomorphism $\Phi_{\cup,\oplus}$} We have already mentioned the existence of this diffeomorphism, which is a pseudo-bundle map
$$\Phi_{\cup,\oplus}:(V_1\cup_{\tilde{f}}V_2)\oplus(V_1'\cup_{\tilde{f}'}V_2')\to(V_1\oplus V_1')\cup_{\tilde{f}\oplus\tilde{f}'}(V_2\oplus V_2')$$
that covers the identity map on $X_1\cup_f X_2$. We now add that this map can be described (in fact, fully defined) by the following identities:
$$\Phi_{\cup,\oplus}\circ(j_1^{V_1}\oplus j_1^{V_1'})=j_1^{V_1\oplus V_1'}\,\,\mbox{ and }\,\, \Phi_{\cup,\oplus}\circ(j_2^{V_2}\oplus j_2^{V_2'})=j_2^{V_2\oplus V_2'}.$$

\paragraph{The diffeomorphism $\Phi_{\cup,\otimes}$} Once again, the case of the tensor product is very similar to that of the direct sum. The already-mentioned diffeomorphism
$$\Phi_{\cup,\otimes}:(V_1\cup_{\tilde{f}}V_2)\otimes(V_1'\cup_{\tilde{f}'}V_2')\to(V_1\otimes V_1')\cup_{\tilde{f}\otimes\tilde{f}'}(V_2\otimes V_2')$$ is uniquely determined by the identities
$$\Phi_{\cup,\otimes}\circ(j_1^{V_1}\otimes j_1^{V_1'})=j_1^{V_1\otimes V_1'}\,\,\mbox{ and }\,\, \Phi_{\cup,\otimes}\circ(j_2^{V_2}\otimes j_2^{V_2'})=j_2^{V_2\otimes V_2'};$$ notice that these, themselves,
suffice to ensure that it covers the identity map on $X_1\cup_f X_2$.

\paragraph{The gluing-dual commutativity conditions, and diffeomorphism $\Phi_{\cup,*}$} We have already described the situation regarding the gluing-dual commutativity, first of all the fact, that it is far from
being always present. At this moment we concentrate on what it actually means for the gluing to commute with taking duals (once again leaving aside the question when this does happen). Specifically, we
say that the \textbf{gluing-dual commutativity condition} holds, if there exists a diffeomorphism
$$\Phi_{\cup,*}:(V_1\cup_{\tilde{f}}V_2)^*\to V_2^*\cup_{\tilde{f}^*}V_1^*$$ that covers the switch map, that is,
$$(\pi_2^*\cup_{(\tilde{f}^*,f^{-1})}\pi_1^*)\circ\Phi_{\cup,*}=\varphi_{X_1\leftrightarrow X_2}\circ(\pi_1\cup_{(\tilde{f},f)}\pi_2)^*,$$ and such that the following are true:
$$\left\{\begin{array}{ll}
\Phi_{\cup,*}\circ((j_1^{V_1})^*)^{-1}=j_2^{V_1^*} & \mbox{on }(\pi_2^*\cup_{(\tilde{f}^*,f^{-1})}\pi_1^*)^{-1}(i_2^{X_1}(X_1\setminus Y)),\\
\Phi_{\cup,*}\circ((j_2^{V_2})^*)^{-1}=j_2^{V_1^*}\circ\tilde{f}^* & \mbox{on }(\pi_2^*\cup_{(\tilde{f}^*,f^{-1})}\pi_1^*)^{-1}(i_2^{X_1}(Y)),\\
\Phi_{\cup,*}\circ((j_2^{V_2})^*)^{-1}=j_1^{V_2^*} & \mbox{on }(\pi_2^*\cup_{(\tilde{f}^*,f^{-1})}\pi_1^*)^{-1}(i_1^{X_2}(X_2\setminus f(Y))).\end{array}\right.$$

\subsubsection{Gluing and pseudo-metrics}

The behavior of pseudo-metrics under gluing depends significantly on whether the gluing-dual commutativity condition is satisfied. More precisely, if we glue together two pseudo-bundles carrying a
pseudo-metric each, then under a certain natural compatibility condition (see below) for these pseudo-metrics, the new pseudo-bundle carries a pseudo-metric as well; but how the latter is
constructed depends on the existence of the commutativity diffeomorphism $\Phi_{\cup,*}$.

\paragraph{The compatibility notion for pseudo-metrics} Let $\pi_1:V_1\to X_1$ and $\pi_2:V_2\to X_2$ be two finite-dimensional diffeological vector pseudo-bundles, endowed with pseudo-metrics $g_1$
and $g_2$ respectively, and let $(\tilde{f},f)$ be a pair of smooth maps that defines gluing of $X_1$ to $X_2$; let $Y\subseteq X_1$ be the domain of definition of $f$. The pseudo-metrics $g_1$ and $g_2$
are said to be \textbf{compatible with (the gluing along) the pair $(\tilde{f},f)$} if for all $y\in Y$ and for all $v,w\in\pi_1^{-1}(y)$ the following is true:
$$g_1(y)(v,w)=g_2(f(y))(\tilde{f}(v),\tilde{f}(w)).$$ If $f$ is invertible, this means that $g_2$ and $g_1$ are $(f^{-1},\tilde{f}^*\otimes\tilde{f}^*)$-compatible as smooth maps
$X_2\to V_2^*\otimes V_2^*$ and $X_1\to V_1^*\otimes V_1^*$ respectively.

As can be expected,\footnote{The notion of a pseudo-metric is designed to be a generalization from a scalar product on a vector space, where the compatibility with a given $f$ is equivalent to $f$
being an isometry (of its domain with its range).} the existence of compatible pseudo-metrics on two pseudo-bundles imposes substantial restrictions on their fibres over the domain of gluing. Later in the
paper we will make precise statements to this effect.

\paragraph{The induced pseudo-metric in the presence of $\Phi_{\cup,*}$} If we assume that the gluing-dual commutativity condition is satisfied, this implies also that $f$ is invertible
(with smooth inverse). In this case we can use the map
$$g_2\cup_{(f^{-1},\tilde{f}^*\otimes\tilde{f}^*)}g_1:X_2\cup_{f^{-1}}X_1\to(V_2^*\otimes V_2^*)\cup_{\tilde{f}^*\otimes\tilde{f}^*}(V_1^*\otimes V_1^*)$$ to
construct a pseudo-metric on $V_1\cup_{\tilde{f}}V_2$ by taking the following composition of it with the switch map and the commutativity diffeomorphisms:
$$\tilde{g}=\left(\Phi_{\cup,*}^{-1}\otimes\Phi_{\cup,*}^{-1}\right)\circ\Phi_{\otimes,\cup}\circ(g_2\cup_{(f^{-1},\tilde{f}^*\otimes\tilde{f}^*)}g_1)\circ\varphi_{X_1\leftrightarrow X_2},$$ where
$\varphi_{X_1\leftrightarrow X_2}$ is the switch map, $\Phi_{\cup,*}$ (of which we need the inverse) is the just-seen gluing-dual commutativity diffeomorphism, while
$\Phi_{\otimes,\cup}:(V_2^*\otimes V_2^*)\cup_{\tilde{f}^*\otimes\tilde{f}^*}(V_1^*\otimes V_1^*)\to(V_2^*\cup_{\tilde{f}^*}V_1^*)\otimes(V_2^*\cup_{\tilde{f}^*}V_1^*)$
is the appropriate version of the tensor product-gluing commutativity diffeomorphism.

\paragraph{Constructing a pseudo-metric on $V_1\cup_{\tilde{f}}V_2$ when $\Phi_{\cup,*}$ does not exist} Although we will mostly deal with the cases where the gluing-dual commutativity condition is
present (and so the above definition of the pseudo-metric $\tilde{g}$ on $V_1\cup_{\tilde{f}}V_2$ is sufficient), we briefly mention that even if such condition does not hold, the flexibility
of diffeology allows for a direct construction of a pseudo-metric on $V_1\cup_{\tilde{f}}V_2$. This construction uses the fact that each fibre of $V_1\cup_{\tilde{f}}V_2$ is naturally identified with one
of either $V_1$ or $V_2$; accordingly, $\tilde{g}$ can be defined to coincide with either $g_1$ or $g_2$ on each fibre individually. The surprising fact is that $\tilde{g}$ coming from this construction is
still diffeologically smooth across the fibres; see \cite{pseudometric-pseudobundle} for details.

\subsection{The pseudo-bundle of smooth linear maps}

The last more-or-less standard construction that we need is that of \textbf{pseudo-bundle of smooth linear maps}. Let $\pi_1:V_1\to X$ and $\pi_2:V_2\to X$ be two finite-dimensional diffeological vector
pseudo-bundles with the same space $X$. For every $x\in X$ the space $L^{\infty}(\pi_1^{-1}(x),\pi_2^{-1}(x))$ of smooth linear maps $\pi_1^{-1}(x)\to\pi_2^{-1}(x)$ is a (finite-dimensional)
diffeological vector space for the functional diffeology. The union
$$\mathcal{L}(V_1,V_2)=\cup_{x\in X}L^{\infty}(\pi_1^{-1}(x),\pi_2^{-1}(x))$$ of all these spaces has the obvious projection (denoted $\pi^L$) to $X$, and the pre-image
of each point under this projection has vector space structure. It becomes a diffeological vector pseudo-bundle when endowed with the \textbf{pseudo-bundle functional diffeology}, that is defined as the
finest diffeology containing all maps $p:U\to\mathcal{L}(V_1,V_2)$, with $U\subseteq\matR^m$ an arbitrary domain, that possess the following property: for every plot $q:\matR^{m'}\supseteq U'\to V_1$
of $V_1$ the corresponding evaluation map $(u,u')\mapsto p(u)(q(u'))\in V_2$ defined on $Y'=\{(u,u')\,|\,\pi^L(p(u))=\pi_1(q(u'))\}\subset U\times U'$ is smooth for the subset diffeology of $Y'$. This is the type
of object where the Clifford actions live.

\subsection{The pseudo-bundles of Clifford algebras and Clifford modules}

For diffeological pseudo-bundles, these have already been described in the abstract setting (see \cite{clifford-alg}). We briefly summarize the main points that appear therein, noting that the main
conclusions do not differ from the usual case, or are as expected anyhow.

\subsubsection{The pseudo-bundle $cl(V,g)$}

Let $\pi:V\to X$ be a finite-dimensional diffeological pseudo-bundle endowed with a pseudo-metric $g$. The construction of the corresponding pseudo-bundle $cl(V,g)$ of Clifford algebras is the
immediate one, since all the operations involved have already been described. Specifically, the \textbf{pseudo-bundle of Clifford algebras $\pi^{Cl}:cl(V,g)\to X$} is given by
$$cl(V,g):=\cup_{x\in X}cl(\pi^{-1}(x),g(x))$$ and is endowed with the quotient diffeology coming from \textbf{pseudo-bundle of tensor algebras} $\pi^{T(V)}:T(V)\to X$.
The latter pseudo-bundle has total space given by $T(V):=\cup_{x\in X}T(\pi^{-1}(x))$, where $T(\pi^{-1}(x)):=\bigoplus_{r}(\pi^{-1}(x))^{\otimes r}$ is the usual tensor algebra of the diffeological vector space
$\pi^{-1}(x)$ (in particular, it is endowed with the vector space direct sum diffeology relative to the tensor product diffeology on each factor).

\begin{rem}
We will not make much use of the algebra structure on this pseudo-bundle, and so will actually consider all the direct sums involved to be finite (limited by the maximum of the dimensions of
fibres of $V$ in question --- this includes the assumption that such maximum exists), thus considering, instead of the whole $T(V)$ its finite-dimensional sub-bundle $T_{\leqslant n}(V)$, with fibre at
$x$ the space $T_{\leqslant n}(\pi^{-1}(x))$ consisting of all tensors in $T(\pi^{-1}(x))$ of degree at most $n$. These fibres are not algebras, but of course each of them is a vector subspace of the
corresponding fibre of $T(V)$.
\end{rem}

Recall that the subset diffeology on each fibre of $T(V)$ is that of the tensor algebra\footnote{Or of its appropriate vector subspace, see the remark above.} of the individual fibre $\pi^{-1}(x)$. In
each such fibre we choose the subspace $W_x$ that is the kernel of the universal map $T(\pi^{-1}(x))\to\cl(\pi^{-1}(x),g(x))$. Then, as is generally the case, $W=\cup_{x\in X}W_x\subset T(V)$ endowed with
the subset diffeology relative to this inclusion is a sub-bundle of $T(V)$. The fibre of the corresponding quotient pseudo-bundle at any given point $x\in X$ is $cl(\pi^{-1}(x),g(x))$, and the quotient
diffeology on the fibre is that of the Clifford algebra over the vector space $\pi^{-1}(x)$. This is exactly the diffeology that we endow $\cl(V,g)$ with.

\subsubsection{The pseudo-bundle $\cl(V_1\cup_{\tilde{f}}V_2,\tilde{g})$ as the result of a gluing}

The main result, that we immediately state and that appears in \cite{clifford-alg}, is the following one.

\begin{thm}\label{gluing:clifford:cited:thm}
Let $\pi_1:V_1\to X_1$ and $\pi_2:V_2\to X_2$ be two diffeological vector pseudo-bundles, let $(\tilde{f},f)$ be two maps defining a gluing between these two pseudo-bundles, both of which are
diffeomorphisms, and let $g_1$ and $g_2$ be pseudo-metrics on $V_1$ and $V_2$ respectively, compatible with this gluing. Let $\tilde{g}$ be the pseudo-metric on $V_1\cup_{\tilde{f}}V_2$ induced by
$g_1$ and $g_2$. Then there exists a map $\tilde{F}^{\cl}$ defining a gluing of the pseudo-bundles $\cl(V_1,g_1)$ and $\cl(V_2,g_2)$, and a diffeomorphism $\Phi^{\cl}$ between the pseudo-bundles
$\cl(V_1,g_1)\cup_{\tilde{F}^{\cl}}\cl(V_2,g_2)$ and $\cl(V_1\cup_{\tilde{f}}V_2,\tilde{g})$ covering the identity on $X_1\cup_f X_2$.
\end{thm}

Let us briefly describe the maps $\tilde{F}^{\cl}$ and $\Phi^{\cl}$. The construction of $\tilde{F}^{\cl}$ is the immediately obvious one. It is defined on each fibre over a point $y\in Y$ as the map
$\cl(\pi_1^{-1}(y),g_1|_{\pi_1^{-1}(y)})\to\cl(\pi_2^{-1}(f(y)),g_2|_{\pi_2^{-1}(f(y))})$ induced by $\tilde{f}$ via the universal property of Clifford algebras. In practice, this
means that on each fibre $\tilde{F}^{\cl}$ is linear and multiplicative (with respect to the tensor product), so if $v_1\otimes\ldots\otimes v_k$ is a representative of an equivalence
class in $\cl(\pi_1^{-1}(y),g_1|_{\pi_1^{-1}(y)})$ (viewed as the appropriate quotient of $T(\pi_1^{-1}(y))$) then by this definition
$$\tilde{F}^{\cl}(v_1\otimes\ldots\otimes v_k)=\tilde{F}^{Cl}(v_1)\otimes\ldots\otimes\tilde{F}^{\cl}(v_k)=\tilde{f}(v_1)\otimes\ldots\otimes\tilde{f}(v_k).$$ That this is well-defined as a map
$\cl(\pi_1^{-1}(y),g_1|_{\pi_1^{-1}(y)})\to\cl(\pi_2^{-1}(f(y)),g_2|_{\pi_2^{-1}(f(y))})$ follows from the compatibility of pseudo-metrics $g_1$ and $g_2$. Indeed, if $v,w\in\pi_1^{-1}(y)$ then
$$\tilde{F}^{\cl}(v\otimes w+w\otimes v+2g_1(y)(v,w))=\tilde{f}(v)\otimes\tilde{f}(w)+\tilde{f}(w)\otimes \tilde{f}(v)+2g_1(y)(v,w)$$ by the above formula, and $2g_1(y)(v,w)=2g_2(f(y))(\tilde{f}(v),\tilde{f}(w))$
by the compatibility. Thus, $\tilde{F}^{\cl}$ preserves the defining relation for Clifford algebras, so indeed it is well-defined (on each fibre; hence on the whole pseudo-bundle).

The diffeomorphism $\Phi^{\cl}$, which we specify to be $\cl(V_1,g_1)\cup_{\tilde{F}^{\cl}}\cl(V_2,g_2)\to\cl(V_1\cup_{\tilde{f}}V_2,\tilde{g})$, is then the natural
identification; namely, by definitions of the gluing operation and that of the induced pseudo-metric $\tilde{g}$, over a point of form $x=i_1^{X_1}(x_1)$ both the fibre of
$\cl(V_1,g_1)\cup_{\tilde{F}^{\cl}}\cl(V_2,g_2)$ and that of $\cl(V_1\cup_{\tilde{f}}V_2,\tilde{g})$ are naturally identified with $\cl(\pi_1^{-1}(x),g_1|_{\pi_1^{-1}(x)})$, while over any point of
form $x=i_2^{X_2}(x_2)$ they are identified with $\cl(\pi_2^{-1}(x),g_2|_{\pi_2^{-1}(x)})$.

\subsubsection{Gluing of pseudo-bundles of Clifford modules}

A statement similar to that of the Theorem cited in the previous section can also be obtained for given Clifford modules over the algebras $\cl(V_1,g_1)$ and $\cl(V_2,g_2)$. This requires some
additional assumptions on these modules, and the appropriate notion of the compatibility of the actions.

\paragraph{The two Clifford modules} Let $\pi_1:V_1\to X_1$, $\pi_2:V_2\to X_2$, $(\tilde{f},f)$, $g_1$, and $g_2$ be as in Theorem \ref{gluing:clifford:cited:thm}. Recall that this yields the
following pseudo-bundles, $\pi_1^{\cl}:\cl(V_1,g_1)\to X_1$, $\pi_2^{\cl}:\cl(V_2,g_2)\to X_2$, and
$\pi_1^{\cl}\cup_{(\tilde{F}^{\cl},f)}\pi_2^{\cl}: \cl(V_1,g_1)\cup_{\tilde{F}^{\cl}}\cl(V_2,g_2)=\cl(V_1\cup_{\tilde{f}}V_2,\tilde{g})\to X_1\cup_f X_2$.

Now suppose that we are given two pseudo-bundles of Clifford modules, $\chi_1:E_1\to X_1$ and $\chi_2:E_2\to X_2$ (over $\cl(V_1,g_1)$ and $\cl(V_2,g_2)$ respectively), that is, there is a
smooth pseudo-bundle map $c_i:cl(V_i,g_i)\to\mathcal{L}(E_i,E_i)$ that covers the identity on the bases. Suppose further that there is a smooth fibrewise linear map
$\tilde{f}':\chi_1^{-1}(Y)\to\chi_2^{-1}(f(Y))$ that covers $f$. We describe the pseudo-bundle $E_1\cup_{\tilde{f}'}E_2$ as a Clifford module over $\cl(V_1\cup_{\tilde{f}}V_2,\tilde{g})$, with respect to
an action induced by $c_1$ and $c_2$.

\paragraph{Compatibility of $c_1$ and $c_2$} Similar to how it occurs for smooth maps, there is always an action of $\cl(V_1\cup_{\tilde{f}}V_2,\tilde{g})$ on $E_1\cup_{\tilde{f}'}E_2$
induced by $c_1$ and $c_2$, it is smooth on each fibre but in general, it might not be smooth across the fibres. For it to be so, we need a notion of compatibility for Clifford actions, which is as follows.

\begin{defn}
The actions $c_1$ and $c_2$ are \textbf{compatible} (with respect to $\tilde{F}^{\cl}$ and $\tilde{f}'$) if for all $y\in Y$, for all $v\in(\pi_1^{\cl})^{-1}(y)$, and for all $e_1\in\chi_1^{-1}(y)$ we have
$$\tilde{f}'(c_1(v)(e_1))=c_2(\tilde{F}^{\cl}(v))(\tilde{f}'(e_1)).$$
\end{defn}

We note that the compatibility of $c_1$ and $c_2$ as it has been just defined, does not automatically translate into their $(\tilde{F}^{\cl},\tilde{f}')$ compatibility as smooth maps in
$\cl(V_i,g_i)\to\mathcal{L}(E_i,E_i)$; see \cite{clifford-alg} for a discussion on this.

\paragraph{The induced action} Assuming now that the two given actions $c_1$ and $c_2$ are compatible in the sense just stated, we can define an induced action on $E_1\cup_{\tilde{f}'}E_2$, that is, a
smooth homomorphism
$$c:\cl(V_1\cup_{\tilde{f}}V_2,\tilde{g})\to\mathcal{L}(E_1\cup_{\tilde{f}'}E_2,E_1\cup_{\tilde{f}'}E_2).$$ Using the already-mentioned identification, via the diffeomorphism
$\Phi^{\cl}$, of $\cl(V_1\cup_{\tilde{f}}V_2,\tilde{g})$ with $\cl(V_1,g_1)\cup_{\tilde{F}^{\cl}}\cl(V_2,g_2)$, the action $c$ can be described by defining first
$$c'(v)(e)=\left\{\begin{array}{ll}
j_1^{E_1}\left(c_1((j_1^{\cl(V_1,g_1)})^{-1}(v))((j_1^{E_1})^{-1}(e))\right) & \mbox{if }v\in\mbox{Im}(j_1^{\cl(V_1,g_1)})\Rightarrow e\in\mbox{Im}(j_1^{E_1}), \\
j_2^{E_2}\left(c_2((j_2^{cl(V_2,g_2)})^{-1}(v))((j_2^{E_2})^{-1}(e))\right) & \mbox{if }v\in\mbox{Im}(j_2^{cl(V_2,g_2)})\Rightarrow e\in\mbox{Im}(j_2^{E_2}). \end{array}\right. $$
Since the images of the inductions $j_1^{\cl(V_1,g_1)}$ and $j_2^{\cl(V_2,g_2)}$ are disjoint and cover $\cl(V_1,g_1)\cup_{\tilde{F}^{\cl}}\cl(V_2,g_2)$, and those of
$j_1^{E_1}$ and $j_2^{E_2}$ cover $E_1\cup_{\tilde{f}'}E_2$ (and are disjoint as well), this is a well-defined fibrewise action of the former on the latter. From the formal point of view, we must also
pre-compose it with the inverse of $\Phi^{\cl}$, to obtain an action
$$c=c'\circ(\Phi^{\cl})^{-1}:\cl(V_1\cup_{\tilde{f}}V_2,\tilde{g})\to\mathcal{L}(E_1\cup_{\tilde{f}'}E_2,E_1\cup_{\tilde{f}'}E_2).$$ Then, the following is true (see \cite{clifford-alg}).

\begin{thm}
The action $c$ is smooth as a map $\cl(V_1\cup_{\tilde{f}}V_2,\tilde{g})\to\mathcal{L}(E_1\cup_{\tilde{f}'}E_2,E_1\cup_{\tilde{f}'}E_2)$.
\end{thm}

\section{The induced pseudo-metrics on dual pseudo-bundles}

Let $\pi:V\to X$ be a locally trivial finite-dimensional diffeological vector pseudo-bundle endowed with a pseudo-metric $g$; then its dual pseudo-bundle $\pi^*:V^*\to X$ admits an induced
pseudo-metric $g^*$ (see \cite{pseudometric-pseudobundle}; we also recall the definition below). In this section we consider gluings of two pseudo-bundles endowed with compatible pseudo-metrics, and the
corresponding dual constructions; starting from indicating which restrictions are imposed in the initial pseudo-bundles by the existence of compatible pseudo-metrics on them, we proceed to
discuss when the induced pseudo-metrics on the dual pseudo-bundles are compatible in their turn, and finally, how it relates to the gluing-dual commutativity condition.

\subsection{The induced pseudo-metric on $V^*$}

Let $\pi:V\to X$ be a locally trivial finite-dimensional diffeological vector pseudo-bundle, and let $g$ be a pseudo-metric on it. Under the assumption of local triviality of
$V$,\footnote{Which at the moment does not appear to be particularly limiting, in the sense that we do not know of any non-locally-trivial pseudo-bundles that admit pseudo-metrics in the
first place.} the dual pseudo-bundle $V^*$ carries an induced pseudo-metric $g^*$, which is obtained by what can be considered a diffeological counterpart of the usual natural pairing.

Specifically, consider the pseudo-bundle map $\Phi:V\to V^*$ defined by
$$\Phi(v)=g(\pi(v))(v,\cdot)\mbox{ for all }v\in V.$$ It is not hard to show (see \cite{pseudometric} for the case of a single diffeological vector space, and then \cite{pseudometric-pseudobundle} for the case of
pseudo-bundles) that $\Phi$ is surjective, smooth, and linear on each fibre.

The \textbf{induced pseudo-metric $g^*$} is given by the following equality:
$$g^*(x)(\Phi(v),\Phi(w)):=g(x)(v,w)\mbox{ for all }x\in X \mbox{ and for all }v,w\in V\mbox{ such that }\pi(v)=\pi(w)=x.$$ This is well-defined, because whenever $\Phi(v)=\Phi(v')$ (which
obviously can occur only for $v,v'$ belonging to the same fibre), the vectors $v$ and $v'$ differ by an element of the isotropic subspace of the fibre to which they belong. Furthermore, since each
fibre of the dual pseudo-bundle (in the finite-dimensional case) carries the standard diffeology (see \cite{pseudometric}), $g^*(x)$ is always a scalar product. Notice that without the requirement of local 
triviality, we cannot guarantee that $g^*$ is indeed a pseudo-metric, and more precisely, that it is smooth (the map $\Phi$ always has a right inverse, which \emph{a priori} may not be smooth). 

\begin{lemma}
Let $\pi:V\to X$ be a finite-dimensional diffeological vector pseudo-bundle endowed with a pseudo-metric $g$, let $\pi^*:V^*\to X$ be the dual pseudo-bundle, and let $g^*$ be the induced
pseudo-metric. Then for all $x\in X$ the symmetric bilinear form $g^*(x)$ on $(\pi^*)^{-1}(x)$ is non-degenerate.
\end{lemma}

\subsection{Existence of compatible pseudo-metrics: the case of a diffeological vector space}

Before treating various issues regarding the induced pseudo-metrics, it makes sense to consider in more detail what the compatibility of two pseudo-metrics means. We do so starting with the case of just
diffeological vector spaces (we consider the duals of vector spaces in the section that immediately follows, and pseudo-bundles in the one after that).

Let $V$ and $W$ be finite-dimensional diffeological vector spaces, let $g_V$ be a pseudo-metric on $V$, and let $g_W$ be a pseudo-metric on $W$. We assume that we are given a smooth linear
map $f:V\to W$, with respect to which $g_V$ and $g_W$ are compatible, $g_V(v_1,v_2)=g_W(f(v_1),f(v_2))$. We show that, quite similarly to usual vector spaces and scalar products, there are
pairs of diffeological ones such that no pair of pseudo-metrics is compatible with respect to any smooth linear $f$. The similarity that we are referring to has to do with the fact that the
compatibility of two pseudo-metrics with respect to $f$ essentially amounts to $f$ being a diffeological analogue of an isometry onto a subspace. As is well-known, between two usual vector spaces such
isometry may not exist (it is necessary that the dimension of the domain space must be less or equal to that of the target space), and something similar happens for diffeological vector spaces; and then
further conditions are added in terms of their diffeological structures.

\subsubsection{The characteristic subspaces of $V$ and $W$}

Assuming that two given pseudo-metrics $g_V$ and $g_W$ on diffeological vector spaces $V$ and $W$ respectively are compatible with a given $f:V\to W$ has several implications for the
diffeological structures of $V$ and $W$; describing these requires the following notion.

Given a pseudo-metric $g$ on a finite-dimensional diffeological vector space $V$, the subspace $V_0\leqslant V$ generated by all the eigenvectors of $g$ relative to the non-zero eigenvalues has subset
diffeology that is standard; and among all subspaces of $V$ whose diffeology is standard, it has the maximal dimension, which is equal to $\dim(V^*)$. In general, $V$ contains more than one subspace of
dimension $\dim(V^*)$ whose diffeology is standard. But the subspace $V_0$ is the only one that also splits off as a smooth direct summand.\footnote{Which means that the direct sum diffeology
coincides with $V$'s or $W$'s own diffeology, or, alternatively, that the composition of each plot of $V$ (respectively $W$) with the projection on $V_0$ (respectively $W_0$) is a plot of the latter.}
Thus, $V_0$ does not actually depend on the choice of a pseudo-metric and is an invariant of the space itself (see \cite{pseudometric}). We call this subspace the \textbf{characteristic subspace} of $V$.

Let us now return to the two diffeological vector spaces $V$ and $W$ above. Let $V_0$ and $W_0$ be their characteristic subspaces, and let $V_1\leqslant V$ and $W_1\leqslant W$ be the isotropic
subspaces relative to $g_V$ and $g_W$ respectively, such that $V=V_0\oplus V_1$ and $W=W_0\oplus W_1$ with each of these decompositions being smooth. We also recall \cite{pseudometric} that
$V_0$ not only has the same dimension as $V^*$, but for any fixed pseudo-metric is diffeomorphic to it, via (the restriction to $V_0$ of) the map $\Phi_V:v\mapsto g_V(v,\cdot)$; likewise, $W_0$ is
diffeomorphic to $W^*$ via $\Phi_W:w\mapsto g_W(w,\cdot)$.

\subsubsection{The necessary conditions}

Let us now assume that the given $g_V$, $g_W$, and $f$ satisfy the compatibility condition. The corollaries of this assumption can be described in terms of the characteristic subspaces of $V$ and $W$,
and therefore in terms of their diffeological duals.

\paragraph{The kernel of $f$} The first corollary is quite trivial, and starts with a simple linear algebra argument. Let $v\in V$ belong to the kernel of $f$. Then by the compatibility assumption
for $g_V$ and $g_W$ we have
$$g_V(v,v')=g_W(0,f(v'))=0\mbox{ for any }v'\in V.$$ Thus, the kernel of $f$ is contained in the maximal isotropic subspace $V_1$, therefore the restriction of $f$ to $V_0$ is a
bijection with its image. This restriction is of course a smooth map, and since $V_0$ splits off as a smooth direct summand, it is an induction (that is, a diffeomorphism with its image). Finally, $f$
itself is a diffeomorphism of $V_0\oplus\left(V_1/\mbox{Ker}(f)\right)$ with its image in $W$. In particular, we have the following.

\begin{lemma}\label{exist:compatible:pseudo-metrics:f:necessary:lem}
Let $V$ and $W$ be finite-dimensional diffeological vector spaces, and let $f:V\to W$ be a smooth linear map. If $V$ and $W$ admit pseudo-metrics compatible with $f$ then the $\mbox{Ker}(f)\cap V_0=\{0\}$.
\end{lemma}

Notice that in the standard case $V$ and $W$ would be vector spaces, $g_V$ and $g_W$ scalar products on them, and $f$ an isometry of $V$ with its image in $W$. In particular, $f$ would be injective;
Lemma \ref{exist:compatible:pseudo-metrics:f:necessary:lem} is the diffeological counterpart of that.

\paragraph{The dimensions of $V$ and $W$, and those of $V^*$ and $W^*$} Continuing the analogy with the standard case, we observe that the standard inequality $\dim(V)\leqslant\dim(W)$ does not have
to hold in the diffeological setting. What instead is true, is the corresponding inequality for the dimensions of their dual spaces, which follows from the lemma below.

\begin{lemma}
Let $V$ and $W$ be finite-dimensional diffeological vector spaces, let $f:V\to W$ be a smooth linear map, and suppose that $V$ and $W$ carry pseudo-metrics $g_V$ and $g_W$ respectively, compatible with
respect to $f$. Then the subset diffeology of $f(V_0)$ is the standard one.
\end{lemma}

\begin{proof}
Let $e_1,\ldots,e_n$ be a $g_V$-orthonormal basis of $V_0$; then by compatibility $f(e_1),\ldots,f(e_n)$ is a $g_W$-orthonormal basis of $f(V_0)$, which can be completed to a basis of eigenvectors of
$g_W$. It suffices to show that the projection of $W$ on the line generated by each $f(e_i)$ is a usual smooth function. Since this projection is given by $w\mapsto g_W(f(e_i),w)$, the claim follows from the
smoothness of $g_W$.
\end{proof}

Now, the fact that $f(V_0)$ carries the standard diffeology, does not automatically imply that it is contained in $W_0$ --- there are standard subspaces that are not (we will however show later on that
this inclusion does hold for $f(V_0)$). However, $f(V_0)$ is still a standard subspace of $W$, and since $W_0$ has maximal dimension among such subspaces, we have
$$\dim(V^*)=\dim(V_0)=\dim(f(V_0))\leqslant\dim(W_0)=\dim(W^*).$$ Therefore we have the following statement.

\begin{prop}
Let $V$ and $W$ be finite-dimensional diffeological vector spaces. If there exist a smooth linear map $f:V\to W$ and pseudo-metrics $g_V$ and $g_W$ on $V$ and $W$ respectively, compatible with respect
to $f$, then
$$\dim(V^*)\leqslant\dim(W^*).$$
\end{prop}

In other words, if $\dim(V^*)>\dim(W^*)$, then no two pseudo-metrics on $V$ and $W$ are compatible, whatever the map $f$ (which obviously mimics the standard situation: there is no isometry from the space
of a bigger dimension to one of smaller dimension).\footnote{The choices of $f$ however could be plenty; it suffices to take $V$ the standard $\matR^n$ and $W$ any other diffeological vector space of
dimension strictly smaller than $n$. Any linear map from $V$ to $W$ is then going to be smooth (see Section 3.9 in \cite{iglesiasBook}).}

\paragraph{The subspace $f(V_0)$ in $W$} We now show that the \emph{a priori} case when $f(V_0)$ is not contained in $W_0$ is actually impossible, that is, if $f:V\to W$ is such that $V$ and $W$ admit
compatible pseudo-metrics then $f$ sends the characteristic subspace of $V$ to the characteristic subspace of $W$.

\begin{lemma}\label{compatible:pseudometrics:image:smooth:summand:lem}
Let $V$ and $W$ be finite-dimensional diffeological vector spaces, let $f:V\to W$ be a smooth linear map, and suppose that $V$ and $W$ admit compatible pseudo-metrics $g_V$ and $g_W$ respectively.
Then $f(V_0)$ splits off smoothly in $W$.
\end{lemma}

\begin{proof}
Let $e_1,\ldots,e_n$ be a $g_V$-orthonormal basis of $V_0$. Then by assumption $f(e_1),\ldots,f(e_n)$ is a $g_W$-orthonormal basis of $f(V_0)$. This can be completed to an orthogonal basis of $W$
composed of eigenvectors of $g_W$; denote by $u_1,\ldots,u_k$ the elements added, ordered in such a way that the eigenvectors corresponding to the zero eigenvalue are the last $m$ vectors. Let
us show that the usual direct sum decomposition $W=f(V_0)\oplus\mbox{Span}(u_1,\ldots,u_k)$ is a smooth one.

Let $p:U\to W$ be a plot of $W$, and let $p'$ be its composition with the projection (associated to the direct sum decomposition just mentioned) of $W$ to $f(V_0)$. It suffices to show that $p'$ is a
plot of $f(V_0)$. Notice that by the choice of the basis $f(e_1),\ldots,f(e_n),u_1,\ldots,u_k$ of $W$ (more precisely, by the $g_W$-orthogonality of said basis) we have
$$p'(u)=g_W(f(e_1),p(u))f(e_1)+\ldots+g_W(f(e_n),p(u))f(e_n),$$ where each coefficient $g_W(f(e_i),p(u))$ is an ordinary smooth function $U\to\matR$ by the smoothness of the pseudo-metric $g_W$.
This means precisely that $p'$ is a plot of $f(V_0)$, whence the claim.
\end{proof}

From the lemma just proven, we can now easily draw the following conclusion.

\begin{cor}
Under the assumptions of Lemma \ref{compatible:pseudometrics:image:smooth:summand:lem}, there is the inclusion $f(V_0)\leqslant W_0$.
\end{cor}

\begin{proof}
The subspace $W_0$ is the only subspace of dimension equal to that of $W^*$ that has standard diffeology and splits off smoothly. Since $f(V_0)\oplus W_0'$ has all the same properties, we obtain that
$f(V_0)\oplus W_0'=W_0$.
\end{proof}

\paragraph{The summary of necessary conditions} We collect the conclusions of this section in the following statement.

\begin{thm}\label{necessary:exist:pseudometrics:vspaces:thm}
Let $V$ and $W$ be two finite-dimensional diffeological vector spaces, and let $f:V\to W$ be a smooth linear map. If there exist pseudo-metrics $g_V$ and $g_W$ on $V$ and $W$ respectively that are
compatible with respect to $f$ then the following are true:
\begin{enumerate}
\item $\dim(V^*)\leqslant\dim(W^*)$;
\item $\mbox{Ker}(f)\cap V_0=\{0\}$, where $V_0$ is the characteristic subspace of $V$;
\item The subset diffeology on $f(V_0)$ relative to its inclusion into $W$ is standard;
\item $f(V_0)$ splits off smoothly in $W$.
\end{enumerate}
\end{thm}

\subsubsection{Sufficient conditions}

Suppose now that $V$ and $W$ are such that the just-mentioned necessary condition is satisfied, and let $f:V\to W$ be a smooth linear map such that $\mbox{Ker}(f)\cap V_0=\{0\}$ (where $V_0$ is
the characteristic subspace of $V$). By definition of a pseudo-metric, if $V=V_0\oplus V_1$ is a smooth decomposition of $V$\footnote{As we have said already, $V_0$ is uniquely defined by
$V$; this is not necessarily true of $V_1$, which is defined uniquely only when the pseudo-metric has been fixed already. However, there is always at least one choice of $V_1$; what we mean
at the moment is that such a choice is fixed arbitrarily.} then $g_V$ is defined by its restriction to $V_0$ (which is a scalar product) and is extended by zero elsewhere. The same is true of
$g_W$ and the corresponding smooth decomposition $W_0\oplus W_1$. In this way we obtain the following.

\begin{prop}\label{compatible:pseudometrics:vspace:st-to-st:prop}
Let $V$ and $W$ be finite-dimensional diffeological vector spaces such that $\dim(V^*)\leqslant\dim(W^*)$, and let $f:V\to W$ be a smooth linear map such that $\mbox{Ker}(f)\cap V_0=\{0\}$ and
$f(V_0)\leqslant W_0$. Then $V$ and $W$ admit pseudo-metrics compatible with respect to $f$.
\end{prop}

\begin{proof}
Let us fix smooth decompositions $V=V_0\oplus V_1$ and $W=W_0\oplus W_1$. To construct a pseudo-metric $g_V$, choose a basis $v_1,\ldots,v_k$ of $V_0$ and a basis $v_{k+1},\ldots,v_n$ of $V_1$;
then set
$$g_V(v_i,v_j)=\delta_{i,j}\mbox{ for }i,j=1,\ldots,k\mbox{ and }g_V(v_i,v_{k+j})=0\mbox{ for }i=1,\ldots,n,\,j=1,\ldots,n-k,$$ and extend by bilinearity and symmetricity. To define $g_W$, then,
consider $f(v_1),\ldots,f(v_k)\in W_0$; notice that they are linearly independent by the assumption on $\mbox{Ker}(f)$. Add first $u_1,\ldots,u_l\in W_0$ to obtain the basis
$f(v_1),\ldots,f(v_k),u_1,\ldots,u_l$ of $W_0$. Finally, choose a basis $w_1,\ldots,w_m$ of $W_1$ to obtain the basis $f(v_1),\ldots,f(v_k),u_1,\ldots,u_l,w_1,\ldots,w_m$ of the whole
$W$. It then suffices to define $g_W$ to be
$$g_W(f(v_i),f(v_j))=\delta_{i,j},\,\,\,g_W(u_i,u_j)=\delta_{i,j},\,\,\,g_W(f(v_i),u_j)=0,\,\,\,g_W(f(v_i),w_p)=0,\,\,\, g_W(u_i,w_p)=0$$ and extend by bilinearity and symmetry. The bilinear maps $g_V$
and $g_W$ thus obtained are smooth, because each of the characteristic subspaces $V_0$ and $W_0$ splits off as a smooth direct summand, and by construction $g_V$ and $g_W$ are zero maps
outside of $V_0$ and $W_0$ respectively. Finally, that they are pseudo-metrics and are compatible with each other is immediate from their definitions, whence the conclusion.
\end{proof}

We are ready to establish the final criterion of the existence of compatible pseudo-metrics on a pair of diffeological vector spaces, that we state in the following form.

\begin{thm}\label{criterio:exist:pseudometrics:vspaces:thm}
Let $V$ and $W$ be two finite-dimensional diffeological vector spaces, and let $f:V\to W$ be a smooth linear map. Then $V$ and $W$ admit compatible pseudo-metrics if and only if
$\mbox{Ker}(f)\cap V_0=\{0\}$ and $f(V_0)\leqslant W_0$.
\end{thm}

\begin{proof}
The fact that these two conditions are necessary follows from Lemma \ref{exist:compatible:pseudo-metrics:f:necessary:lem} and Lemma \ref{compatible:pseudometrics:image:smooth:summand:lem}, so let us
show that they are sufficient. Let $V=V_0\oplus V_1$ be a smooth decomposition, let $g_V$ be any pseudo-metric on $V$, and let $W=f(V_0)\oplus W'$ be a smooth decomposition that exists by
assumption. Notice that, since $W'$ is just another instance of a finite-dimensional diffeological space, it has its own smooth decomposition of form $W_0'\oplus W_1$, where $W_0'$ is standard and
$W_1$ has trivial diffeological dual; so the whole of $W$ smoothly decomposes as $W=\left(f(V_0)\oplus W_0'\right)\oplus W_1$. Notice that this implies that $W^*=\left(f(V_0)\oplus W_0'\right)^*$ (by
the smoothness of the decomposition), in particular, they have the same dimension.

Let us define a pseudo-metric $g_W$, by setting it to coincide with $g_V$ (in the obvious sense) on $f(V_0)$, choosing any scalar product for its restriction on $W_0'$, while requiring $W_0'$ to be
orthogonal to $f(V_0)$, and finally setting $W_1$ to be an isotropic subspace. That this is indeed a pseudo-metric follows from the considerations above, so it remains to show that $g_W$ is indeed
compatible with $g_V$. This essentially follows from the construction, more precisely, from the fact that $f(V_0)$ is orthogonal to any its direct complement. Indeed, if
$v'=v_0'+v_1'$ and $v''=v_0''+v_1''$ are any two elements of $V$, then
$$g_V(v',v'')=g_V(v_0',v_0'')=g_W(f(v_0'),f(v_0''))=g_W(f(v_0')+f(v_1'),f(v_0'')+f(v_1''))=g_W(f(v'),f(v'')),$$ where the third equality is by the orthogonality just mentioned. This means that $g_V$ and $g_W$ are
compatible with $f$, and the proof is finished.
\end{proof}

\subsection{Compatibility of the dual pseudo-metrics: diffeological vector spaces}

We now consider the induced pseudo-metrics on the duals of diffeological vector spaces; the main question that we aim to answer is, under what conditions the pair of pseudo-metrics dual to
(induced by) two compatible ones is in turn compatible.

\paragraph{The induced pseudo-metric $g^*$ on $V^*$: definition} Recall (\cite{pseudometric}) that, given a finite-dimensional diffeological vector space $V$ endowed with a pseudo-metric $g$, the
diffeological dual of $V$ carries the induced pseudo-metric $g^*$ (actually, a scalar product, since the diffeological dual of any finite-dimensional diffeological vector space is standard) defined by
$$g^*(v_1^*,v_2^*):=g(v_1,v_2),$$ where $v_i\in V$ is any element such that $v_i^*(\cdot)=g(v_i,\cdot)$ for $i=1,2$. That this is well-defined, \emph{i.e.}, the result does not depend on the choice of $v_i$ (as
long as $g(v_i,\cdot)$ remains the same), and that $v_i^*$ always admits such a form, was shown in \cite{pseudometric}.

\paragraph{The compatibility for the induced pseudo-metrics} Let $g_V$ and $g_W$ be pseudo-metrics on $V$ and $W$ respectively, compatible with respect to $f$. Let $w_1^*,w_2^*\in W^*$; then there
exist $w_1,w_2\in W$, defined up to the cosets of the isotropic subspace of $g_W$, such that $w_i^*(\cdot)=g_W(w_i,\cdot)$ for $i=1,2$, by definition of the dual pseudo-metric
$$g_W^*(w_1^*,w_2^*)=g_W(w_1,w_2),$$ and finally, $f^*(w_i^*)(\cdot)=w_i^*(f(\cdot))=g_W(w_i,f(\cdot))$.

The compatibility condition that we need to check is the following one:
$$g_W^*(w_1^*,w_2^*)=g_V^*(f^*(w_1^*),f^*(w_2^*)).$$ Now, in order to calculate the right-hand term in this expression, we must choose $v_1$ and $v_2$, again defined up to their cosets with respect to
the isotropic subspace of $g_V$, such that $g_V(v_i,v')=f^*(w_i^*)(v')=g_W(w_i,f(v'))$, for all elements of $v'\in V$ and for $w_1^*,w_2^*\in W^*$. The term on the right then becomes
$g_V^*(f^*(w_1^*),f^*(w_2^*))=g_V(v_1,v_2)$.

\paragraph{The dual pseudo-metrics and compatibility} Let us now consider the pseudo-metrics on $V^*$ and $W^*$ dual to a pair of compatible pseudo-metrics on $V$ and $W$. We observe right away that 
in general, the induced pseudo-metrics are not compatible. This follows from Lemma \ref{exist:compatible:pseudo-metrics:f:necessary:lem}, as well as from the standard theory, all diffeological constructions 
being in fact extensions of the standard ones.

\begin{example}
Let $V$ be the standard $\matR^n$, with the canonical basis denoted by $e_1,\ldots,e_n$, and let $W$ be the standard $\matR^{n+k}$, with the canonical basis denoted by
$u_1,\ldots,u_n,u_{n+1},\ldots,u_{n+k}$. Let $f:V\to W$ be the embedding of $V$ via the identification of $V$ with the subspace generated by $u_1,\ldots,u_n$, given by $e_i\mapsto u_i$ for
$i=1,\ldots,n$. Let $g_V$ be any scalar product on $\matR^n$; this trivially induces a scalar product on $f(V)=\mbox{Span}(u_1,\ldots,u_n)\leqslant W_0$, and let $g_W$ be
any extension of it to a scalar product on the whole $W$.

Let us consider the dual map on the dual the standard complement of the subspace $\mbox{Span}(u_1,\ldots,u_n)$, that is, on the dual of $\mbox{Span}(u_{n+1},\ldots,u_{n+k})$. This dual is the usual dual,
so it is $\mbox{Span}(u^{n+1},\ldots,u^{n+k})$. Let $v$ be any element of $V$; since $f(v)\in\mbox{Span}(u_1,\ldots,u_n)$, we have
$$f^*(u^{n+i})(v)=u^{n+i}(f(v))=0,$$ so in the end we obtain that $\mbox{Ker}(f^*)=\mbox{Span}(u^{n+1},\ldots,u^{n+k})$.

Finally, let us consider the compatibility condition. We observe that
$$g_W^*(u^{n+i},u^{n+i})=g_W(u_{n+i},u_{n+i})>0,$$ since $g_W$ is a scalar product, while, of course,
$$g_V^*(f^*(u^{n+i}),f^*(u^{n+i}))=0.$$ Quite evidently, the compatibility condition cannot be satisfied (unless $k=0$).
\end{example}

\paragraph{Sufficient conditions for compatibility of the induced pseudo-metrics} It can be inferred from the above example that the induced pseudo-metrics on the duals of standard spaces are
compatible only if the spaces have the same dimension (which is not surprising, since in this case the notion of the induced pseudo-metric itself coincides with the standard one). This can be
generalized to the following statement.

\begin{thm}\label{crit:dual:pseudometrics:comp:vspaces:thm}
Let $V$ and $W$ be two finite-dimensional diffeological vector spaces, and let $f:V\to W$ be a smooth linear map such that $\mbox{Ker}(f)\cap V_0=\{0\}$ and $f(V_0)\leqslant W_0$. Let $g_V$
and $g_W$ be compatible pseudo-metrics on $V$ and $W$ respectively. Then the induced pseudo-metrics $g_W^*$ and $g_V^*$ are compatible with $f^*$ if and only if $f^*:W^*\to V^*$ is a diffeomorphism.
\end{thm}

\begin{proof}
The \emph{only if} part of the statement, illustrated by the example above, follows from standard reasoning. Indeed, $g_W^*$ and $g_V^*$ are usual scalar products on standard spaces $W^*$ and $V^*$
respectively, and their compatibility means that $f^*$ is a usual isometry, whose existence implies that $W^*$ and $V^*$ have the same dimension, and being standard spaces, this means that they are
diffeomorphic as diffeological vector spaces.

Let us prove the \emph{if} part, namely, that $g_W^*$ and $g_V^*$ are compatible under the assumptions of the proposition. Let $e_1,\ldots,e_n$ be a $g_V$-orthogonal basis of $V_0$. Since $g_V$
and $g_W$ are compatible, $f(e_1),\ldots,f(e_n)$ is a $g_W$-orthogonal basis of $f(V_0)$. Now, $(e_1)^*,\ldots,(e_n)^*$ (recall here that $v^*$ for $v\in V$ stands for the map
$v^*(\cdot)=g_V(v,\cdot)$) form a basis of $V^*$ (which is also orthogonal with respect to the induced pseudo-metric $g_V^*$), while $(f(e_1))^*,\ldots,(f(e_n))^*$ are linearly independent elements of
$W^*$. Since $V^*$ and $W^*$ have the same dimension, $(f(e_1))^*,\ldots,(f(e_n))^*$ actually forms a basis of $W^*$; and so, $g_W^*$ is entirely determined by its values on pairs
$(f(e_i))^*,(f(e_j))^*$, and moreover, we have $$g_W^*((f(e_i))^*,(f(e_j))^*)=g_W(f(e_i),f(e_j))=g_V(e_i,e_j)=g_V^*(e_i^*,e_j^*).$$

Finally, $f^*((f(e_i))^*)(v)=(f(e_i))^*(f(v))=g_W(f(e_i),f(v))=g_V(e_i,v)=e_i^*(v)$ for all $v\in V$. Therefore
$$g_V^*(e_i^*,e_j^*)=g_V^*(f^*((f(e_i))^*),f^*((f(e_j))^*))=g_W^*((f(e_i))^*,(f(e_j))^*),$$ at which point the compatibility, with respect to $f^*$, of the pseudo-metrics $g_W^*$ and $g_V^*$ follows from
$(f(e_1))^*,\ldots,(f(e_n))^*$ being a basis of $W^*$.
\end{proof}

\subsection{Compatibility of the dual pseudo-metrics: diffeological vector pseudo-bundles}

Let us now consider the following question: if two given pseudo-metrics $g_1$ and $g_2$ are compatible with respect to the gluing along a given pair of maps $(f,\tilde{f})$, when is it true
that $g_2^*$ and $g_1^*$ are compatible with the gluing defined by $(f^{-1},\tilde{f}^*)$? (Obviously, we assume here that $f$ is invertible).

\paragraph{The compatibility condition for $g_2^*$ and $g_1^*$} Let $\pi_1:V_1\to X_1$ and $\pi_2:V_2\to X_2$ be locally trivial finite-dimensional diffeological vector pseudo-bundles, let
$(\tilde{f},f)$ be a gluing of the former to the latter such that $f$ is a diffeomorphism with its image, and let $g_1$ and $g_2$ be pseudo-metrics on $V_1$ and $V_2$ respectively, compatible with
respect to the given gluing. The latter induces a well-defined gluing, along the maps $\tilde{f}^*$ and $f^{-1}$, of the dual pseudo-bundle $\pi_2^*:V_2^*\to X_2$ to the pseudo-bundle
$\pi_1^*:V_1^*\to X_1$, the result of which is the pseudo-bundle $\pi_2^*\cup_{(\tilde{f}^*,f^{-1})}\pi_1^*:V_2^*\cup_{\tilde{f}^*}V_1^*\to X_2\cup_{f^{-1}}X_1$, while $g_2$ and $g_1$ induce
pseudo-metrics $g_2^*:X_2\to(V_2^*)^*\otimes(V_2^*)^*$ and $g_1^*:X_1\to(V_1^*)^*\otimes(V_1^*)^*$ on the dual pseudo-bundles. They satisfy the usual compatibility condition if
$$g_1^*(f^{-1}(y'))(\tilde{f}^*(v^*),\tilde{f}^*(w^*))=g_2^*(y')(v^*,w^*)$$ $$\mbox{for all }y'\in Y'=f(Y)\mbox{ and for all }v^*,w^*\in(\pi_2^{-1}(y'))^*.$$

\paragraph{The necessary condition} The compatibility between $g_2^*$ and $g_1^*$ implies in particular that for all $y\in Y$ the pseudo-metrics $g_2^*(f(y))$ and $g_1^*(y)$ are compatible with the
smooth linear map $\tilde{f}^*|_{\pi_1^{-1}(y)}$ between diffeological vector spaces $(\pi_2^{-1}(f(y)))^*$ and $(\pi_1^{-1}(y))^*$. Thus, Theorem \ref{crit:dual:pseudometrics:comp:vspaces:thm} implies that
$\tilde{f}^*$ is a diffeomorphism on each fibre.

\begin{prop}
Let $\pi_1:V_1\to X_1$ and $\pi_2:V_2\to X_2$ be diffeological vector pseudo-bundles, locally trivial and with finite-dimensional fibres, and let $(\tilde{f},f)$ be an invertible gluing between
them. Suppose that $g_1$ and $g_2$ are two pseudo-metrics on these pseudo-bundles compatible with the gluing along $(\tilde{f},f)$. If the induced pseudo-metrics $g_2^*$ and $g_1^*$ are compatible with
the gluing along $(\tilde{f}^*,f^{-1})$ then the restriction of $\tilde{f}^*$ on each fibre in its domain of definition is a diffeomorphism.
\end{prop}

\paragraph{Criterion of compatibility} The statement that follows shows that, for the two induced pseudo-metrics to be compatible with the induced gluing, the map dual to the gluing map $\tilde{f}$ must
satisfy a rather stringent condition (although an expected one). 

\begin{thm}\label{when:dual:pseudometrics:compatible:bundles:thm}
Let $\pi_1:V_1\to X_1$ and $\pi_2:V_2\to X_2$ be two diffeological vector pseudo-bundles, locally trivial and with finite-dimensional fibres, let $(\tilde{f},f)$ be a gluing between them, and let $g_1$
and $g_2$ be compatible pseudo-metrics on $V_1$ and $V_2$ respectively. Then the induced pseudo-metrics $g_2^*$ and $g_1^*$ on the corresponding dual pseudo-bundles are compatible if and only if
$\tilde{f}^*$ is a pseudo-bundle diffeomorphism of its domain with its image.
\end{thm}

Notice that diffeological vector spaces may have diffeomorphic duals without being diffeomorphic themselves, and the same is true for diffeological vector pseudo-bundles.

\begin{proof}
By assumption, $g_1$ and $g_2$ are compatible with the gluing given by the pair $(\tilde{f},f)$, that is
$$g_2(f(y))(\tilde{f}(v),\tilde{f}(w))=g_1(y)(v,w)\mbox{ for all }y\in Y\mbox{ and for all }v,w\in\pi_1^{-1}(y).$$ Suppose first that $\tilde{f}^*$ is a diffeomorphism with its image.
Then, first of all, by the definition of $g_2^*$ we have 
$$g_2^*(y')(v^*,w^*)=g_2(y')(v,w)$$ for all $y'\in f(Y)$, for all $v^*,w^*\in(\pi_2^{-1}(y'))^*$, and for $v,w\in(\pi_2^{-1}(y'))_0$. Notice that $v$ and $w$ are uniquely
defined by the latter condition; and they are such that $v^*(\cdot)=g_2(y')(v,\cdot)$ and $w^*(\cdot)=g_2(y')(w,\cdot)$. Notice also (we will need this immediately below) that this means
$$\tilde{f}^*(v^*)(\cdot)=v^*(\tilde{f}(\cdot))=g_2(y')(v,\tilde{f}(\cdot))=g_2(y')(\tilde{f}(v_1),\tilde{f}(\cdot))=g_1(f^{-1}(y'))(v_1,\cdot),$$
where $v_1\in(\pi_1^{-1}(f^{-1}(y')))_0$ is such that $v=\tilde{f}(v_1)$; such an element exists and is uniquely defined because $\tilde{f}^*$ being a diffeomorphism is equivalent to
$\tilde{f}$ being a diffeomorphism between each pair of subspaces $(\pi_1^{-1}(f^{-1}(y')))_0$ and $(\pi_2^{-1}(y'))_0$. Similarly, we have
$$\tilde{f}^*(w^*)(\cdot)=g_2(y')(\tilde{f}(w_1),\tilde{f}(\cdot))=g_1(f^{-1})(w_1,\cdot)$$ for $w_1\in(\pi_1^{-1}(f^{-1}(y')))_0$ such that $w=\tilde{f}(w_1)$. It remains now to consider the left-hand
part of the compatibility condition. We have:
$$g_1^*(f^{-1}(y'))(\tilde{f}^*(v^*),\tilde{f}^*(w))=g_1(f^{-1}(y'))(v_1,w_1)=g_2(y')(\tilde{f}(v_1),\tilde{f}(w_1))=g_2(y')(v,w)=g_2^*(y')(v^*,w^*),$$ as wanted.

Let us prove the \emph{vice versa} of the statement, that is, let us assume that $g_2^*$ and $g_1^*$ are compatible, and let us show that $\tilde{f}^*$ is a diffeomorphism. We notice first of all that it
follows from the considerations made for individual vector spaces that $\tilde{f}^*$ is bijective and, as is the case for any dual map, it is smooth. Finally, the smoothness of its inverse follows
from the fact that $V_1^*$ and $V_2^*$ are locally trivial and have standard fibres.
\end{proof}

\section{Compatibility of pseudo-metrics and the gluing-dual commutativity conditions}

In this section we consider which correlations there might be between the notion of compatible pseudo-metrics on two given pseudo-bundles (with a specified gluing), and the gluing-dual
commutativity conditions. We start by taking our two usual pseudo-bundles, $\pi_1:V_1\to X_1$ and $\pi_2:V_2\to X_2$, and a gluing along $(\tilde{f},f)$ between them. Assuming that these
pseudo-bundles admit pseudo-metrics $g_1$ and $g_2$ compatible with the gluing, we consider the following questions: 1) what are the implications of the existence of compatible pseudo-metrics for the
pseudo-bundles themselves? 2) if the gluing-dual commutativity condition holds for $V_1$, $V_2$, and $(\tilde{f},f)$, does it necessarily hold for $V_2^*$, $V_1^*$, and $(\tilde{f}^*,f^{-1})$?
3) if $g_2^*$ and $g_1^*$ exist, under what conditions are they compatible with $(\tilde{f}^*,f^{-1})$, in particular, is their compatibility equivalent to the gluing-dual commutativity? 4) does
taking the dual pseudo-metric commute (in the notation to be introduced, this will be the equality $\tilde{g}^*=\widetilde{g^*}$) with the gluing of pseudo-metrics (as defined in
\cite{pseudometric-pseudobundle})? We consider these questions in order, after quickly introducing a preliminary notion.

\subsection{The characteristic sub-bundle of a finite-dimensional vector pseudo-bundle}

Let $\pi:V\to X$ be a finite-dimensional diffeological vector pseudo-bundle, and let $(\pi^{-1}(x))_0$ be the characteristic subspace of the fibre $\pi^{-1}(x)$. Denote by $V_0$ the sub-bundle
of $V$ defined as 
$$V_0:=\cup_{x\in X}(\pi^{-1}(x))_0.$$ We say that $V_0$ is the \textbf{characteristic sub-bundle} of the pseudo-bundle $V$. It is evident from the construction that the
characteristic sub-bundle of a locally trivial pseudo-bundle is itself locally trivial. Furthermore, every pseudo-metric on $V$ is uniquely defined by its restriction to $V_0$. Finally, the
\emph{vice versa} of the latter statement is also true, if we assume $V$ to be locally trivial.

\begin{lemma}
Let $\pi:V\to X$ be a locally trivial diffeological vector pseudo-bundle, let $\pi_0:V_0\to X$ be its characteristic sub-bundle, and let $g_0$ be a pseudo-metric on $V_0$. Then there
exists one, and only one, pseudo-metric $g$ on $V$ whose restriction on $V_0$ coincides with $g_0$.
\end{lemma}

\begin{proof}
It suffices to define $g(x)(v',v'')=\left\{\begin{array}{ll} g_0(x)(v',v'') & \mbox{if }v',v''\in\pi_0^{-1}(x)\\ 0 & \mbox{otherwise}; \end{array}\right.$ the conclusion then follows from the definitions
of a pseudo-metric and that of the characteristic sub-bundle.
\end{proof}

\begin{prop}
Let $\pi:V\to X$ be a locally trivial finite-dimensional diffeological vector pseudo-bundle that admits a pseudo-metric $g$. Then its characteristic sub-bundle $\pi_0:V_0\to X$ is diffeomorphic
to its dual pseudo-bundle $\pi^*:V^*\to X$ via the natural pairing map associated to $g$.
\end{prop}

\begin{proof}
Let $\psi_g:V_0\to V^*$ be the natural pairing associated to $g$, that is,
$$\psi_g(v)(\cdot)=g(\pi(v))(v,\cdot).$$ That this is a bijection follows from its fibrewise nature and it being a bijection on each individual fibre (see
\cite{pseudometric}); furthermore (see the same source), as a map on the characteristic subspace it is a diffeomorphism with the dual fibre. By the assumption of local triviality this implies that
$\psi_g$, as well as its inverse, are smooth across the fibres as well, so they are smooth as a whole, whence the conclusion.\footnote{It is easy prove that $\psi_g$ is smooth even
if we do not assume $V$ to be locally trivial.}
\end{proof}

\subsection{Implications of compatibility of $g_1$ and $g_2$ for $V_1$ and $V_2$}

This extends the criterion for diffeological vector spaces (Theorem \ref{criterio:exist:pseudometrics:vspaces:thm}). The pseudo-bundle version is an immediate consequence and is as follows.

\begin{prop}
Let $\pi_1:V_1\to X_1$ and $\pi_2:V_2\to X_2$ be two diffeological vector pseudo-bundles, locally trivial and with finite-dimensional fibres, let $(\tilde{f},f)$ be a gluing between them, and let $g_1$
and $g_2$ be compatible pseudo-metrics. Then $\tilde{f}$ determines, over the domain of gluing, a smooth embedding of the characteristic sub-bundle of $V_1$ into the characteristic sub-bundle of $V_2$.
\end{prop}

\begin{proof}
This follows directly from the already-mentioned Theorem \ref{criterio:exist:pseudometrics:vspaces:thm}, applied to the restriction of $\tilde{f}$ on each fibre in its domain of
definition; the theorem affirms that such restriction is an embedding of each characteristic subspace, so the fibre of the characteristic sub-bundle, of $V_1$ into that of $V_2$. We should
only add that the restriction of $\tilde{f}$ onto the intersection of its domain of definition with the characteristic sub-bundle of $V_1$ is smooth across the fibres, because $\tilde{f}$ is so.
\end{proof}

\subsection{The gluing-dual commutativity condition and gluing along a diffeomorphism}

We now recall a statement (which essentially appears in \cite{pseudobundles}, Lemma 5.17) that (together with some results from the previous sections) will allow us to deduce the gluing-dual
commutativity in a number of cases. The statement basically is that if the gluing of two pseudo-bundles is performed along a diffeomorphism, then the gluing-dual commutativity condition always
holds. We also add the explicit construction of the commutativity diffeomorphism (which was not specified in the above source).

\begin{thm}\label{gluing:diffeo:implies:gluing-dual:commute:thm}
Let $\chi_1:W_1\to X_1$ and $\chi_2:W_2\to X_2$ be two diffeological vector pseudo-bundles, let $h:X_1\supseteq Y\to X_2$ be a smooth invertible map with smooth inverse, and let $\tilde{h}$ be its
smooth fibrewise linear lift that is a diffeomorphism of its domain with its image. Then the map
$$\Psi_{\cup,*}:(W_1\cup_{\tilde{h}}W_2)^*\to W_2^*\cup_{\tilde{h}^*}W_1^*$$ defined by 
$$\Psi_{\cup,*}=\left\{\begin{array}{ll}
j_2^{W_1^*}\circ(j_1^{W_1})^* & \mbox{on }((\chi_1\cup_{(\tilde{h},h)}\chi_2)^*)^{-1}(i_1^{X_1}(X_1\setminus Y)) \\
j_2^{W_1^*}\circ\tilde{h}^*\circ(j_2^{W_2})^* & \mbox{on }((\chi_1\cup_{(\tilde{h},h)}\chi_2)^*)^{-1}(i_2^{X_2}(h(Y)) \\
j_1^{W_2^*}\circ(j_2^{W_2})^* & \mbox{on }((\chi_1\cup_{(\tilde{h},h)}\chi_2)^*)^{-1}(i_2^{X_2}(X_2\setminus h(Y))
\end{array}\right.$$ is a pseudo-bundle diffeomorphism covering the switch map $\varphi_{X_1\leftrightarrow X_2}$.
\end{thm}

\begin{proof}
It is easy to see that $\Psi_{\cup,*}$ is a bijection, so let us show that it is smooth (the proof that its inverse is smooth is then analogous). Let us first consider the general shape of an arbitrary
plot $q^*$ of $(W_1\cup_{\tilde{h}}W_2)^*$, and that of an arbitrary plot $s^*$ of $W_2^*\cup_{\tilde{h}^*}W_1^*$. Let $q^*:U\to(W_1\cup_{\tilde{h}}W_2)^*$ be any plot; we can however
assume that $U$ is connected, so $(\chi_1\cup_{(\tilde{h},h)}\chi_2)\circ q$, which is a plot of $X_1\cup_h X_2$, lifts to a plot of $X_1$ or to a plot of $X_2$. This means that $q^*$ acts only 
on fibres of $W_1\cup_{\tilde{h}}W_2$ that pullback to $W_1$ or to $W_2$, respectively. In the former case we have that there exists a plot $q_1^*$ of $W_1^*$ such that
$$q^*=\left\{\begin{array}{ll} ((j_1^{W_1})^*)^{-1}\circ q_1^* & \mbox{over }i_1^{X_1}(X_1\setminus Y) \\
((j_2^{W_2})^*)^{-1}\circ(\tilde{h}^*)^{-1}\circ q_1^* & \mbox{over }i_2^{X_2}(h(Y));
\end{array}\right.$$ in the latter case there exists a plot $q_2^*$ of $W_2^*$ such that
$$q^*=((j_2^{W_2})^*)^{-1}\circ q_2^*.$$ For analogous reasons, if $s^*:U'\to W_2^*\cup_{\tilde{h}^*}W_1^*$ defined on a connected domain $U'$, then either there exists a plot $s_2^*$ of $W_2^*$ such that
$$s^*=\left\{\begin{array}{ll} j_1^{W_2^*}\circ s_2^* & \mbox{over }i_1^{X_2}(X_2\setminus h(Y)) \\
j_2^{W_1^*}\circ\tilde{h}^*\circ s_2^* & \mbox{over }i_2^{X_1}(h(Y)),
\end{array}\right.$$ or else there exists a plot $s_1^*$ of $W_1^*$ such that
$$s^*=j_2^{W_2^*}\circ s_1^*.$$

Let us consider $\Psi_{\cup,*}\circ q^*$ for $q^*$ of the first and the second type. Assume that $q^*$ is of the first type. Then by direct calculation we obtain
$$\Psi_{\cup,*}\circ q^*=\left\{\begin{array}{ll}
j_2^{W_1^*}\circ q_1^* & \mbox{over }i_2^{X_1}(X_1\setminus Y), \\
j_2^{W_1^*}\circ\tilde{h}^*\circ(j_2^{W_2})^*\circ((j_2^{W_2})^*)^{-1}\circ(\tilde{h}^*)^{-1}\circ q_1^* & \mbox{over }i_2^{X_1}(Y),
\end{array}\right.$$ that is, $\Psi_{\cup,*}\circ q^*=j_2^{W_1^*}\circ q_1^*$ over the whole of $i_2^{X_1}(X_1)$, which corresponds to a plot $s^*$ of the second type, for $s_1^*:=q_1^*$.

Similarly, if $q^*$ has its second possible form, we obtain
$$\Psi_{\cup,*}\circ q^*=\left\{\begin{array}{ll}
j_2^{W_1^*}\circ\tilde{h}^*\circ q_2^* & \mbox{over }i_2^{X_1}(Y)\\
j_1^{W_2^*}\circ q_2^* & \mbox{over }i_1^{X_2}(X_2\setminus h(Y)),
\end{array}\right.$$ that is, the first possible form of a plot $s^*$, with $s_2^*:=q_2^*$.

Finally, the smoothness of $(\Psi_{\cup,*})^{-1}$, whose formula
$$(\Psi_{\cup,*})^{-1}=\left\{\begin{array}{ll}
((j_1^{W_1})^*)^{-1}\circ(j_2^{W_1^*})^{-1} & \mbox{on }(\chi_2^*\cup_{(\tilde{h},h)}\chi_1^*)^{-1}(i_2^{X_1}(X_1\setminus Y)) \\
((j_2^{W_2})^*)^{-1}\circ(\tilde{h}^*)^{-1}\circ(j_2^{W_1^*})^{-1} & \mbox{on }(\chi_2^*\cup_{(\tilde{h},h)}\chi_1^*)^{-1}(i_2^{X_1}(Y) \\
((j_2^{W_2})^*)^{-1}\circ(j_1^{W_2^*})^{-1} & \mbox{on }(\chi_2^*\cup_{(\tilde{h},h)}\chi_1^*)^{-1}(i_1^{X_2}(X_2\setminus h(Y))
\end{array}\right.$$ is given by the inverses of the three maps that $\Psi_{\cup,*}$ itself, is established in a completely analogous fashion.
\end{proof}

\subsection{Gluing-dual commutativity conditions for $(\tilde{f},f)$ and $(\tilde{f}^*,f^{-1})$}

We now consider the gluing-dual commutativity condition for $V_2^*$, $V_1^*$, and $(\tilde{f}^*,f^{-1})$, under the assumption that such condition holds for $V_1$, $V_2$, and $(\tilde{f},f)$. For the
duals, this condition takes form of the existence of a diffeomorphism
$$\Phi_{\cup,*}^{(*)}:(V_2^*\cup_{\tilde{f}^*}V_1^*)^*\to V_1^{**}\cup_{\tilde{f}^{**}}V_2^{**}$$ covering the inverse of the switch map 
$\varphi_{X_1\leftrightarrow X_2}$ and satisfying the following:
$$\left\{\begin{array}{ll}
\Phi_{\cup,*}^{(*)}\circ((j_2^{V_2^*})^*)^{-1}=j_1^{V_1^{**}} & \mbox{on }(\pi_1^{**}\cup_{(\tilde{f}^{**},f)}\pi_2^{**})^{-1}(i_1^{X_1}(X_1\setminus Y)),\\
\Phi_{\cup,*}^{(*)}\circ((j_1^{V_1^*})^*)^{-1}=j_2^{V_2^{**}}\circ\tilde{f}^* & \mbox{on }(\pi_1^{**}\cup_{(\tilde{f}^{**},f)}\pi_2^{**})^{-1}(i_2^{X_2}(f(Y))),\\
\Phi_{\cup,*}^{(*)}\circ((j_2^{V_1^*})^*)^{-1}=j_2^{V_2^{**}} & \mbox{on }(\pi_1^{**}\cup_{(\tilde{f}^{**},f)}\pi_2^{**})^{-1}(i_2^{X_2}(X_2\setminus f(Y))).\end{array}\right.$$ 
Notice that this formula can serve as a definition of a certain map $\Phi_{\cup,*}^{(*)}$ between the domain and the range indicated; what we really need to do is to show that
it is a diffeomorphism.

\begin{thm}
Let $\pi_1:V_1\to X_1$ and $\pi_2:V_2\to X_2$ be two locally trivial finite-dimensional diffeological vector pseudo-bundles, let $(\tilde{f},f)$ be a pair of smooth maps that defines a gluing of
the former pseudo-bundle to the latter, and let $\Phi_{\cup,*}:(V_1\cup_{\tilde{f}}V_2)^*\to V_2^*\cup_{\tilde{f}^*}V_1^*$ be the canonical gluing-dual commutativity diffeomorphism. 
Let $g_1$ and $g_2$ be pseudo-metrics on $V_1$ and $V_2$ respectively, compatible with respect to the gluing. Then there exists a diffeomorphism
$$\Phi_{\cup,*}^{(*)}:(V_2^*\cup_{\tilde{f}^*}V_1^*)^*\to V_1^{**}\cup_{\tilde{f}^{**}}V_2^{**}$$ covering the map
$(\varphi_{X_1\leftrightarrow X_2})^{-1}:X_2\cup_{f^{-1}}X_1\to X_1\cup_f X_2$.
\end{thm}

The claim of the theorem could be restated by saying that the dual pseudo-bundles $V_2^*$ and $V_1^*$ satisfy the gluing-dual commutativity condition for the gluing pair $(\tilde{f}^*,f^{-1})$.

\begin{proof}
Recall that the gluing-dual commutativity condition for the initial pseudo-bundles, that is, for $V_1$ and $V_2$, implies that $\tilde{f}^*$ is a diffeomorphism of its domain with its image. It
is then a direct consequence of Theorem \ref{gluing:diffeo:implies:gluing-dual:commute:thm} that the following map
$$\Phi_{\cup,*}^{(*)}=\left\{\begin{array}{ll}
j_2^{V_2^{**}}\circ\left(j_1^{V_2^*}\right)^* & \mbox{on }\left((\pi_2^*\cup_{(\tilde{f}^*,f^{-1})}\pi_1^*)^*\right)^{-1}(i_1^{X_2}(X_2\setminus f(Y))),\\
j_2^{V_2^{**}}\circ\tilde{f}^*\circ\left(j_2^{V_1^*}\right)^* & \mbox{on }\left((\pi_2^*\cup_{(\tilde{f}^*,f^{-1})}\pi_1^*)^*\right)^{-1}(i_2^{X_1}(Y)),\\
j_1^{V_1^{**}}\circ\left(j_2^{V_1^*}\right)^* & \mbox{on }\left((\pi_2^*\cup_{(\tilde{f}^*,f^{-1})}\pi_1^*)^*\right)^{-1}(i_2^{X_1}(X_1\setminus Y)) 
\end{array}\right.$$ is the desired gluing-dual commutativity diffeomorphism.
\end{proof}

\section{Compatibility of $g_2^*$ and $g_1^*$ implies the gluing-dual commutativity condition for $V_1$ and $V_2$}

So far we have spoken of the gluing-dual commutativity condition as a prerequisite to obtaining a canonical construction of the induced pseudo-metric on the pseudo-bundle obtained by gluing. In principle,
it is not a necessary condition (a pseudo-metric on $V_1\cup_{\tilde{f}}V_2$ can be constructed directly out of compatible pseudo-metrics on $V_1$ and $V_2$, using the flexibility
of diffeology in piecing together smooth maps); however, if we want at the same time to consider the dual pseudo-bundles $V_2^*$ and $V_1^*$, and to ensure that the induced pseudo-metrics on them are
again compatible, the reasoning involved starts to come rather close to the gluing-dual commutativity. Indeed, in this section we show that there is essentially an equivalence between the compatibility of 
$g_2^*$ with $g_1^*$, and the existence of a gluing-dual commutativity diffeomorphism $\Phi_{\cup,*}$.

\subsection{From $\Phi_{\cup,*}$ to the compatibility of $g_2^*$, $g_1^*$, and $(\tilde{f}^*,f^{-1})$}

It is not difficult to show that if we assume the gluing-dual commutativity, and more precisely, the existence of the specific diffeomorphism $\Phi_{\cup,*}:(V_1\cup_{\tilde{f}}V_2)^*\to
V_2^*\cup_{\tilde{f}^*}V_1^*$ given by
$$\Phi_{\cup,*}=\left\{\begin{array}{ll}
j_2^{V_1^*}\circ(j_1^{V_1})^* & \mbox{on }((\pi_1\cup_{(\tilde{f},f)}\pi_2)^*)^{-1}(i_1^{X_1}(X_1\setminus Y)) \\
j_2^{V_1^*}\circ\tilde{f}^*\circ(j_2^{V_2})^* & \mbox{on }((\pi_1\cup_{(\tilde{f},f)}\pi_2)^*)^{-1}(i_2^{X_2}(f(Y)))\\
j_1^{V_2^*}\circ(j_2^{V_2})^* & \mbox{on }((\pi_1\cup_{(\tilde{f},f)}\pi_2)^*)^{-1}(i_2^{X_2}(X_2\setminus f(Y)))
\end{array}\right.$$ then the dual pseudo-metrics, if they exist, are compatible.

\begin{thm}\label{gluing-dual:commute:dual:pseudometrics:compatible:thm}
Let $\pi_1:V_1\to X_1$ and $\pi_2:V_2\to X_2$ be two finite-dimensional locally trivial diffeological vector pseudo-bundles, let $(\tilde{f},f)$ be a pair of smooth maps
defining a gluing of $V_1$ to $V_2$, and let $g_1$ and $g_2$ be pseudo-metrics on $V_1$ and $V_2$ respectively, compatible with this gluing; assume that $V_1$, $V_2$, and $(\tilde{f},f)$ satisfy the
gluing-dual commutativity condition. Then $g_2^*$ and $g_1^*$ are compatible with the gluing of $V_2^*$ to $V_1^*$ along the pair $(\tilde{f}^*,f^{-1})$.
\end{thm}

\begin{proof}
By Theorem \ref{when:dual:pseudometrics:compatible:bundles:thm} it suffices to show that $\tilde{f}^*$ is a diffeomorphism of its domain with its image, and this is a trivial consequence of the form
in which we stated the gluing-dual commutativity condition, namely, as the smoothness of the specific diffeomorphism $\Phi_{\cup,*}$. Indeed, denoting for brevity
$Z_2^*:=((\pi_1\cup_{(\tilde{f},f)}\pi_2)^*)^{-1}(i_2^{X_2}(f(Y)))$, we immediately obtain
$$\tilde{f}^*=(j_2^{V_1^*})^{-1}\circ\Phi_{\cup,*}|_{Z_2^*}\circ((j_2^{V_2})^*)^{-1};$$ thus, $\tilde{f}^*$ is a composition of three diffeomorphisms, so a diffeomorphism itself.
\end{proof}

\subsection{The \emph{vice versa}: the smoothness of $\Phi_{\cup,*}$ out of compatibility of $g_2^*$ and $g_1^*$}

The reverse implication, that is, obtaining a smooth $\Phi_{\cup,*}$ assuming the compatibility of $g_2^*$ and $g_1^*$, is now easily obtained by applying Theorem
\ref{gluing:diffeo:implies:gluing-dual:commute:thm}, and the criteria for compatibility of the dual pseudo-metrics. Do note that the application of Theorem
\ref{gluing:diffeo:implies:gluing-dual:commute:thm} is not straightforward; indeed, it speaks of gluing along a diffeomorphism, and the assumptions formulated in terms of various pseudo-metrics do
not provide for $\tilde{f}$ being one. Therefore we need some preliminary considerations.

\subsubsection{Characteristic sub-bundles and the respective dual pseudo-bundles}

In order to obtain our desired conclusion, namely, that the compatibility of pseudo-metrics dual to a given pair of compatible pseudo-metrics provides for the gluing-dual commutativity, we need
several preliminary statements. We collect them in this section; they are more or less of independent interest.

\paragraph{The pseudo-metric $\tilde{g}$ on the pseudo-bundle $W_1\cup_{\tilde{h}}W_2$} Assuming that $W_1$ and $W_2$ admit compatible pseudo-metrics $g_1$ and $g_2$ allows (without any
further assumptions on $\tilde{h}$) for a direct construction of a pseudo-metric $\tilde{g}$ on $W_1\cup_{\tilde{h}}W_2$, which fibrewise coincides with either $g_1$ or $g_2$, as appropriate. This
pseudo-metric is defined by the following formula:
$$\tilde{g}(x)=\left\{\begin{array}{ll}
g_1((i_1^{X_1})^{-1}(x))\circ((j_1^{W_1})^{-1}\otimes(j_1^{W_1})^{-1}) & \mbox{for }x\in i_1^{X_1}(X_1\setminus Y) \\
g_2((i_2^{X_2})^{-1}(x))\circ((j_2^{W_2})^{-1}\otimes(j_2^{W_2})^{-1}) & \mbox{for }x\in i_2^{X_2}(X_2).
\end{array}\right.$$

\paragraph{The switch map for characteristic sub-bundles} As we have already commented, a pseudo-metric on a pseudo-bundle is essentially determined by its behavior on the characteristic sub-bundle; and
furthermore, assuming the local triviality and the existence of a pseudo-metric, the characteristic sub-bundle is diffeomorphic to the dual pseudo-bundle. Thus, we can expect significant correlations
between these three notions; and in this paragraph we specify some of them, as needed to reach the final aim of this section.

\begin{lemma}\label{switch:characteristic:is:diffeo:lem}
Let $\chi_1:W_1\to X_1$ and $\chi_2:W_2\to X_2$ be two locally trivial diffeological vector pseudo-bundles, let $W_1^0$ and $W_2^0$ be their characteristic sub-bundles, let $h:X_1\supseteq Y\to X_2$
be a smooth invertible map with smooth inverse, and let $\tilde{h}$ be its smooth fibrewise linear lift such that its restriction $\tilde{h}_0$ on $\mbox{Domain}(\tilde{h})\cap W_1^0$ is a
diffeomorphism. Then there is a canonical diffeomorphism
$$\varphi_{W_1^0\leftrightarrow W_2^0}:W_1^0\cup_{\tilde{h}_0}W_2^0\to W_2^0\cup_{\tilde{h}_0^{-1}}W_1^0$$ covering the switch map $\varphi_{X_1\leftrightarrow X_2}$.
\end{lemma}

\begin{proof}
The diffeomorphism in question is in fact the same concept as the switch map (which we implied is a diffeomorphism, but did not prove that). Indeed, we denote the map obtained by analogy with
$\varphi_{X_1\leftrightarrow X_2}$, by $\varphi_{W_1^0\leftrightarrow W_2^0}$ and define it to be
$$\varphi_{W_1^0\leftrightarrow W_2^0}=\left\{\begin{array}{ll}
j_2^{W_1^0}\circ(j_1^{W_1^0})^{-1} & \mbox{on }(\chi_1^0\cup_{(\tilde{h}_0,h)}\chi_2^0)^{-1}(i_1^{X_1}(X_1\setminus Y)) \\
j_2^{W_1^0}\circ(\tilde{h}_0)^{-1}\circ(j_2^{W_2^0})^{-1} & \mbox{on }(\chi_1^0\cup_{(\tilde{h}_0,h)}\chi_2^0)^{-1}(i_2^{X_2}(h(Y))) \\
j_1^{W_2^0}\circ(j_2^{W_2^0})^{-1} & \mbox{on }(\chi_1^0\cup_{(\tilde{h}_0,h)}\chi_2^0)^{-1}(i_2^{X_2}(X_2\setminus f(Y))),
\end{array}\right.$$ where by $\chi_i^0$ we denote the restriction of $\chi_i$ to $W_i^0$. Notice that it is its own inverse; let us show that it is smooth.

Let $p:U\to W_1^0\cup_{\tilde{h}_0}W_2^0$; it suffices to consider the case when $U$ is connected. Under such assumption, $p$ lifts to either a plot $p_1$ of $W_1^0$ or to a plot $p_2$ of $W_2^0$,
therefore $p$ itself either has form
$$p=\left\{\begin{array}{ll} j_1^{W_1^0}\circ p_1 & \mbox{on }p_1^{-1}(W_1^0\setminus\mbox{Domain}(\tilde{h}_0)) \\
j_2^{W_2^0}\circ\tilde{h}_0\circ p_1 & \mbox{on }p_1^{-1}(\mbox{Domain}(\tilde{h}_0))
\end{array}\right.$$ in the former case, or it has form $p=j_2^{W_2^0}\circ p_2$ in the latter case. Accordingly, by direct calculation we obtain that $\varphi_{W_1^0\leftrightarrow W_2^0}\circ
p=j_2^{W_2^0}\circ p_1$ for $p$ that lifts to $p_1$ and
$$\varphi_{W_1^0\leftrightarrow W_2^0}\circ p=\left\{\begin{array}{ll}
j_2^{W_1^0}\circ(\tilde{h}_0)^{-1}\circ p_2 & \mbox{on }p_2^{-1}(\mbox{Range}(\tilde{h}_0)) \\
j_1^{W_2^0}\circ p_2 & \mbox{on }p_2^{-1}(W_2^0\setminus\mbox{Range}(\tilde{h}_0))
\end{array}\right.$$ for $p$ that lifts for $p_2$. Clearly, in both cases the result is a plot of $W_2^0\cup_{(\tilde{h}_0)^{-1}}W_1^0$, whence the conclusion.
\end{proof}

The lemma just proven is a preliminary statement which will be needed to establish a link between the following two statements; all three put together will allows us to relate the gluing-dual commutativity
to compatibility of (dual) pseudo-metrics.

\paragraph{The triple $W_1^0\cup_{\tilde{h}_0}W_2^0\cong(W_1\cup_{\tilde{h}}W_2)^0\cong(W_1\cup_{\tilde{h}}W_2)^*$} The facts that we prove here ensure a kind of total commutativity
between $()^0$ (characteristic) and $()^*$ (dual); this phrase is of course very informal, what we really mean shall be clear from the two statements that follow.

\begin{prop}\label{dual:characteristic:gluing:commute:prop}
Let $\chi_1:W_1\to X_1$ and $\chi_2:W_2\to X_2$ be two locally trivial diffeological vector pseudo-bundles, let $W_1^0$ and $W_2^0$ be their characteristic sub-bundles, let $h:X_1\supseteq Y\to X_2$
be a smooth invertible map with smooth inverse, and let $\tilde{h}$ be its smooth fibrewise linear lift. Let $g_1$ and $g_2$ be pseudo-metrics on $W_1$ and $W_2$ respectively, compatible with the
gluing along $(\tilde{h},h)$, and assume that the dual pseudo-metrics $g_2^*$ and $g_1^*$ are compatible with $(\tilde{h}^*,h^{-1})$. Let $\tilde{h}_0$ be the restriction of
$\tilde{h}$ on $\mbox{Domain}(\tilde{h})\cap W_1^0$. Then:
\begin{enumerate}
\item $\tilde{h}_0$ is a diffeomorphism with values in $W_2^0$;
\item There is a pseudo-bundle diffeomorphism
$$\Phi^0:W_1^0\cup_{\tilde{h}_0}W_2^0\to(W_1\cup_{\tilde{h}}W_2)^0$$ covering the identity map $X_1\cup_f X_2\to X_1\cup_f X_2$;
\item The natural pairing map $\Psi_{\tilde{g}}^0:(W_1\cup_{\tilde{h}}W_2)^0\to(W_1\cup_{\tilde{h}}W_2)^*$ associated to the pseudo-metric $\tilde{g}$ on $W_1\cup_{\tilde{h}}W_2$ is a diffeomorphism of the
characteristic sub-bundle $(W_1\cup_{\tilde{h}}W_2)^0$ with the dual pseudo-bundle $(W_1\cup_{\tilde{h}}W_2)^*$.
\end{enumerate}
\end{prop}

\begin{proof}
1. The compatibility of $g_1$ and $g_2$ means that their restrictions $g_1(y)$ and $g_2(h(y))$ on each fibre in the domain of definition and, respectively, the range of $\tilde{h}$ are
compatible pseudo-metrics on the diffeological vector spaces $\chi_1^{-1}(y)$ and $\chi_2^{-1}(h(y))$. By Theorem \ref{criterio:exist:pseudometrics:vspaces:thm} this means that the
restriction $\tilde{h}_0$ of $\tilde{h}$ to $\mbox{Domain}(\tilde{h})\cap W_1^0$ is a smooth injection, whose range is contained in $\mbox{Range}(\tilde{h})\cap W_2^0$.
Furthermore, by compatibility of $g_2^*$ and $g_1^*$ and Theorem \ref{crit:dual:pseudometrics:comp:vspaces:thm}, the dual spaces $(\chi_2^{-1}(h(y)))^*$ and $(\chi_1^{-1}(y))^*$ have the same
dimension, which is equal to the dimension of the corresponding characteristic subspaces; therefore $\tilde{h}_0$ is also surjective. Finally, that its inverse is smooth, follows from local
triviality and the fact that its restriction onto each fibre is a smooth linear map between finite-dimensional vector spaces.

2. This is a direct consequence of the definition of a characteristic sub-bundle. The diffeomorphism $\Phi^0:W_1^0\cup_{\tilde{h}_0}W_2^0\to(W_1\cup_{\tilde{h}}W_2)^0$
is defined by
$$\Phi^0=\left\{\begin{array}{ll} j_1^{W_1}\circ(j_1^{W_1^0})^{-1} & \mbox{over }i_1^{X_1}(X_1\setminus Y)\\
j_2^{W_2}\circ(j_2^{W_2^0})^{-1} & \mbox{over }i_2^{X_2}(X_2);
\end{array}\right.$$ it is essentially the natural inclusion map.

3. The diffeomorphism in question is the pairing map
$$\Psi_{\tilde{g}}^0:(W_1\cup_{\tilde{h}}W_2)^0\to(W_1\cup_{\tilde{h}}W_2)^*$$ restricted to the characteristic sub-bundle and defined in the usual way, \emph{i.e.}, by
$$\mbox{if }w\in(W_1\cup_{\tilde{h}}W_2)^0\Rightarrow\Psi_{\tilde{g}}^0(w)(\cdot)=\tilde{g}((\chi_1\cup_{(\tilde{h},h)}\chi_2)(w))(w,\cdot).$$ That it is smooth, follows from smoothness 
of $\tilde{g}$; that it is bijective, is easily obtained by examining its restriction on each fibre, where, since the fibres of characteristic sub-bundle have standard diffeology, it becomes the usual
isomorphism-by-duality on some standard $\matR^n$. Finally, the smoothness of its inverse follows from the assumption of local triviality.
\end{proof}

Under the same assumptions as those of Proposition \ref{dual:characteristic:gluing:commute:prop}, we also have the following:

\begin{prop}\label{dual:characteristic:gluing:commute:second:part:prop}
There is a canonical diffeomorphism
$$\Psi_{g_2}^0\cup_{(\tilde{h}_0^{-1},\tilde{h}^*)}\Psi_{g_1}^0:W_2^0\cup_{\tilde{h}_0^{-1}}W_1^0\to W_2^*\cup_{\tilde{h}^*}W_1^*,$$ whose restrictions onto the factors of gluing coincide with the
natural pairing maps associated to a pair of compatible pseudo-metrics $g_2$ and $g_1$.
\end{prop}

\begin{proof}
It follows from the assumptions, specifically, the assumption of compatibility of $g_2^*$ and $g_1^*$, that $\tilde{h}^*$ is a diffeomorphism; by Proposition
\ref{dual:characteristic:gluing:commute:prop} the $\tilde{h}_0^{-1}$ is also a diffeomorphism. Thus, the desired diffeomorphism of the two pseudo-bundles in the statement is the result
$\Psi_{g_2}^0\cup_{(\tilde{h}_0^{-1},\tilde{h}^*)}\Psi_{g_1}^0$ of a gluing (see Section 1.3.1 for definitions) of the natural pairing maps $\Psi_{g_2}^0:W_2^0\to W_2^*$ and $\Psi_{g_1}^0:W_1^0\to W_1^*$. 
Since the gluing of two diffeomorphisms along a pair of diffeomorphisms yields again a diffeomorphism, we only need to check that $\Psi_{g_2}^0$ and $\Psi_{g_1}^0$ are
$(\tilde{h}_0^{-1},\tilde{h}^0)$-compatible, that is, that for any $w_2\in\mbox{Domain}(\tilde{h}_0^{-1})$ we have
$$\tilde{h}^*(\Psi_{g_2}^0(w_2))=\Psi_{g_1}^0(\tilde{h}_0^{-1}(w_2)).$$ Let us verify why this is true.

The left-hand side of the expression is by definition
$$\tilde{h}^*(\Psi_{g_2}^0(w_2))(\cdot)=\tilde{h}^*(g_2(\chi_2(w_2))(w_2,\cdot))=g_2(\chi_2(w_2))(w_2,\tilde{h}(\cdot))=g_1(h^{-1}(\chi_2(w_2)))(\tilde{h}_0^{-1}(w_2),\cdot),$$
where the last equality follows from the compatibility of pseudo-metrics $g_1$ and $g_2$. The right-hand side of the expression is
$$\Psi_{g_1}^0(\tilde{h}_0^{-1}(w_2))(\cdot)=g_1(\chi_1(\tilde{h}_0^{-1}(w_2)))(\tilde{h}_0^{-1}(w_2),\cdot),$$ and it remains to observe that
$h^{-1}(\chi_2(w_2))=\chi_1(\tilde{h}_0^{-1}(w_2))$, simply because $\tilde{h}_0$ is a lift of $h$, and the statement is proven.
\end{proof}

\paragraph{The map $\Psi_{g_2}^0\cup_{(\tilde{h}_0^{-1},\tilde{h}^*)}\Psi_{g_1}^0$} We conclude this section by giving the precise formula for the map
$\Psi_{g_2}^0\cup_{(\tilde{h}_0^{-1},\tilde{h}^*)}\Psi_{g_1}^0$, which we will need in the next section. As follows from the general construction of gluing of two smooth maps, we have
$$\left(\Psi_{g_2}^0\cup_{(\tilde{h}_0^{-1},\tilde{h}^*)}\Psi_{g_1}^0\right)(w^0)=\left\{\begin{array}{ll}
j_1^{W_2^*}\circ\Psi_{g_2}^0\circ(j_1^{W_2^0})^{-1} & \mbox{on }(\chi_2^0\cup_{(\tilde{h}_0^{-1},h^{-1})}\chi_1^0)^{-1}(i_1^{X_2}(X_2\setminus h(Y)))\\
j_2^{W_1^*}\circ\Psi_{g_1}^0\circ(j_2^{W_1^0})^{-1} & \mbox{on }(\chi_2^0\cup_{(\tilde{h}_0^{-1},h^{-1})}\chi_1^0)^{-1}(i_2^{X_1}(X_1)).
\end{array}\right.$$

\subsubsection{Proving the gluing-dual commutativity}

We now give our final statement, which is a sufficient condition (and, together with Theorem \ref{gluing-dual:commute:dual:pseudometrics:compatible:thm}, a
criterion) for the gluing-dual commutativity condition to be satisfied.

\begin{thm}
Let $\pi_1:V_1\to X_1$ and $\pi_2:V_2\to X_2$ be diffeological vector pseudo-bundles, locally trivial and with finite-dimensional fibres, and let $(\tilde{f},f)$ be a gluing of $V_1$ to $V_2$, with
an invertible $f$. Suppose that $V_1$ and $V_2$ admit pseudo-metrics compatible with this gluing, and let $g_1$ and $g_2$ be a fixed choice of such pseudo-metrics. Assume, finally, that the induced
pseudo-metrics $g_2^*$ and $g_1^*$ on the dual pseudo-bundles $V_2^*$ and $V_1^*$ are compatible with the gluing along $(\tilde{f}^*,f^{-1})$. Then $V_1$, $V_2$, and $(\tilde{f},f)$
satisfy the gluing-dual commutativity condition.
\end{thm}

\begin{proof}
Let us first show that $(V_1\cup_{\tilde{f}}V_2)^*$ and $V_2^*\cup_{\tilde{f}^*}V_1^*$ are diffeomorphic, and then explain why there is a canonical diffeomorphism between them. Applying
Proposition \ref{dual:characteristic:gluing:commute:prop}(3) and then (2), we obtain
$$(V_1\cup_{\tilde{f}}V_2)^*\cong(V_1\cup_{\tilde{f}}V_2)^0\cong V_1^0\cup_{\tilde{f}_0}V_2^0;$$ next, by Lemma \ref{switch:characteristic:is:diffeo:lem} and Proposition
\ref{dual:characteristic:gluing:commute:second:part:prop} we obtain 
$$V_1^0\cup_{\tilde{f}_0}V_2^0\cong V_2^0\cup_{\tilde{f}_0^{-1}}V_1^0\cong V_2^*\cup_{\tilde{f}^*}V_1^*,$$ as wanted. It remains to see that the diffeomorphisms involved produce in the end the canonical
commutativity diffeomorphism $\Phi_{\cup,*}$.

Let us now specify the final diffeomorphism that we obtain from the above sequence. Let $\tilde{g}$ be the pseudo-metric on $V_1\cup_{\tilde{f}}V_2$ induced by $g_1$ and $g_2$. Then the first
diffeomorphism of the sequence is $(\Psi_{\tilde{g}}^0)^{-1}$; the second is $(\Phi^0)^{-1}$, the inverse of the diffeomorphism described in the proof of Proposition
\ref{dual:characteristic:gluing:commute:prop}, the third is the switch-like map $\varphi_{V_1^0\leftrightarrow V_2^0}$, and the fourth is
$\Psi_{g_2}^0\cup_{(\tilde{f}_0^{-1},\tilde{f}^*)}\Psi_{g_1}^0$. We need to see that the composition
$$\Phi=\left(\Psi_{g_2}^0\cup_{(\tilde{f}_0^{-1},\tilde{f}^*)}\Psi_{g_1}^0\right)\circ\left(\varphi_{V_1^0\leftrightarrow V_2^0}\right)\circ(\Phi^0)^{-1}\circ(\Psi_{\tilde{g}}^0)^{-1}$$
coincides with the appropriate canonically defined map $\Phi_{\cup,*}$.

Indeed, after some obvious cancelations the pointwise description of the diffeomorphism $\Phi$ is as follows:
$$\Phi(v)=\left\{\begin{array}{ll}
\left(j_2^{V_1^*}\circ\Psi_{g_1}^0\circ(j_1^{V_1})^{-1}\circ(\Psi_{\tilde{g}}^0)^{-1}\right)(v) & \mbox{if }(\pi_1\cup_{(\tilde{f},f)}\pi_2)^*(v)\in i_1^{X_1}(X_1\setminus Y)\\
\left(j_2^{V_1^*}\circ\Psi_{g_1}^0\circ\tilde{f}_0^{-1}\circ(j_2^{V_2})^{-1}\circ(\Psi_{\tilde{g}}^0)^{-1}\right)(v) & \mbox{if }(\pi_1\cup_{(\tilde{f},f)}\pi_2)^*(v)\in i_2^{X_2}(f(Y))\\
\left(j_1^{V_2^*}\circ\Psi_{g_2}^0\circ(j_2^{V_2})^{-1}\circ(\Psi_{\tilde{g}}^0)^{-1}\right)(v) & \mbox{if }(\pi_1\cup_{(\tilde{f},f)}\pi_2)^*(v)\in i_2^{X_2}(X_2\setminus f(Y)).
\end{array}\right.$$ Let us consider the three cases indicated.

Let first $v$ be such that $(\pi_1\cup_{(\tilde{f},f)}\pi_2)^*(v)\in i_1^{X_1}(X_1\setminus Y)$; then we have
\begin{flushleft}
$\Phi(v)=\left(j_2^{V_1^*}\circ\Psi_{g_1}^0\circ(j_1^{V_1})^{-1}\right)\left((\Psi_{\tilde{g}}^0)^{-1}(v)\right),\mbox{ where } (\Psi_{\tilde{g}}^0)^{-1}(v)=v^0\in
j_1^{V_1}(\pi_1^{-1}(X_1\setminus Y))\cap(V_1\cup_{\tilde{f}}V_2)^0$
\end{flushleft}
\begin{flushright}
such that $v(\cdot)=g_1(x)((j_1^{V_1})^{-1}(v^0),(j_1^{V_1})^{-1}(\cdot))$ for $x=\pi_1((j_1^{V_1})^{-1}(v^0))$,
\end{flushright} where $(\cdot)$ stands for the argument of $v\in(V_1\cup_{\tilde{f}}V_2)^*$, that is being taken in $j_1^{V_1}(V_1\setminus\pi_1^{-1}(Y))$. Hence we obtain
$$\Phi(v)(\cdot)=j_2^{V_1^*}\left(g_1(x)((j_1^{V_1})^{-1}(v^0),\cdot)\right).$$ This time $(\cdot)$ stands for an element of $V_1^0$; it is related to the argument of $v(\cdot)$ by the map $j_1^{V_1^0}$, so we have in
the end
$$\Phi(v)(\cdot)=j_2^{V_1^*}\left(v(j_1^{V_1^0}(\cdot))\right)\Rightarrow \Phi(v)=\left(j_2^{V_1^*}\circ(j_1^{V_1})^*\right)(v)\Rightarrow\Phi=j_2^{V_1^*}\circ(j_1^{V_1})^*,$$ \emph{i.e.}, the
canonical form of the gluing-dual commutativity diffeomorphism.

The other two cases are rather similar. Let $v\in(V_1\cup_{\tilde{f}}V_2)^*$ be such that $(\pi_1\cup_{(\tilde{f},f)}\pi_2)^*(v)\in i_2^{X_2}(f(Y))$; we then have
\begin{flushleft}
$\Phi(v)=\left(j_2^{V_1^*}\circ\Psi_{g_1}^0\circ\tilde{f}_0^{-1}\circ(j_2^{V_2})^{-1}\right)\left((\Psi_{\tilde{g}}^0)^{-1}(v)\right)$, where
$(\Psi_{\tilde{g}}^0)^{-1}(v)=v^0\in j_2^{V_2}\left(\pi_2^{-1}(f(Y))\right)\cap(V_1\cup_{\tilde{f}}V_2)^0$
\end{flushleft}
\begin{flushright}
such that $v(\cdot)=g_2(x)((j_2^{V_2})^{-1}(v^0),(j_2^{V_2})^{-1}(\cdot))$ for $x=\pi_2((j_2^{V_2})^{-1}(v^0))$.
\end{flushright}
From this, we obtain
$$\Phi(v)=\left(j_2^{V_1^*}\circ\Psi_{g_1}^0\circ\tilde{f}_0^{-1}\circ(j_2^{V_2})^{-1}\right)(v^0)(\cdot)=j_2^{V_1^*}\left(g_1(f^{-1}(x))((\tilde{f}_0^{-1}\circ(j_2^{V_2})^{-1})(v^0),\cdot)\right).$$
Once again, we should relate this to the expression for $v(\cdot)$, this time keeping in mind the compatibility of the pseudo-metrics $g_1$ and $g_2$. It suffices to recall that the argument $(\cdot)$
in this case is being taken in $V_1^0$ and is related to that of $v(\cdot)$ by the map $j_2^{V_2}\circ\tilde{f}_0$. Therefore we can rewrite the expression for $\Phi(v)(\cdot)$ as follows:
$$\Phi(v)(\cdot)=j_2^{V_1^*}\left(g_1(f^{-1}(x))((\tilde{f}_0^{-1}\circ(j_2^{V_2})^{-1})(v^0),\cdot)\right)=(j_2^{V_1^*}\circ\tilde{f}_0^*)\left(g_2(x)((j_2^{V_2})^{-1}(v^0),\tilde{f}_0(\cdot))\right),$$ 
which then allows us to conclude that
$$\Phi(v)(\cdot)=\left(j_2^{V_1^*}\circ\tilde{f}_0^*\circ(j_2^{V_2})^*\right)(v)(\cdot)\Rightarrow\Phi=j_2^{V_1^*}\circ\tilde{f}_0^*\circ(j_2^{V_2})^*.$$

It remains to consider the third part of the definition of $\Phi$. Let $v\in(V_1\cup_{\tilde{f}}V_2)^*$ be such that $(\pi_1\cup_{(\tilde{f},f)}\pi_2)^*(v)\in i_2^{X_2}(X_2\setminus f(Y))$; we then have
\begin{flushleft}
$\Phi(v)=\left(j_1^{V_2^*}\circ\Psi_{g_2}^0\circ(j_2^{V_2})^{-1}\right)\left((\Psi_{\tilde{g}}^0)^{-1}(v)\right)$, where $(\Psi_{\tilde{g}}^0)^{-1}(v)=v^0\in
j_2^{V_2}(\pi_2^{-1}(X_2\setminus f(Y)))\cap(V_1\cup_{\tilde{f}}V_2)^0$
\end{flushleft}
\begin{flushright}
such that $v(\cdot)=g_2(x)((j_2^{V_2})^{-1}(v^0),(j_2^{V_2})^{-1}(\cdot))$ for $x=\pi_2((j_2^{V_2})^{-1}(v^0))$.
\end{flushright}
We therefore have
$$\Phi(v)(\cdot)=\left(j_1^{V_2^*}\circ\Psi_{g_2}^0\circ(j_2^{V_2})^{-1}\right)(v^0)(\cdot)=j_1^{V_2^*}\left(g_2(x)((j_2^{V_2})^{-1}(v^0),\cdot)\right).$$ By the same considerations regarding the argument 
$(\cdot)$, which in this case is related to that of $v(\cdot)$ by the map $j_2^{V_2}$, we obtain
$$\Phi(v)(\cdot)=j_1^{V_2^*}\left(v(j_2^{V_2^0}(\cdot))\right)\Rightarrow\Phi=j_1^{V_2^*}\circ(j_2^{V_2})^*,$$ therefore $\Phi$ has the canonical form also in the third case, whence the final claim.
\end{proof}

\subsection{Final remarks on the gluing-commutativity condition}

To conclude the discussion on the gluing-dual commutativity, we stress that the crucial point\footnote{Under the assumption that the pseudo-bundles involved are finite-dimensional and locally trivial.}
throughout was the dual map $\tilde{f}^*$ being a diffeomorphism (of its domain with its image). As follows from our proofs, it is this condition that is equivalent to both
\begin{itemize}
\item the \emph{specific} map $\Phi_{\cup,*}$ being a diffeomorphism; and
\item the dual pseudo-metrics $g_2^*$ and $g_1^*$ being compatible.
\end{itemize}
It remains to observe that this also justifies our choice to state right away the gluing-dual commutativity condition in terms of $\Phi_{\cup,*}$ being smooth, rather than just asking for the
existence of \emph{some} diffeomorphism between $(V_1\cup_{\tilde{f}}V_2)^*$ and $V_2^*\cup_{\tilde{f}^*}V_1^*$: the two are equivalent.

\section{The pseudo-metrics $\tilde{g}^*$ and $\widetilde{g^*}$ on pseudo-bundles $(V_1\cup_{\tilde{f}}V_2)^*$ and $V_2^*\cup_{\tilde{f}^*}V_1^*$}

Assuming that $V_1$, $V_2$, and $(\tilde{f},f)$ satisfy the gluing-dual commutativity condition implies, in particular, that there are two ways to construct a pseudo-metric on the pseudo-bundle
$(V_1\cup_{\tilde{f}}V_2)^*\cong V_2^*\cup_{\tilde{f}^*}V_1^*$, which correspond, respectively, to the left-hand side and the right-hand side of this expression. Specifically, the (specific
expression for the) pseudo-bundle on the left carries the pseudo-metric $\tilde{g}^*$ that is induced by, or dual to, the pseudo-metric $\tilde{g}$. The pseudo-bundle on the right is
obtained from a given gluing of two pseudo-bundles carrying compatible pseudo-metrics; it therefore carries a pseudo-metric $\widetilde{g^*}$ that corresponds to this gluing (we mentioned this
in Section 1; the details can be found in \cite{pseudometric-pseudobundle}, and we recall what we need immediately below). We show that this is actually the same
pseudo-metric.

\subsection{The pseudo-metric $\tilde{g}^*$ on $(V_1\cup_{\tilde{f}}V_2)^*$}

This is the pseudo-metric dual to\footnote{It would be more precise to say, induced by duality.} the pseudo-metric $\tilde{g}$ on the pseudo-bundle $V_1\cup_{\tilde{f}}V_2$; the latter is defined as the
composition
$$\tilde{g}=\left(\Phi_{\cup,*}^{-1}\otimes\Phi_{\cup,*}^{-1}\right)\circ\Phi_{\otimes,\cup}\circ(g_2\cup_{(f^{-1},\tilde{f}^*\otimes\tilde{f}^*)}g_1)\circ\varphi_{X_1\leftrightarrow X_2},$$ where
$\varphi_{X_1\leftrightarrow X_2}$ is the switch map, and
$$\Phi_{\cup,*}:(V_1\cup_{\tilde{f}}V_2)^*\to V_2^*\cup_{\tilde{f}^*}V_1^*\,\,\mbox{ and }\,\, \Phi_{\otimes,\cup}:(V_2^*\otimes V_2^*)
\cup_{\tilde{f}^*\otimes\tilde{f}^*}(V_1^*\otimes V_1^*)\to(V_2^*\cup_{\tilde{f}^*}V_1^*)\otimes(V_2^*\cup_{\tilde{f}^*}V_1^*)$$ are the appropriate versions of the commutativity diffeomorphisms
(see Section 4.2.1 for the explicit formula).

The pseudo-metric $\tilde{g}^*$ is then defined as the pseudo-metric dual to $\tilde{g}$, which by the usual definition means that, if $\Psi_{\tilde{g}}:V_1\cup_{\tilde{f}}V_2\to(V_1\cup_{\tilde{f}}V_2)^*$
is the (already-seen) pairing map relative to $\tilde{g}$, that is, $\Psi_{\tilde{g}}(v)=\tilde{g}((\pi_1\cup_{(\tilde{f},f)}\pi_2)(v))(v,\cdot)$, then for any $x\in X_1\cup_f X_2$ and any
$v^*,w^*\in((\pi_1\cup_{(\tilde{f},f)}\pi_2)^*)^{-1}(x)$ we have by definition
$$\tilde{g}^*(x)(v^*,w^*):=\tilde{g}(x)(v,w),$$ where $v,w\in(\pi_1\cup_{(\tilde{f},f)}\pi_2)^{-1}(x)$ are any two elements such that $\Psi_{\tilde{g}}(v)=v^*$ and
$\Psi_{\tilde{g}}(w)=w^*$. We can avail ourselves of the already-mentioned restriction $\Psi_{\tilde{g}}^0$ of $\Psi_{\tilde{g}}$ to the characteristic sub-bundle
$(V_1\cup_{\tilde{f}}V_2)^0$ and thus define
$$\tilde{g}^*(x)(v^*,w^*):=\tilde{g}(x)\left((\Psi_{\tilde{g}}^0)^{-1}(v^*),(\Psi_{\tilde{g}}^0)^{-1}(w^*)\right).$$ Finally, an even more explicit formula for $\tilde{g}^*$ is
$$\tilde{g}^*(x)(v^*,w^*)=\left\{\begin{array}{ll}
g_1((i_1^{X_1})^{-1}(x))(v_1,w_1), & \mbox{if }x\in\mbox{Range}(i_1^{X_1}),\\
g_2((i_2^{X_2})^{-1}(x))(v_2,w_2), & \mbox{if }x\in\mbox{Range}(i_2^{X_2}),
\end{array}\right.$$ where we have by definition
$$\left\{\begin{array}{ll}
v_1=\left((j_1^{V_1})^{-1}\circ(\Psi_{\tilde{g}}^0)^{-1}\right)(v^*), & w_1=\left((j_1^{V_1})^{-1}\circ(\Psi_{\tilde{g}}^0)^{-1}\right)(w^*)\\
v_2=\left((j_2^{V_2})^{-1}\circ(\Psi_{\tilde{g}}^0)^{-1}\right)(v^*), & w_2=\left((j_2^{V_2})^{-1}\circ(\Psi_{\tilde{g}}^0)^{-1}\right)(w^*).
\end{array}\right.$$

\subsection{The pseudo-metric $\widetilde{g^*}$ on $V_2^*\cup_{\tilde{f}^*}V_1^*$}

The pseudo-metric $\widetilde{g^*}$ is defined on the pseudo-bundle $V_2^*\cup_{\tilde{f}^*}V_1^*$ fibrewise, by imposing it to coincide with $g_2^*$ or $g_1^*$ on the appropriate subsets. Specifically,
for $i=1,2$ let, as before, $\Psi_{g_i}:V_i\to V_i^*$ be the pairing map associated to $g_1$ and $g_2$ respectively; then for any given $x\in X_2\cup_{f^{-1}}X_1$ and any two
$v^*,w^*\in(\pi_2^*\cup_{(\tilde{f}^*,f^{-1})}\pi_1^*)^{-1}(x)$ we define
$$\widetilde{g^*}(x)(v^*,w^*)=\left\{\begin{array}{ll}
g_2^*((i_1^{X_2})^{-1}(x))((j_1^{V_2^*})^{-1}(v^*),(j_1^{V_2^*})^{-1}(w^*))=g_2((i_1^{X_2})^{-1}(x))(v_2,w_2) & \mbox{if }x\in\mbox{Range}(i_1^{X_2}), \\
g_1^*((i_2^{X_1})^{-1}(x))((j_2^{V_1^*})^{-1}(v^*),(j_2^{V_1^*})^{-1}(w^*))=g_1((i_2^{X_1})^{-1}(x))(v_1,w_1) & \mbox{if }x\in\mbox{Range}(i_2^{X_1}),
\end{array}\right.$$ where again we make reference to the characteristic sub-bundles and the corresponding invertible restrictions $\Psi_{g_2}^0$ of $\Psi_{g_2}$ and $\Psi_{g_1}^0$ of
$\Psi_{g_1}$, since by construction
$$\left\{\begin{array}{ll}
v_2=\left((\Psi_{g_2}^0)^{-1}\circ(j_1^{V_2^*})^{-1}\right)(v^*), & w_2=\left((\Psi_{g_2}^0)^{-1}\circ(j_1^{V_2^*})^{-1}\right)(w^*),\\
v_1=\left((\Psi_{g_1}^0)^{-1}\circ(j_2^{V_1^*})^{-1}\right)(v^*), & w_1=\left((\Psi_{g_1}^0)^{-1}\circ(j_2^{V_1^*})^{-1}\right)(w^*).
\end{array}\right.$$

\subsection{Comparing $\tilde{g}^*$ and $\widetilde{g^*}$}

In the case we are considering, we have assumed\footnote{As is, in fact, necessary for the two pseudo-metrics just described to be well-defined.} that $V_1$, $V_2$, and $(\tilde{f},f)$ satisfy the
gluing-dual commutativity condition. By the very definition of the latter, this means that the pseudo-bundles $(V_1\cup_{\tilde{f}}V_2)^*$ and $V_2^*\cup_{\tilde{f}^*}V_1^*$ are
diffeomorphic, and in a canonical way. Since both of these pseudo-bundles also carry the canonical pseudo-metrics, described in the two sections above, it is natural to ask next whether their
canonical identification, via the gluing-dual commutativity diffeomorphism $\Phi_{\cup,*}$, is an isometry relative to these pseudo-metrics. In this section we prove that it is.

More precisely, since $\tilde{g}^*$ and $\widetilde{g^*}$ are maps 
$$\tilde{g}^*:X_1\cup_f X_2\to(V_1\cup_{\tilde{f}}V_2)^{**}\otimes(V_1\cup_{\tilde{f}}V_2)^{**}\,\,\mbox{ and }\,\,\widetilde{g^*}:X_2\cup_{f^{-1}}X_1\to(V_2^*\cup_{\tilde{f}^*}V_1^*)^*\otimes
(V_2^*\cup_{\tilde{f}^*}V_1^*)^*,$$ we show that by adding the diffeomorphism
$$(\Phi_{\cup,*}^{-1})^*\otimes(\Phi_{\cup,*}^{-1})^*:(V_1\cup_{\tilde{f}}V_2)^{**}\otimes(V_1\cup_{\tilde{f}}V_2)^{**}\to(V_2^*\cup_{\tilde{f}^*}V_1^*)^*\otimes(V_2^*\cup_{\tilde{f}^*}V_1^*)^*$$ 
and the switch map $\varphi_{X_1\leftrightarrow X_2}:X_1\cup_f X_2\to X_2\cup_{f^{-1}}X_2$, we obtain
$$(\Phi_{\cup,*}^{-1})^*\otimes(\Phi_{\cup,*}^{-1})^*\circ\tilde{g}^*=\widetilde{g^*}\circ(\varphi_{X_1\leftrightarrow X_2}).$$ The full statement is as follows.

\begin{thm}
Let $\pi_1:V_1\to X_1$ and $\pi_2:V_2\to X_2$ be two finite-dimensional locally trivial diffeological vector pseudo-bundle, and let $(\tilde{f},f)$ be a gluing between them such
that $f$ and $\tilde{f}^*$ are diffeomorphisms. Assume that there exist pseudo-metrics $g_1$ and $g_2$ on $V_1$ and $V_2$ respectively, that are compatible with the gluing along
$(\tilde{f},f)$. Then the following is true:
$$\left((\Phi_{\cup,*}^{-1})^*\otimes(\Phi_{\cup,*}^{-1})^*\right)\circ\tilde{g}^*=\widetilde{g^*}\circ(\varphi_{X_1\leftrightarrow X_2}).$$
\end{thm}

Notice that the assumptions of the theorem provide for the existence of all the maps that appear in the claim, that is, for the existence of the smooth switch map $\varphi_{X_1\leftrightarrow X_2}$, that of
the gluing-dual commutativity diffeomorphism $\Phi_{\cup,*}$, and the existence and compatibility of the dual pseudo-metrics $g_2^*$ and $g_1^*$, through which the pseudo-metrics $\tilde{g}^*$ and
$\widetilde{g^*}$ are defined.

\begin{proof}
The actual meaning of the formula that we wish to prove is as follows: taken an arbitrary $x\in X_1\cup_f X_2$ and arbitrary $v^*,w^*\in(\pi_2^*\cup_{(\tilde{f}^*,f^{-1})}\pi_1^*)^{-1}(\varphi_{X_1\leftrightarrow
X_2}(x))\in V_2^*\cup_{\tilde{f}^*}V_1^*$, we should have
$$\widetilde{g^*}(\varphi_{X_1\leftrightarrow X_2}(x))(v^*,w^*)=\tilde{g}^*(x)(\Phi_{\cup,*}^{-1}(v^*),\Phi_{\cup,*}^{-1}(w^*)).$$ Since this formula involves the switch map and a gluing-dual
commutativity diffeomorphism, both of which are defined separately in three cases, we should check the desired equality in the same three cases as well. These cases are: $x\in i_1^{X_1}(X_1\setminus Y)$, 
$x\in i_2^{X_2}(f(Y))$, and $x\in i_2^{X_2}(X_2\setminus f(Y))$.

Let $x\in i_1^{X_1}(X_1\setminus Y)$; then $v^*,w^*\in\mbox{Range}(j_2^{V_1^*})$ and $\varphi_{X_1\leftrightarrow X_2}(x)=i_2^{X_1}((i_1^{X_1})^{-1}(x))$. Notice that the
corresponding $v$ and $w$ that appear in the definition of the pseudo-metric $\widetilde{g^*}$ are then
$$v=\left((\Psi_{g_1}^0)^{-1}\circ(j_2^{V_1^*})^{-1}\right)(v^*),\,\,\,w=\left((\Psi_{g_1}^0)^{-1}\circ(j_2^{V_1^*})^{-1}\right)(w^*),$$ therefore we have
\begin{flushleft}
$\widetilde{g^*}(\varphi_{X_1\leftrightarrow X_2}(x))(v^*,w^*)=\widetilde{g^*}(i_2^{X_1}((i_1^{X_1})^{-1}(x)))(v^*,w^*)=$
\end{flushleft}
\begin{flushright}
$=g_1((i_1^{X_1})^{-1}(x))\left(((\Psi_{g_1}^0)^{-1}\circ(j_2^{V_1^*})^{-1})(v^*),((\Psi_{g_1}^0)^{-1}\circ(j_2^{V_1^*})^{-1})(w^*)\right)$.
\end{flushright}

On the other hand, when $\tilde{g}^*$ is applied to two elements $v_1^*,w_1^*$ in the fibre of $(V_1\cup_{\tilde{f}}V_2)^*$ over a point in $i_1^{X_1}(X_1\setminus Y)$, its value is
$g_1((i_1^{X_1})^{-1}(x))(v_1,w_1)$, where
$$v_1=\left((j_1^{V_1})^{-1}\circ(\Psi_{\tilde{g}}^0)^{-1}\right)(v_1^*),\,\,\,w_1=\left((j_1^{V_1})^{-1}\circ(\Psi_{\tilde{g}}^0)^{-1}\right)(w_1^*);$$ in our case we will have $v_1^*:=\Phi_{\cup,*}^{-1}(v^*)$ and
$w_1^*:=\Phi_{\cup,*}^{-1}(w^*)$. Observe now that
$$\Phi_{\cup,*}^{-1}(v^*)=\left(((j_1^{V_1})^*)^{-1}\circ(j_2^{V_1^*})^{-1}\right)(v^*)\,\,\,\mbox{ and }\,\,\,\Phi_{\cup,*}^{-1}(w^*)=\left(((j_1^{V_1})^*)^{-1}\circ(j_2^{V_1^*})^{-1}\right)(w^*),$$
and that in the case we are considering the relation between $(\Psi_{\tilde{g}}^0)^{-1}$ and $(\Psi_{g_1}^0)^{-1}$ is as follows:
$$(\Psi_{\tilde{g}}^0)^{-1}=j_1^{V_1}\circ(\Psi_{g_1}^0)^{-1}\circ(j_1^{V_1})^*.$$ Therefore we have
$$v_1=\left((j_1^{V_1})^{-1}\circ(\Psi_{\tilde{g}}^0)^{-1}\right)\left(\Phi_{\cup,*}^{-1}(v^*)\right)=\left((\Psi_{g_1}^0)^{-1}\circ(j_2^{V_1^*})^{-1}\right)(v^*)$$
$$w_1=\left((j_1^{V_1})^{-1}\circ(\Psi_{\tilde{g}}^0)^{-1}\right)\left(\Phi_{\cup,*}^{-1}(w^*)\right)=\left((\Psi_{g_1}^0)^{-1}\circ(j_2^{V_1^*})^{-1}\right)(w^*),$$
hence
\begin{flushleft}
$\tilde{g}^*(x)(\Phi_{\cup,*}^{-1}(v^*),\Phi_{\cup,*}^{-1}(w^*))=$
\end{flushleft}
$$=g_1\left((i_1^{X_1})^{-1}(x)\right)\left(((\Psi_{g_1}^0)^{-1}\circ(j_2^{V_1^*})^{-1})(v^*),((\Psi_{g_1}^0)^{-1}\circ(j_2^{V_1^*})^{-1})(w^*)\right)=$$
\begin{flushright}
$=\widetilde{g^*}(\varphi_{X_1\leftrightarrow X_2}(x))(v^*,w^*)$,
\end{flushright} as wanted.

Turning to the second case, let $x\in i_2^{X_2}(f(Y))$, so that $v^*,w^*\in\mbox{Range}(j_2^{V_1^*})$ and $\varphi_{X_1\leftrightarrow X_2}(x)=(i_2^{X_1}\circ f^{-1}\circ(i_2^{X_2})^{-1})(x)$. 
To calculate $\widetilde{g^*}(\varphi_{X_1\leftrightarrow X_2}(x))(v^*,w^*)$, write
$$v=\left((\Psi_{g_1}^0)^{-1}\circ(j_2^{V_1^*})^{-1}\right)(v^*),\,\,\,w=\left((\Psi_{g_1}^0)^{-1}\circ(j_2^{V_1^*})^{-1}\right)(w^*);$$ this implies that
\begin{flushleft}
$\widetilde{g^*}(\varphi_{X_1\leftrightarrow X_2}(x))(v^*,w^*)=$
\end{flushleft}
$$=g_1\left((f^{-1}\circ(i_2^{X_2})^{-1})(x)\right)\left(((\Psi_{g_1}^0)^{-1}\circ(j_2^{V_1^*})^{-1})(v^*),((\Psi_{g_1}^0)^{-1}\circ(j_2^{V_1^*})^{-1})(w^*)\right)=$$
\begin{flushright}
$=g_2((i_2^{X_2})^{-1})(x))\left((\tilde{f}\circ(\Psi_{g_1}^0)^{-1}\circ(j_2^{V_1^*})^{-1})(v^*),(\tilde{f}\circ(\Psi_{g_1}^0)^{-1}\circ(j_2^{V_1^*})^{-1})(w^*)\right)$,
\end{flushright}
where we have used the compatibility of the pseudo-metrics $g_1$ and $g_2$ and the fact that $v$ and $w$ belong to the characteristic sub-bundle, on which $\tilde{f}$ is invertible by assumption.

To calculate now the second part of the identity to verify, recall that by definition
$$\tilde{g}^*(x)(v_2^*,w_2^*)=g_2((i_2^{X_2})^{-1}(x))(v_2,w_2),$$ where
$$v_2:=\left((j_2^{V_2})^{-1}\circ(\Psi_{\tilde{g}}^0)^{-1}\right)(v_2^*)\,\,\,\mbox{ and }w_2:=\left((j_2^{V_2})^{-1}\circ(\Psi_{\tilde{g}}^0)^{-1}\right)(w_2^*),$$
with $v_2^*$, $w_2^*$ denoting for brevity
$$v_2^*:=\Phi_{\cup,*}^{-1}(v^*),\,\,\,w_2^*:=\Phi_{\cup,*}^{-1}(w^*).$$ We now recall that in the case we are considering,
$$\Phi_{\cup,*}^{-1}(v^*)=\left(((j_1^{V_1})^*)^{-1}\circ\tilde{f}^*\circ(j_2^{V_2^*})^{-1}\right)(v^*)\,\,\,\mbox{ and }\,\,\,
\Phi_{\cup,*}^{-1}(w^*)=\left(((j_1^{V_1})^*)^{-1}\circ\tilde{f}^*\circ(j_2^{V_2^*})^{-1}\right)(w^*),$$ and there are the following relations between $\Psi_{\tilde{g}}^0$,
$\Psi_{g_1}^0$, and $\Psi_{g_2}^0$ (which we state immediately for their inverses):
$$(\Psi_{\tilde{g}}^0)^{-1}=j_2^{V_2}\circ(\Psi_{g_2}^0)^{-1}\circ(\tilde{f}^*)^{-1}\circ(j_1^{V_1})^*\,\,\,\mbox{ and }\,\,\,
(\Psi_{g_2}^0)^{-1}=\tilde{f}\circ(\Psi_{g_1}^0)^{-1}.$$ Thus, by direct calculation
$$v_2=\left((\Psi_{g_2}^0)^{-1}\circ(j_2^{V_2^*})^{-1}\right)(v^*)\,\,\,\mbox{ and }\,\,\,w_2=\left((\Psi_{g_2}^0)^{-1}\circ(j_2^{V_2^*})^{-1}\right)(w^*),$$
which means that
$$\tilde{g}^*(x)(v_2^*,w_2^*)=\widetilde{g^*}(\varphi_{X_1\leftrightarrow X_2}(x))(v^*,w^*),$$ again as wanted.

Finally, let us consider the third case, that of $x\in i_2^{X_2}(X_2\setminus f(Y))$; then $v^*,w^*\in\mbox{Range}(j_1^{V_2^*})$ and
$\varphi_{X_1\leftrightarrow X_2}(x)=\left(i_1^{X_2}\circ(i_2^{X_2})^{-1}\right)(x)$. Furthermore,
$$\widetilde{g^*}(\varphi_{X_1\leftrightarrow X_2}(x))(v^*,w^*)=g_2((i_2^{X_2})^{-1}(x))(v,w),$$ where
$$v=\left((\Psi_{g_2}^0)^{-1}\circ(j_1^{V_2^*})^{-1}\right)(v^*)\,\,\,\mbox{ and }\,\,\,w=\left((\Psi_{g_2}^0)^{-1}\circ(j_1^{V_2^*})^{-1}\right)(w^*).$$

On the other side, we have
$$\tilde{g}^*(x)(v_2^*,w_2^*)=g_2((i_2^{X_2})^{-1}(x))(v_2,w_2),$$ where
$$v_2=\left((j_2^{V_2})^{-1}\circ(\Psi_{\tilde{g}}^0)^{-1}\right)(v_2^*)\,\,\,\mbox{ and }\,\,\,w_2=\left((j_2^{V_2})^{-1}\circ(\Psi_{\tilde{g}}^0)^{-1}\right)(w_2^*),$$
and in turn
$$v_2^*=(\Phi_{\cup,*})^{-1}(v^*)=\left(((j_2^{V_2})^*)^{-1}\circ(j_1^{V_2^*})^{-1}\right)(v^*),\,\,\,w_2^*=(\Phi_{\cup,*})^{-1}(v^*)=\left(((j_2^{V_2})^*)^{-1}\circ(j_1^{V_2^*})^{-1}\right)(w^*).$$
Finally, we observe that there is the following relation between $\Psi_{\tilde{g}}^0$ and $\Psi_{g_2}^0$ (stated again for their inverses):
$$(\Psi_{\tilde{g}}^0)^{-1}=j_2^{V_2}\circ(\Psi_{g_2}^0)^{-1}\circ(j_2^{V_2})^*.$$ Thus, putting together consecutively the expressions for $v_2$, $v_2^*$, and $(\Psi_{\tilde{g}}^0)^{-1}$ (and likewise, for $w_2$,
$w_2^*$, and $(\Psi_{\tilde{g}}^0)^{-1}$), we obtain
$$v_2=\left((\Psi_{g_2}^0)^{-1}\circ(j_1^{V_2^*})^{-1}\right)(v^*)=v\,\,\,\mbox{ and }\,\,\,w_2=\left((\Psi_{g_2}^0)^{-1}\circ(j_1^{V_2^*})^{-1}\right)(w^*)=w,$$ which implies that
$$\tilde{g}^*(x)(v_2^*,w_2^*)=\widetilde{g^*}(\varphi_{X_1\leftrightarrow X_2}(x))(v^*,w^*),$$ and therefore concludes our consideration of the third case. All cases having thus been exhausted, the proof is
finished.
\end{proof}

We can re-phrase our main conclusion by stating the following.

\begin{cor}\label{gluing:dual:commut:diffeo:is:isometry:cor}
The gluing-dual commutativity diffeomorphism $\Phi_{\cup,*}$ is a pseudo-bundle isometry between $\left((V_1\cup_{\tilde{f}}V_2)^*,\tilde{g}^*\right)$ and
$\left(V_2^*\cup_{\tilde{f}^*}V_1^*,\widetilde{g^*}\right)$.
\end{cor}

\section{The covariant Clifford algebras}

We choose this umbrella term to refer to the pseudo-bundles of Clifford algebras that are associated to whatever pseudo-bundles we obtain by interposing the operations of gluing and taking the dual
pseudo-bundle. These are, first of all, the pseudo-bundles of Clifford algebras associated to $(V_1\cup_{\tilde{f}}V_2)^*$ and $V_2^*\cup_{\tilde{f}^*}V_1^*$; as we have seen above, these
pseudo-bundles are \emph{a priori} different, but they are naturally identified under appropriate assumptions. Once these assumptions are imposed, we still have another \emph{a priori} difference, that of
the two natural pseudo-metrics $\tilde{g}^*$ and $\widetilde{g^*}$ on them being different, the possibility treated in the preceding section. Finally, there is a third option, that of the result of
gluing of two pseudo-bundles of Clifford algebras, those associated to $V_2^*$ and $V_1^*$. In this section we complete our consideration of these three cases, and of how they are
interrelated.

\subsection{The diffeomorphism $\cl((V_1\cup_{\tilde{f}}V_2)^*,\tilde{g}^*)\cong\cl(V_2^*\cup_{\tilde{f}^*}V_1^*,\widetilde{g^*})$}

The diffeomorphism in question is easily obtained from the gluing-dual commutativity diffeomorphism $\Phi_{\cup,*}$, whose existence we assume and which in this case both guarantees that the
pseudo-metrics $\tilde{g}^*$ and $\widetilde{g^*}$ exist and are well-defined, and, by Corollary \ref{gluing-dual:commute:dual:pseudometrics:compatible:thm}, is an
isometry with respect to them. Extending $\Phi_{\cup,*}$ to a diffeomorphism between $\cl((V_1\cup_{\tilde{f}}V_2)^*,\tilde{g}^*)$ and $\cl(V_2^*\cup_{\tilde{f}^*}V_1^*,\widetilde{g^*})$ is then a
standard procedure, whose result we denote by
$$\Phi_{\cup,*}^{\cl}:\cl((V_1\cup_{\tilde{f}}V_2)^*,\tilde{g}^*)\to\cl(V_2^*\cup_{\tilde{f}^*}V_1^*,\widetilde{g^*}).$$ To describe this diffeomorphism, it suffices to recall that for any equivalence class
of form
$$[v_1\otimes\ldots\otimes v_k]\in\left((\pi_1\cup_{(\tilde{f},f)}\pi_2)^{\cl}\right)^{-1}(Y)\subset\cl((V_1\cup_{\tilde{f}}V_2)^*,\tilde{g}^*),$$ with a representative
$v_1\otimes\ldots\otimes v_k$, its image is the equivalence class in $\cl(V_2^*\cup_{\tilde{f}^*}V_1^*,\widetilde{g^*})$ of $$\Phi_{\cup,*}(v_1')\otimes\ldots\otimes\Phi_{\cup,*}(v_k').$$
In other words, $\Phi_{\cup,*}^{\cl}$ is the pushdown, by the quotient projections
$$\pi^{T((V_1\cup_{\tilde{f}}V_2)^*)}:T((V_1\cup_{\tilde{f}}V_2)^*)\to\cl((V_1\cup_{\tilde{f}}V_2)^*,\tilde{g}^*)\mbox{ and }
\pi^{T(V_2^*\cup_{\tilde{f}^*}V_1^*)}:T(V_2^*\cup_{\tilde{f}^*}V_1^*)\to\cl(V_2^*\cup_{\tilde{f}^*}V_1^*,\widetilde{g^*}),$$ of the map
$\bigoplus_n\Phi_{\cup,*}^{\otimes n}$, so that we have
$$\Phi_{\cup,*}^{\cl}\circ\pi^{T((V_1\cup_{\tilde{f}}V_2)^*)}=\pi^{T(V_2^*\cup_{\tilde{f}^*}V_1^*)}\circ\left(\bigoplus_n\Phi_{\cup,*}^{\otimes n}\right).$$

\begin{thm}
The map $\Phi_{\cup,*}^{\cl}:\cl((V_1\cup_{\tilde{f}}V_2)^*,\tilde{g}^*)\to\cl(V_2^*\cup_{\tilde{f}^*}V_1^*,\widetilde{g^*})$ given by the identity
$$\Phi_{\cup,*}^{\cl}\circ\pi^{T((V_1\cup_{\tilde{f}}V_2)^*)}=\pi^{T(V_2^*\cup_{\tilde{f}^*}V_1^*)}\circ\left(\bigoplus_n\Phi_{\cup,*}^{\otimes n}\right)$$ is a well-defined diffeomorphism.
\end{thm}

\begin{proof}
This is a consequence of Corollary \ref{gluing:dual:commut:diffeo:is:isometry:cor}.
\end{proof}

\subsection{The pseudo-bundle $\cl(V_2^*,g_2^*)\cup_{(\tilde{F}^*)^{\cl}}\cl(V_1^*,g_1^*)$}

There is a third pseudo-bundle, with all fibres Clifford algebras, that is naturally associated to our data, the pseudo-bundle $\cl(V_2^*,g_2^*)\cup_{(\tilde{F}^*)^{\cl}}\cl(V_1^*,g_1^*)$. It is
obtained by gluing together the pseudo-bundles of Clifford algebras relative to, respectively, $(V_2^*,g_2^*)$ and $(V_1^*,g_1^*)$, along the map $(\tilde{F}^*)^{\cl}$, induced by the map
$\tilde{f}^*$; this is the construction described in Section 1.5.2. To indicate how this construction works specifically for this case, it suffices to say that, for any given equivalence class in
$((\pi_2^*)^{\cl})^{-1}(f(Y))\subset\cl(V_2^*,g_2^*)$, with an arbitrary representative $v_1^*\otimes\ldots\otimes v_k^*$, its image under $(\tilde{F}^*)^{\cl}$ is the equivalence class in
$\cl(V_1^*,g_1^*)$ of the element $$\tilde{f}^*(v_1^*)\otimes\ldots\otimes\tilde{f}^*(v_k^*).$$

The main point, however, is that the resulting pseudo-bundle $\cl(V_2^*,g_2^*)\cup_{(\tilde{F}^*)^{\cl}}\cl(V_1^*,g_1^*)$ is naturally diffeomorphic to the pseudo-bundle
$\cl(V_2^*\cup_{\tilde{f}^*}V_1^*,\widetilde{g^*})$ via the diffeomorphism
$$\Phi^{\cl(*)}:\cl(V_2^*,g_2^*)\cup_{(\tilde{F}^*)^{\cl}}\cl(V_1^*,g_1^*)\to\cl(V_2^*\cup_{\tilde{f}^*}V_1^*,\widetilde{g^*})$$ that is determined by the following formula:
$$\Phi^{\cl(*)}=\left\{\begin{array}{ll} (j_1^{V_2^*})^{\cl}\circ(j_1^{\cl(V_2^*)})^{-1} & \mbox{on }\left((\pi_2^*)^{\cl}\cup_{((\tilde{F}^*)^{\cl},f^{-1})}(\pi_1^*)^{\cl}\right)^{-1}(i_1^{X_2}(X_2\setminus f(Y))) \\
(j_2^{V_1^*})^{\cl}\circ(j_2^{\cl(V_1^*)})^{-1} & \mbox{on }\left((\pi_2^*)^{\cl}\cup_{((\tilde{F}^*)^{\cl},f^{-1})}(\pi_1^*)^{\cl}\right)^{-1}(i_2^{X_1}(X_1)),
\end{array}\right.$$ where
$$(j_1^{V_2^*})^{\cl}:\cl(V_2^*,g_2^*)\setminus((\pi_2^*)^{Cl})^{-1}(f(Y))\to\cl(V_2^*\cup_{\tilde{f}^*}V_1^*,\widetilde{g^*}) \mbox{ and }
(j_2^{V_1^*})^{\cl}:\cl(V_1^*,g_1^*)\to\cl(V_2^*\cup_{\tilde{f}^*}V_1^*,\widetilde{g^*})$$ are the fibrewise extensions to $\cl(V_2^*,g_2^*)\setminus((\pi_2^*)^{Cl})^{-1}(f(Y))$
and to $\cl(V_1^*,g_1^*)$, respectively, of the two natural inclusions
$$V_2^*\setminus(\pi_2^*)^{-1}(f(Y))\hookrightarrow\cl(V_2^*\cup_{\tilde{f}^*}V_1^*,\widetilde{g^*})\mbox{ and } V_1^*\hookrightarrow\cl(V_2^*\cup_{\tilde{f}^*}V_1^*,\widetilde{g^*}).$$ The latter
inclusions are in turn given by the compositions of either 
$j_1^{V_2^*}:V_2^*\setminus(\pi_2^*)^{-1}(f(Y))\to V_2^*\cup_{\tilde{f}^*}V_1^*$ or $j_2^{V_1^*}:V_1^*\to V_2^*\cup_{\tilde{f}^*}V_1^*$, with the standard inclusion 
$V_2^*\cup_{\tilde{f}^*}V_1^*\to\cl(V_2^*\cup_{\tilde{f}^*}V_1^*,\widetilde{g^*})$.

\begin{thm}
The map
$\Phi^{\cl(*)}:\cl(V_2^*,g_2^*)\cup_{(\tilde{F}^*)^{\cl}}\cl(V_1^*,g_1^*)\to\cl(V_2^*\cup_{\tilde{f}^*}V_1^*,\widetilde{g^*})$ thus defined is a smooth diffeomorphism.
\end{thm}

\begin{proof}
This is a consequence of Theorem 5.5 in \cite{clifford-alg}, applied to $V_2^*$, $V_1^*$, and $\tilde{f}^*$; notice that $\widetilde{g^*}$ is exactly the counterpart of the pseudo-metric
$\tilde{g}$ that appears in the statement of the theorem just cited, in that it is obtained from $g_2^*$ and $g_1^*$ in precisely the same way as $\tilde{g}$ is obtained from $g_1$ and $g_2$.
\end{proof}

\subsection{The three diffeomorphisms $\Phi_{\cup,*}^{\cl}$, $\Phi^{\cl(*)}$, and $\Phi_{\cup}^{\cl(*)}$}

We now summarize the above by listing the three possible forms of the Clifford algebra pseudo-bundle, together with the assumptions that allow us to identify them to each other, and with the
corresponding diffeomorphisms.

\paragraph{The assumptions} As (almost) everywhere throughout the paper, we consider two finite-dimensional diffeological vector pseudo-bundles $\pi_1:V_1\to X_1$ and $\pi_2:V_2\to X_2$, and a
gluing of the former to the latter along the pair $(\tilde{f},f)$. In order for the three diffeomorphisms to exist, we must also assume the following:
\begin{itemize}
\item the two pseudo-bundles are locally trivial;\footnote{This is sufficient but not necessary. What we really need is that the right inverse of the pairing map, that takes values in the characteristic 
sub-bundle be smooth, and so it suffices that this sub-bundle split off as a smooth direct summand.}
\item $f$ and $\tilde{f}^*$ are diffeomorphisms;
\item $V_1$ and $V_2$ admit compatible pseudo-metrics $g_1$ and $g_2$ respectively.
\end{itemize}

\paragraph{The three shapes of the pseudo-bundle of covariant Clifford algebras, and their equivalences} Under the assumptions just listed, the following three pseudo-bundles are well-defined:
$$\cl((V_1\cup_{\tilde{f}}V_2)^*,\tilde{g}^*),\,\,\,\cl(V_2^*\cup_{\tilde{f}^*}V_1^*,\widetilde{g^*}),\,\,\,\cl(V_2^*,g_2^*)\cup_{(\tilde{F}^*)^{\cl}}\cl(V_1^*,g_1^*).$$ Although
\emph{a priori} these could be three different pseudo-bundles, the same assumptions that guarantee that all three are well-defined at the same time, also guarantee that they are in fact equivalent, via
the following diffeomorphisms, already described above:
\begin{itemize}
\item $\Phi_{\cup,*}^{\cl}:\cl((V_1\cup_{\tilde{f}}V_2)^*,\tilde{g}^*)\to\cl(V_2^*\cup_{\tilde{f}^*}V_1^*,\widetilde{g^*})$;
\item $\Phi^{\cl(*)}:\cl(V_2^*,g_2^*)\cup_{(\tilde{F}^*)^{\cl}}\cl(V_1^*,g_1^*)\to\cl(V_2^*\cup_{\tilde{f}^*}V_1^*,\widetilde{g^*})$; and
\item $\left(\Phi^{\cl(*)}\right)^{-1}\circ\Phi_{\cup,*}^{\cl}=:\Phi_{\cup}^{\cl(*)}:\cl((V_1\cup_{\tilde{f}}V_2)^*,\tilde{g}^*)\to\cl(V_2^*,g_2^*)\cup_{(\tilde{F}^*)^{\cl}}\cl(V_1^*,g_1^*)$.
\end{itemize}

\section{The pseudo-bundles of exterior algebras, and gluing}

We now turn to considering the pseudo-bundles of exterior algebras, first in the contravariant case, and then in the covariant case. We recall their construction, which is standard, and concentrate on the
interactions with the operation of gluing.

\subsection{The contravariant case}

We first consider the contravariant version of the exterior algebra, by which we mean the following. Let first $V$ be a diffeological vector space; for each tensor power of $V$ consider the alternating
operator\footnote{In other terms, the antisymmetrization operator.}
$$\mbox{Alt}:\underbrace{V\otimes\ldots\otimes V}_n\to\underbrace{V\otimes\ldots\otimes V}_n,$$ acting, as usual, by
$$\mbox{Alt}(v_1\otimes\ldots\otimes v_n)=\frac{1}{n!}\sum_{\sigma}(-1)^{\mbox{sgn}(\sigma)}v_{\sigma(1)}\otimes\ldots\otimes v_{\sigma(n)}$$ and extended by linearity. In this section the
$n$-th exterior power of $V$ is the image $$\bigwedge_n(V)=\mbox{Alt}(\underbrace{V\otimes\ldots\otimes V}_n);$$ the whole exterior algebra $\bigwedge_*(V)$ is the direct sum of all
$\bigwedge_n(V)$. We obtain the pseudo-bundle $\bigwedge_*(V)$ of exterior algebras associated to a given pseudo-bundle $\pi:V\to X$ by employing the same operations in the pseudo-bundle version, and
defining the alternating operator fibrewise.

\subsubsection{The induced gluing map $\tilde{f}^{\bigwedge_*}$}

This map is provided by the universal factorization property for alternating maps. Specifically, let $\pi_1:V_1\to X_1$ and $\pi_2:V_2\to X_2$ be two finite-dimensional diffeological vector pseudo-bundles, 
and let $(\tilde{f},f)$ be a gluing between them. Then the restriction $\tilde{f}|_{\pi_1^{-1}(y)}$ of $\tilde{f}$ to each fibre in its domain of definition is a smooth linear map between diffeological vector spaces 
$\pi_1^{-1}(y)$ and $\pi_2^{-1}(f(y))$, and the direct sum of all tensor degrees of $\tilde{f}|_{\pi_1^{-1}(y)}$ is again a smooth linear map between the tensor algebras of these spaces:
$$\bigoplus_n\left(\tilde{f}|_{\pi_1^{-1}(y)}\right)^{\otimes n}:T(\pi_1^{-1}(y))\to T(\pi_2^{-1}(f(y))).$$ By construction, this map commutes with the two respective alternating operators, so its
restriction, that we denote by $\left(\tilde{f}|_{\pi_1^{-1}(y)}\right)^{\bigwedge_*}$, to $\bigwedge_*(\pi_1^{-1}(y))$ is a smooth linear map between the exterior algebras of the two fibres:
$$\left(\tilde{f}|_{\pi_1^{-1}(y)}\right)^{\bigwedge_*}:\bigwedge_*(\pi_1^{-1}(y))\to\bigwedge_*(\pi_2^{-1}(f(y))).$$ Finally, the collection
$$\tilde{f}^{\bigwedge_*}:=\bigcup_{y\in Y}\left(\tilde{f}|_{\pi_1^{-1}(y)}\right)^{\bigwedge_*},$$ where $Y$ is the domain of definition of $f$, yields a smooth and
fibrewise linear map $\tilde{f}^{\bigwedge_*}$ between the appropriate subsets of $\bigwedge_*(V_1)$ and $\bigwedge_*(V_2)$. Thus, it yields an induced gluing between the corresponding
pseudo-bundles of contravariant exterior algebras.

\subsubsection{The pseudo-bundles $\bigwedge_*(V_1\cup_{\tilde{f}}V_2)$ and $\bigwedge_*(V_1)\cup_{\tilde{f}^{\bigwedge_*}}\bigwedge_*(V_2)$}

In a similar manner, for the pseudo-bundle $V_1\cup_{\tilde{f}}V_2$ there is its own alternating operator $\mbox{Alt}$, whose image is the pseudo-bundle $\bigwedge_*(V_1\cup_{\tilde{f}}V_2)$. Since each
fibre of $T(V_1\cup_{\tilde{f}}V_2)$ coincides with either a fibre of $T(V_1)$ or one of $T(V_2)$, and fibrewise each of the three alternating operators under consideration (those relative to $V_1$,
$V_2$, and $V_1\cup_{\tilde{f}}V_2$) is the usual one of a diffeological vector space, it makes sense to expect the two pseudo-bundles of exterior algebras,
$\bigwedge_*(V_1\cup_{\tilde{f}}V_2)$ and $\bigwedge_*(V_1)\cup_{\tilde{f}^{\bigwedge_*}}\bigwedge_*(V_2)$, to be diffeomorphic. Indeed, they are, and it is not difficult to
describe the natural diffeomorphism between them; it is based on the gluing-tensor product commutativity diffeomorphism $\Phi_{\cup,\otimes}$, as the next construction shows.

\paragraph{A preliminary remark} At this moment we explicitly impose the assumption that for each of our pseudo-bundles $V_1$ and $V_2$ (and accordingly, for all the results of their gluings, their duals,
and any mixture of such), the set of the dimensions of their fibres has a finite upper limit. We denote it by
$$\dim_V=\sup_{x_1\in X_1,x_2\in X_2}\{\dim(\pi_1^{-1}(x_1)),\dim(\pi_2^{-1}(x_2))\};$$ we do not go into any detail about how this correlates with any other assumptions of ours, just note that we will need 
for one of the diffeomorphisms that we define in the next paragraph (specifically, we use to ensure that $\Phi^{\Lambda_*}$ is indeed onto).

\paragraph{The diffeomorphism $\Phi_{\cup,\otimes}^{(\otimes n)}:\left(V_1\cup_{\tilde{f}}V_2\right)^{\otimes n}\to V_1^{\otimes n}\cup_{\tilde{f}^{\otimes n}}V_2^{\otimes n}$} We first describe
the construction of this diffeomorphism, which is by induction on $n$. The base of the induction is $n=2$, in which case $\Phi_{\cup,\otimes}^{(\otimes n)}=\Phi_{\cup,\otimes}$, the
already-mentioned gluing-tensor product commutativity diffeomorphism. Suppose that $\Phi_{\cup,\otimes}^{(\otimes(n-1))}$ has already been defined. Then $\Phi_{\cup,\otimes}^{(\otimes n)}$
is obtained as the composition
$$\left(V_1\cup_{\tilde{f}}V_2\right)^{\otimes(n-1)}\otimes\left(V_1\cup_{\tilde{f}}V_2\right)\to\left(V_1^{\otimes(n-1)}\cup_{\tilde{f}^{\otimes n}}V_2^{\otimes
(n-1)}\right)\otimes\left(V_1\cup_{\tilde{f}}V_2\right)\to V_1^{\otimes n}\cup_{\tilde{f}^{\otimes n}}V_2^{\otimes n},$$ where the first arrow stands for
$\Phi_{\cup,\otimes}^{(\otimes(n-1))}\otimes\mbox{Id}_{V_1\cup_{\tilde{f}}V_2}$, and the second one, for the version $\Phi_{\cup,\otimes}^{\tilde{f}^{\otimes(n-1)},\tilde{f}}$ of the
gluing-tensor product commutativity diffeomorphism applied to the case of two factors, $V_1^{\otimes(n-1)}\cup_{\tilde{f}^{\otimes n}}V_2^{\otimes (n-1)}$ and $V_1\cup_{\tilde{f}}V_2$. We can
summarize the whole construction as
$$\Phi_{\cup,\otimes}^{(\otimes n)}=\Phi_{\cup,\otimes}^{\tilde{f}^{\otimes(n-1)},\tilde{f}}\circ\left(\Phi_{\cup,\otimes}^{(\otimes(n-1))}\otimes\mbox{Id}_{V_1\cup_{\tilde{f}}V_2}\right).$$
Note also that we will denote the inverse of $\Phi_{\cup,\otimes}^{(\otimes n)}$ by $\Phi_{\otimes,\cup}^{(\otimes n)}$. Finally, for all $k$ (limited in practice by $\dim_V$) we define a diffeomorphism
$$\Phi_{\cup,\oplus}^{(k)}:\bigoplus_{n=0}^k\left(V_1^{\otimes n}\cup_{\tilde{f}^{\otimes n}}V_2^{\otimes n}\right)\to\left(\oplus_n V_1^{\otimes
n}\right)\cup_{\oplus_n\tilde{f}^{\otimes n}}\left(\oplus_n V_2^{\otimes n}\right);$$ the construction is exactly the same, just using the gluing-direct sum commutativity diffeomorphism
$\Phi_{\cup,\oplus}$ in place of $\Phi_{\cup,\otimes}$.

\paragraph{The diffeomorphism $\Phi_{\cup,\otimes}^{(\otimes n)}$, and the alternating operators} We write $\mbox{Alt}^{(n)}$ for the restriction of the alternating operator $\mbox{Alt}$ on
$V_1\cup_{\tilde{f}}V_2$ onto the $n$-th tensor degree; likewise, $\mbox{Alt}_i^{(n)}$ stands for the same restriction of the alternating operator on $V_i$, for $i=1,2$. It is then quite trivial
to observe that we have
$$\mbox{Alt}^{(n)}=\Phi_{\otimes,\cup}^{(\otimes n)}\circ\left(\mbox{Alt}_1^{(n)}\cup_{\left(\tilde{f}^{\otimes
n},\tilde{f}^{\otimes}\right)}\mbox{Alt}_2^{(n)}\right)\circ \Phi_{\cup,\otimes}^{(\otimes n)};$$ equivalently,
$$\Phi_{\cup,\otimes}^{(\otimes n)}\circ\mbox{Alt}^{(n)}=\left(\mbox{Alt}_1^{(n)}\cup_{\left(\tilde{f}^{\otimes
n},\tilde{f}^{\otimes}\right)}\mbox{Alt}_2^{(n)}\right)\circ\Phi_{\cup,\otimes}^{(\otimes n)}.$$

\paragraph{The diffeomorphism $\Phi^{\bigwedge_*}:\bigwedge_*(V_1\cup_{\tilde{f}}V_2)\to \bigwedge_*(V_1)\cup_{\tilde{f}^{\bigwedge_*}}\bigwedge_*(V_2)$} We
can now define the desired diffeomorphism as 
$$\Phi^{\bigwedge_*}=\Phi_{\cup,\oplus}^{(\dim_V)}\circ\bigoplus_n\Phi_{\cup,\otimes}^{(\otimes n)}\mid_{\bigwedge_*(V_1\cup_{\tilde{f}}V_2)}.$$ It remains to
observe that $\Phi^{\bigwedge_*}$ is indeed onto $\bigwedge_*(V_1)\cup_{\tilde{f}^{\bigwedge_*}}\bigwedge_*(V_2)$, as follows from its commutativity (in the sense explained in the
previous paragraph) with the alternating operators.

\begin{thm}
The map
$$\Phi^{\bigwedge_*}=\Phi_{\cup,\oplus}^{(\dim_V)}\circ\bigoplus_n\Phi_{\cup,\otimes}^{(\otimes n)}\mid_{\bigwedge_*(V_1\cup_{\tilde{f}}V_2)}$$ is a pseudo-bundle diffeomorphism
$\bigwedge_*(V_1\cup_{\tilde{f}}V_2)\to\bigwedge_*(V_1)\cup_{\tilde{f}^{\bigwedge_*}}\bigwedge_*(V_2)$ covering the identity map on $X_1\cup_f X_2$.
\end{thm}

\subsection{The pseudo-bundles of covariant exterior algebras}

We now consider the covariant case. The basic definition is simple: the \textbf{covariant exterior algebra} $\bigwedge(V)$ of a pseudo-bundle $V$ is $\bigwedge_*(V^*)$, the contravariant exterior
algebra of its dual pseudo-bundle. So the reason why we consider it separately is to study its behavior with respect to the gluing, which, as we know, is not always well-behaved with respect to duality.

\subsubsection{The induced map between $\bigwedge(V_2)$ and $\bigwedge(V_1)$}

Indeed, let $\pi_1:V_1\to X_1$ and $\pi_2:V_2\to X_2$ be two finite-dimensional diffeological vector pseudo-bundles, and let $(\tilde{f},f)$ be a gluing between them such that $f$ is
invertible. The gluing map between $\bigwedge(V_2)$ and $\bigwedge(V_1)$ is defined exactly as $\tilde{f}^{\bigwedge_*}$, but it is based on the dual map $\tilde{f}^*$. This gluing map is
denoted by $\tilde{f}^{\bigwedge}$ and is in fact $$\tilde{f}^{\bigwedge}:=(\tilde{f}^*)^{\bigwedge_*}.$$

\subsubsection{The pseudo-bundles $\bigwedge(V_1\cup_{\tilde{f}}V_2)$ and $\bigwedge(V_2)\cup_{\tilde{f}^{\bigwedge}}\bigwedge(V_1)$}

As in the contravariant case, there are two natural pseudo-bundles of exterior algebras to consider, namely those mentioned in the title of this section. It is also natural to wonder whether they are
diffeomorphic; we show that indeed they are, under the assumption that the gluing-dual commutativity condition is satisfied, by constructing a certain pseudo-bundle diffeomorphism
$$\Phi^{\bigwedge}:\bigwedge(V_1\cup_{\tilde{f}}V_2)\to\bigwedge(V_2)\cup_{\tilde{f}^{\bigwedge}}\bigwedge(V_1)$$ covering the switch map.

\paragraph{The $n$-th degree component of $\Phi^{\bigwedge}$} Let $\Phi_{\cup,\otimes}^{(\otimes n)}:\left(V_2^*\cup_{\tilde{f}^*}V_1^*\right)^{\otimes
n}\to(V_2^*)^{\otimes n}\cup_{(\tilde{f}^*)^{\otimes n}}(V_1^*)^{\otimes n}$ be the diffeomorphism constructed in the previous section. The $n$-th tensor degree component of
$\Phi^{\bigwedge}$ is the composition 
$$\Phi_{\cup,\otimes}^{(\otimes n)}\circ\Phi_{\cup,*}^{\otimes n}:\left((V_1\cup_{\tilde{f}}V_2)^*\right)^{\otimes n}\to
\left(V_2^*\cup_{\tilde{f}^*}V_1^*\right)^{\otimes n}\to(V_2^*)^{\otimes n}\cup_{(\tilde{f}^*)^{\otimes n}}(V_1^*)^{\otimes n}.$$ Notice that if
$$\mbox{Alt}_{\cup,*}:\left((V_1\cup_{\tilde{f}}V_2)^*\right)^{\otimes n}\to\left((V_1\cup_{\tilde{f}}V_2)^*\right)^{\otimes n},\,\,\,
\mbox{Alt}_2:(V_2^*)^{\otimes n}\to(V_2^*)^{\otimes n}\mbox{ and }\mbox{Alt}_1:(V_1^*)^{\otimes n}\to(V_1^*)^{\otimes n}$$ are the $n$-th degree alternating operators on $(V_1\cup_{\tilde{f}}V_2)^*$,
$V_2^*$, and $V_1^*$ respectively, then we have 
$$\left(\Phi_{\cup,\otimes}^{(\otimes n)}\circ\Phi_{\cup,*}^{\otimes n}\right)\circ\mbox{Alt}_{\cup,*}=
\left(\mbox{Alt}_2\cup_{\left((\tilde{f}^*)^{\otimes n},(\tilde{f}^*)^{\otimes n}\right)}\mbox{Alt}_1\right)\circ\left(\Phi_{\cup,\otimes}^{(\otimes n)}\circ\Phi_{\cup,*}^{\otimes n}\right).$$

\paragraph{The diffeomorphism $\Phi^{\bigwedge}$} We now employ also the gluing-direct sum commutativity diffeomorphism $\Phi_{\cup,\oplus}^{(dim_{V^*})}$, also from the previous section,
to obtain $\Phi^{\bigwedge}$. Indeed, we define
$$\Phi^{\bigwedge}=\Phi_{\cup,\oplus}^{(\dim_{V^*})}\circ\bigoplus_{n=0}^{(\dim_{V^*})}\left(\Phi_{\cup,\otimes}^{(\otimes n)}
\circ\Phi_{\cup,*}^{\otimes n}\right)\mid_{\bigwedge(V_1\cup_{\tilde{f}}V_2)}.$$ This is a well-defined injective, smooth and fibrewise linear map on $\bigwedge(V_1\cup_{\tilde{f}}V_2)$; that its image is
$\bigwedge(V_2)\cup_{\tilde{f}^{\bigwedge}}\bigwedge(V_1)$, follows from the commutativity between each $\Phi_{\cup,\otimes}^{(\otimes n)}\circ\Phi_{\cup,*}^{\otimes n}$ and the appropriate alternating
operators (see above).

\begin{thm}
The map
$$\Phi^{\bigwedge}=\Phi_{\cup,\oplus}^{(\dim_{V^*})}\circ\bigoplus_{n=0}^{(\dim_{V^*})}\left(\Phi_{\cup,\otimes}^{(\otimes
n)}\circ\Phi_{\cup,*}^{\otimes n}\right)\mid_{\bigwedge(V_1\cup_{\tilde{f}}V_2)}$$ is a pseudo-bundle diffeomorphism
$\bigwedge(V_1\cup_{\tilde{f}}V_2)\to\bigwedge(V_2)\cup_{\tilde{f}^{\bigwedge}}\bigwedge(V_1)$ covering the switch map $\varphi_{X_1\leftrightarrow X_2}$.
\end{thm}

\section{The Clifford actions}

In this section we consider all possible (shapes of) Clifford actions, first outlining what acts on what, and then establishing the various natural equivalences.

\subsection{The outline}

As we have seen in the preceding sections, there is a multitude of formally distinct, but (as we are about to see) equivalent with respect to the diffeomorphisms described in the previous two
sections, Clifford actions relative to a given gluing of $(V_1,g_1)$ and $(V_2,g_2)$. In this section we give a list of these actions and their equivalences, with proofs and details appearing in the two
sections immediately following. As before, we assume that $f$ and $\tilde{f}^*$ are diffeomorphisms, and $g_1$ and $g_2$ are compatible.

\subsubsection{The contravariant case}

Recall that in this case we have two natural Clifford algebras,\footnote{A note on slight change in terminology: in the remainder of the paper we will just say \emph{Clifford algebra}
instead of a \emph{pseudo-bundle of Clifford algebras}, and \emph{exterior algebra} instead of \emph{pseudo-bundle of exterior algebras}; in the present context this is unlikely to cause
confusion.} specifically
$$\cl(V_1\cup_{\tilde{f}}V_2,\tilde{g})\cong\cl(V_1,g_1)\cup_{\tilde{F}^{\cl}}\cl(V_2,g_2)$$ (recall that the diffeomorphism that we have between them is
$\Phi^{\cl}:\cl(V_1,g_1)\cup_{\tilde{F}^{\cl}}\cl(V_2,g_2)\to\cl(V_1\cup_{\tilde{f}}V_2,\tilde{g})$) and two natural exterior algebras,
$$\bigwedge_*(V_1\cup_{\tilde{f}}V_2)\cong\bigwedge_*(V_1)\cup_{\tilde{f}^{\bigwedge_*}}\bigwedge_*(V_2).$$

\paragraph{Summary of actions} The natural actions are, the standard action $c_*$ of $\cl(V_1\cup_{\tilde{f}}V_2,\tilde{g})$ on $\bigwedge_*(V_1\cup_{\tilde{f}}V_2)$, and the composite action
$\tilde{c}_*:=c_1\cup_{(\tilde{F}^{\cl},\tilde{f}^{\bigwedge_*})}c_2$ of $\cl(V_1,g_1)\cup_{\tilde{F}^{\cl}}\cl(V_2,g_2)$ on $\bigwedge_*(V_1)\cup_{\tilde{f}^{\bigwedge_*}}\bigwedge_*(V_2)$,
where $c_i$ is the standard action of $\cl(V_i,g_i)$ on $\bigwedge_*(V_i)$. We will show that this is a partial case of the construction considered in \cite{clifford-alg}.

\paragraph{The equivalence of the two actions} This is expressed by the formula:
$$\Phi^{\bigwedge_*}(c_*(v)(e))=\tilde{c}_*\left((\Phi^{\cl})^{-1}(v)\right)\left(\Phi^{\bigwedge_*}(e)\right)$$ for all $v\in\cl(V_1\cup_{\tilde{f}}V_2,\tilde{g})$ and
$e\in\bigwedge_*(V_1\cup_{\tilde{f}}V_2)$ such that $\pi^{\cl}(v)=\pi^{\bigwedge_*}(e)$. Below we will explain why this relation does hold.

\subsubsection{The covariant case}

There are three Clifford algebras to consider:
$$\cl((V_1\cup_{\tilde{f}}V_2)^*,\tilde{g}^*)\cong\cl(V_2^*\cup_{\tilde{f}^*}V_1^*,\widetilde{g^*})\cong\cl(V_2^*,g_2^*)\cup_{(\tilde{F}^*)^{\cl}}\cl(V_1^*,g_1^*),$$ and
essentially two exterior algebras:
$$\bigwedge(V_1\cup_{\tilde{f}}V_2)\cong\bigwedge(V_2)\cup_{\tilde{f}^{\bigwedge}}\bigwedge(V_1),$$ to which we will also add the contravariant exterior algebra 
$\bigwedge_*(V_2^*\cup_{\tilde{f}^*}V_1^*)$.

\paragraph{Summary of Clifford actions} We now outline which Clifford algebra (or the result of gluing of such) acts on which pseudo-bundle of exterior algebras:
\begin{itemize}
\item $\cl((V_1\cup_{\tilde{f}}V_2)^*,\tilde{g}^*)$ acts on $\bigwedge(V_1\cup_{\tilde{f}}V_2)$ via the standard Clifford action $c^*$;
\item $\cl(V_2^*,g_2^*)\cup_{(\tilde{F}^*)^{\cl}}\cl(V_1^*,g_1^*)$ acts on $\bigwedge(V_2)\cup_{\tilde{f}^{\bigwedge}}\bigwedge(V_1)$ via the action $c_{*,\cup}$ (see Proposition 6.8) induced by the
standard Clifford actions $c_2^*$ and $c_1^*$ of $\cl(V_2^*,g_2^*)$ and $\cl(V_1^*,g_1^*)$ on $\bigwedge(V_2)$ and $\bigwedge(V_1)$ respectively;
\item $\cl(V_2^*\cup_{\tilde{f}^*}V_1^*,\widetilde{g^*})$ has, again, the standard Clifford action, which we have not mentioned yet and which we now denote by $\tilde{c}_{*,\cup}$, on
$\bigwedge_*(V_2^*\cup_{\tilde{f}^*}V_1^*)$.
\end{itemize}

\paragraph{The equivalence of actions} As in the contravariant case, the actions $c^*$, $c_{*,\cup}$, and $\tilde{c}_{*,\cup}$ turn out to be equivalent, with the equivalence established via the
diffeomorphisms $\Phi^{\cl(*)}$, $\Phi_{\cup,*}^{\cl}$, and $\Phi_{\cup}^{\cl(*)}$, as well as $\Phi^{\bigwedge}$ and $\Phi_{\cup,*}^{\bigwedge}:\bigwedge(V_1\cup_{\tilde{f}}V_2)\to
\bigwedge_*(V_2^*\cup_{\tilde{f}^*}V_1^*)$. Specifically, we have:
\begin{itemize}
\item the action $c^*$ of $\cl((V_1\cup_{\tilde{f}}V_2)^*,\tilde{g}^*)$ on $\bigwedge(V_1\cup_{\tilde{f}}V_2)$ is related to the action $\tilde{c}_{*,\cup}$ of
$\cl(V_2^*\cup_{\tilde{f}^*}V_1^*,\widetilde{g^*})$ on $\bigwedge_*(V_2^*\cup_{\tilde{f}^*}V_1^*)$, with respect to the diffeomorphisms $\Phi_{\cup,*}^{\cl}$ and $\Phi_{\cup,*}^{\bigwedge}$, via
$$\Phi_{\cup,*}^{\bigwedge}(c^*(v)(e))=\tilde{c}_{*,\cup}(\Phi_{\cup,*}^{Cl}(v))(\Phi_{\cup,*}^{\bigwedge}(e)),$$ that holds for all $v\in\cl((V_1\cup_{\tilde{f}}V_2)^*,\tilde{g}^*)$
and $e\in\bigwedge(V_1\cup_{\tilde{f}}V_2)$ such that $\pi^{\bigwedge}(e)=\pi^{\cl}(v)$;
\item the action $c_{*,\cup}$ of $\cl(V_2^*,g_2^*)\cup_{(\tilde{F}^*)^{\cl}}\cl(V_1^*,g_1^*)$ on $\bigwedge(V_2)\cup_{\tilde{f}^{\bigwedge}}\bigwedge(V_1)$ is equivalent to the action $\tilde{c}_{*,\cup}$ of
$\cl(V_2^*\cup_{\tilde{f}^*}V_1^*,\widetilde{g^*})$ on $\bigwedge_*(V_2^*\cup_{\tilde{f}^*}V_1^*)$, with respect to the diffeomorphisms $\Phi^{\cl(*)}$ and $\Phi_{\cup}^{\cl(*)}$, via
$$\left(\Phi_{\cup,*}^{\bigwedge}\circ(\Phi^{\bigwedge})^{-1}\right)(c_{*,\cup}(v)(e))=\tilde{c}_{*,\cup}\left(\Phi^{cl(*)}(v)\right)\left((\Phi_{\cup,*}^{\bigwedge}\circ(\Phi^{\bigwedge})^{-1})(e)\right)$$
that is true for all $v\in\cl(V_2^*,g_2^*)\cup_{(\tilde{F}^*)^{\cl}}\cl(V_1^*,g_1^*)$ and $e\in\bigwedge(V_2)\cup_{\tilde{f}^{\bigwedge}}\bigwedge(V_1)$ such that
$(\pi_2^{\bigwedge}\cup_{(\tilde{f}^{\bigwedge},f^{-1})}\pi_1^{\bigwedge})(e)=(\pi_2^{\cl}\cup_{((\tilde{F}^*)^{\cl},f^{-1})}\pi_1^{\cl})(v)$.
\end{itemize}
Notice that these equivalences imply also the equivalence of $c^*$ to $\tilde{c}_{*,\cup}$.

\subsection{The standard Clifford action is smooth}

The basis for several versions of the Clifford(-type) actions listed above is the usual action of the (contravariant) Clifford algebra on the corresponding (also contravariant) exterior algebra. This means
the following.

Let $\pi:V\to X$ be any locally trivial finite-dimensional diffeological vector pseudo-bundle that admits a pseudo-metric; let $g$ be a fixed choice of a pseudo-metric on it. The \textbf{standard
Clifford action} of $\cl(V,g)$ on $\bigwedge_*(V)$ is the map $c:\cl(V,g)\to\mathcal{L}(\bigwedge_*(V),\bigwedge_*(V))$ given by 
$$c(v)(v_1\wedge\ldots\wedge v_k)=v\wedge v_1\wedge\ldots\wedge v_k-\sum_{j=1}^k(-1)^{j+1}v_1\wedge\ldots\wedge g(\pi(v))(v,v_j)\wedge\ldots\wedge v_k.$$ On each fibre $\pi^{-1}(x)$ of $V$, this is 
the usual Clifford action of the Clifford algebra relative to the bilinear symmetric form $g(x)$ on the exterior algebra of $\pi^{-1}(x)$.

\begin{prop}\label{standard:clifford:action:is:smooth:lem}
The action $c$ is smooth as a map $\cl(V,g)\to\mathcal{L}(\bigwedge_*(V),\bigwedge_*(V))$.
\end{prop}

\begin{proof}
Notice first of all that the pseudo-bundle $\bigwedge_*(V)$ smoothly splits as the direct sum $\bigoplus_k\bigwedge^kV$. It then follows from the above presentation of the action $c$ and the definition of
the diffeology of $\cl(V,g)$, that it suffices to show that the following two maps $c_V,c_j:V\to\mathcal{L}(\bigwedge_*(V),\bigwedge_*(V))$ are smooth:
$$c_V(v)(v_1\wedge\ldots\wedge v_k)=v\wedge v_1\wedge\ldots\wedge v_k\,\mbox{ and }\,c_j(v)(v_1\wedge\ldots\wedge v_k)=v_1\wedge\ldots\wedge g(\pi(v))(v,v_j)\wedge\ldots\wedge v_k.$$
Thus, $c_V$ acts as the exterior product, which is smooth by definition (recall that the diffeology on each exterior product degree can be described as the pushforward of the tensor product
diffeology by the alternating operator, which makes it, and the exterior product as a consequence, automatically smooth).

The smoothness of the map $c_j$ follows from the smoothness of the pseudo-metric $g$. To be slightly more explicit, we note that on a small enough neighborhood $U$, we can write a plot of the $k$-th
exterior degree as $(p_1,\ldots,p_k)$, where each $p_i$ is a plot of $V$, acting by $u\mapsto p_1(u)\wedge\ldots\wedge p_k(u)$. Therefore the evaluation map that determines the smoothness of $c_j$ is
locally of form
$$(u',u)\mapsto p_1(u)\wedge\ldots\wedge g(\pi(p(u')))(p(u'),p_j(u))\wedge\ldots\wedge p_k(u)$$ for some other plot $p:U'\to V$ of $V$. Since $(u',u)\mapsto g(\pi(p(u')))(p(u'),p_j(u))$ is a smooth function, 
and the diffeology of $\bigwedge_*(V)$ is a (vector) pseudo-bundle diffeology, we obtain a plot of $\bigwedge_*(V)$, whence the claim.
\end{proof}

\subsection{The compatibility of two standard Clifford actions}

Likewise, we can show that under certain assumptions, two standard Clifford actions are compatible with a given gluing; this happens precisely when the gluing itself is commutative. Here is the precise
statement.

\begin{prop}\label{standard:clifford:actions:compatible:prop}
Let $\pi_1:V_1\to X_1$ and $\pi_2:V_2\to X_2$ be two locally trivial finite-dimensional diffeological vector pseudo-bundles, let $(\tilde{f},f)$ be a gluing between them such that $f$ and
$\tilde{f}$ are diffeomorphisms, and let $g_1$ and $g_2$ be compatible pseudo-metrics on $V_1$ and $V_2$ respectively. Let $c_i$ for $i=1,2$ be the standard Clifford actions of $\cl(V_i,g_i)$ on
$\bigwedge_*(V_i)$. Then for all $v,v_1,\ldots,v_k\in V_1$ such that $\pi_1(v)=\pi_1(v_1)=\ldots=\pi_1(v_k)\in Y$ we have
$$\tilde{f}^{\bigwedge_*}(c_1(v)(v_1\wedge\ldots\wedge v_k))=c_2(\tilde{F}^{\cl}(v))(\tilde{f}^{\bigwedge_*}(v_1\wedge\ldots\wedge v_k)).$$
\end{prop}

\begin{proof}
By the definition of $\tilde{F}^{\cl}$ and that of $\tilde{f}^{\bigwedge_*}$, we have that
$$c_2(\tilde{F}^{\cl}(v))(\tilde{f}^{\bigwedge_*}(v_1\wedge\ldots\wedge v_k))=c_2(\tilde{f}(v))(\tilde{f}(v_1)\wedge\ldots\wedge\tilde{f}(v_k)).$$ The desired condition easily follows from this. Indeed,
\begin{flushleft}
$\tilde{f}^{\bigwedge_*}(c_1(v)(v_1\wedge\ldots\wedge v_k))=$
\end{flushleft}
\begin{flushright}
$=\tilde{f}^{\bigwedge_*}(v\wedge v_1\wedge\ldots\wedge v_k-\sum_{j=1}^k(-1)^{j+1}g_1(\pi_1(v))(v,v_j)v_1\wedge\ldots\wedge v_{j-1}\wedge v_{j+1}\wedge\ldots\wedge v_k)=$
\end{flushright}
\begin{flushleft}
$=\tilde{f}(v)\wedge\tilde{f}(v_1)\wedge\ldots\wedge\tilde{f}(v_k)-$
\end{flushleft}
\begin{flushright}
$-\sum_{j=1}^k(-1)^{j+1}g_1(\pi_1(v))(v,v_j)\tilde{f}(v_1)\wedge\ldots\wedge\tilde{f}(v_{j-1})\wedge\tilde{f}(v_{j+1})\wedge\ldots\wedge\tilde{f}(v_k)).$
\end{flushright}
Now, since
\begin{flushleft}
$c_2(\tilde{f}(v))(\tilde{f}(v_1)\wedge\ldots\wedge\tilde{f}(v_k))=\tilde{f}(v)\wedge\tilde{f}(v_1)\wedge\ldots\wedge\tilde{f}(v_k)-$
\end{flushleft}
\begin{flushright}
$-\sum_{j=1}^k(-1)^{j+1}g_2(\pi_2(\tilde{f}(v)))(\tilde{f}(v),\tilde{f}(v_j))\tilde{f}(v_1)\wedge\ldots\wedge\tilde{f}(v_{j-1})\wedge\tilde{f}(v_{j+1})\wedge\ldots\wedge\tilde{f}(v_k))$.
\end{flushright}
The pseudo-metrics $g_1$ and $g_2$ being compatible ensures that $g_2(\pi_2(\tilde{f}(v)))(\tilde{f}(v),\tilde{f}(v_j))=g_1(\pi_1(v))(v,v_j)$, whence the claim.
\end{proof}

\subsection{The contravariant case: the actions on $\bigwedge_*(V_1\cup_{\tilde{f}}V_2)$ and $\bigwedge_*(V_1)\cup_{\tilde{f}^{\bigwedge_*}}\bigwedge_*(V_2)$}

We now describe the action $\tilde{c}_*$, and prove its equivalence (already stated above) to the action $c_*$. Notice that $c_*$ is an instance of the standard Clifford action, so it is smooth by Lemma
\ref{standard:clifford:action:is:smooth:lem}.

\paragraph{The action $\tilde{c}_*=c_1\cup_{(\tilde{F}^{Cl},\tilde{f}^{\bigwedge_*})}c_2$} This is a partial case of a more general construction described in \cite{clifford-alg}. The construction bears 
some similarity to that of the gluing of smooth maps, although, as mentioned in the same source, it is not quite the same thing. To describe this action, let $v\in cl(V_1,g_1)\cup_{\tilde{F}^{\cl}}\cl(V_2,g_2)$; 
then $\tilde{c}_*(v)$ is an endomorphism of the fibre of $\bigwedge_*(V_1)\cup_{\tilde{f}^{\bigwedge_*}}\bigwedge_*(V_2)$ over the point
$(\pi_1^{\cl}\cup_{(\tilde{F}^{\cl},f)}\pi_2^{\cl})(v)\in X_1\cup_f X_2$. Then the action $\tilde{c}_*$ is defined as follows:
$$\tilde{c}_*(v)(e)=\left\{\begin{array}{ll}
j_1^{\bigwedge_*(V_1)}\left(c_1((j_1^{\cl(V_1,g_1)})^{-1}(v))((j_1^{\bigwedge_*(V_1)})^{-1}(e))\right) & \mbox{over }i_1^{X_1}(X_1\setminus Y)\\
j_2^{\bigwedge_*(V_2)}\left(c_2((j_2^{\cl(V_2,g_2)})^{-1}(v))((j_2^{\bigwedge_*(V_2)})^{-1}(e))\right) & \mbox{over }i_2^{X_2}(X_2).
\end{array}\right.$$ In other words, we just pull back $v$ and $e$ to the respective factors of gluing, apply $c_1$ or $c_2$, as appropriate, and re-insert the result into the pseudo-bundle
$\bigwedge_*(V_1)\cup_{\tilde{f}^{\bigwedge_*}}\bigwedge_*(V_2)$. It now suffices to note that by Proposition \ref{standard:clifford:actions:compatible:prop} $c_1$ and $c_2$ are
compatible as Clifford actions, so it follows from \cite{clifford-alg} that the action $\tilde{c}_{*,\cup}$ is smooth.

\paragraph{The equivalence of $c_*$ to $\tilde{c}_{*,\cup}$} We now prove the already-mentioned statement of equivalence for these actions.

\begin{thm}\label{gluing:compatible:actions:commutes:with:standard:thm}
Let $v\in\cl(V_1\cup_{\tilde{f}}V_2,\tilde{g})$ and $e\in\bigwedge_*(V_1\cup_{\tilde{f}}V_2)$ be such that $\pi^{\cl}(v)=\pi^{\bigwedge_*}(e)$. Then
$$\Phi^{\bigwedge_*}(c_*(v)(e))=\tilde{c}_*\left((\Phi^{\cl})^{-1}(v)\right)\left(\Phi^{\bigwedge_*}(e)\right).$$
\end{thm}

\begin{proof}
The proof is almost trivial if we adopt the following viewpoint: since both actions are fibrewise based on the standard Clifford action, it suffices to assume that $v$ is an element of the copy of
$V_1\cup_{\tilde{f}}V_2$ naturally contained in $\cl(V_1\cup_{\tilde{f}}V_2,\tilde{g})$, and that $e$ belongs to the $k$-th exterior degree of $V_1\cup_{\tilde{f}}V_2$, that is,
$e=v_1\wedge\ldots\wedge v_k$ for $v_1,\ldots,v_k\in V_1\cup_{\tilde{f}}V_2$. Formally, there are two cases to consider: that of $\pi^{\cl}(v)\in i_1^{X_1}(X_1\setminus Y)$ and that of
$\pi^{\cl}(v)\in i_2^{X_2}(X_2)$.

Thus, suppose that $\pi^{\cl}(v)\in i_1^{X_1}(X_1\setminus Y)$. Since $c_*$ is the standard action, we have 
$$c_*(v)(e)=v\wedge v_1\wedge\ldots\wedge v_k-\sum_{j=1}^k(-1)^{j+1}v_1\wedge\ldots\wedge\tilde{g}(\pi^{\cl}(v))(v,v_j)\wedge \ldots\wedge v_k.$$ Therefore
\begin{flushleft}
$\Phi^{\bigwedge_*}(c_*(v)(e))=j_1^{\bigwedge_*(V_1)}((j_1^{V_1})^{-1}(v)\wedge(j_1^{V_1})^{-1}(v_1)\wedge\ldots\wedge(j_1^{V_1})^{-1}(v_k)- $
\end{flushleft}
\begin{flushright}
$-\sum_{j=1}^k(-1)^{j+1}(j_1^{V_1})^{-1}(v_1)\wedge\ldots\wedge g_1(\pi_1((j_1^{V_1})^{-1}(v)))((j_1^{V_1})^{-1}(v),(j_1^{V_1})^{-1}(v_j))\wedge\ldots\wedge(j_1^{V_1})^{-1}(v_k))$.
\end{flushright}
It therefore suffices to note that $j_1^{\cl(V_1,g_1)}((j_1^{V_1})^{-1}(v))=(\Phi^{\cl})^{-1}(v)$ and 
$j_1^{\bigwedge_*(V_1)}((j_1^{V_1})^{-1}(v_1)\wedge\ldots\wedge(j_1^{V_1})^{-1}(v_k))=\Phi^{\bigwedge_*}(v_1\wedge\ldots\wedge v_k)$, to draw the desired conclusion. Since the treatment of the case 
$\pi^{\cl}(v)\in i_2^{X_2}(X_2)$ is exactly the same, the proof is finished.
\end{proof}

\subsection{The covariant case}

In this case we have three potential actions, corresponding to the three shapes of the Clifford algebra and those of the three exterior algebras (one of which is actually a contravariant algebra,
trivially identified to a covariant one). After a detailed description of the actions, we prove their equivalences, already announced in Section 8.1.2.

\subsubsection{The action $c^*$ of $\cl((V_1\cup_{\tilde{f}}V_2)^*,\tilde{g}^*)$ on $\bigwedge(V_1\cup_{\tilde{f}}V_2)$}

This is a case of a standard Clifford action, considered in Lemma \ref{standard:clifford:action:is:smooth:lem}; this lemma, in particular, ensures, that $c^*$ is a smooth action. Recall indeed
that $\bigwedge(V_1\cup_{\tilde{f}}V_2)=\bigwedge_*((V_1\cup_{\tilde{f}}V_2)^*)$ by its definition.

\subsubsection{The action $c_{*,\cup}$ of $\cl(V_2^*,g_2^*)\cup_{(\tilde{F}^*)^{\cl}}\cl(V_1^*,g_1^*)$ on $\bigwedge(V_2)\cup_{\tilde{f}^{\bigwedge}}\bigwedge(V_1)$}

Let $c_2^*:\cl(V_2^*,g_2^*)\to\mathcal{L}(\bigwedge(V_2),\bigwedge(V_2))$ and $c_1^*:\cl(V_1^*,g_1^*)\to\mathcal{L}(\bigwedge(V_1),\bigwedge(V_1))$ be the standard Clifford actions. By Proposition
\ref{standard:clifford:actions:compatible:prop}, they are compatible with the gluings that yield respectively $\cl(V_2^*,g_2^*)\cup_{(\tilde{F}^*)^{\cl}}\cl(V_1^*,g_1^*)$ and
$\bigwedge(V_2)\cup_{\tilde{f}^{\bigwedge}}\bigwedge(V_1)$, with respect to the maps $\tilde{f}^{\bigwedge}$ and $(\tilde{F}^*)^{\cl}$. Thus, the procedure described in
\cite{clifford-alg} yields a smooth action $c_{*,\cup}=c_2^*\cup_{((\tilde{F}^*)^{\cl},\tilde{f}^{\bigwedge})}c_1^*$. The formula that describes it is as follows.

Let $v\in\cl(V_2^*,g_2^*)\cup_{(\tilde{F}^*)^{\cl}}\cl(V_1^*,g_1^*)$ and $e\in\bigwedge(V_2)\cup_{\tilde{f}^{\bigwedge}}\bigwedge(V_1)$ be such that
$(\pi_2^{\bigwedge}\cup_{(\tilde{f}^{\bigwedge},f^{-1})}\pi_1^{\bigwedge})(e)=(\pi_2^{\cl}\cup_{((\tilde{F}^*)^{\cl},f^{-1})}\pi_1^{\cl})(v)$. Then we will have
$$c_{*,\cup}(v)(e)=\left\{\begin{array}{ll}
j_1^{\bigwedge(V_2)}\left(c_2^*((j_1^{\cl(V_2^*,g_2^*)})^{-1}(v))((j_1^{\bigwedge(V_2)})^{-1}(e))\right) & \mbox{over }i_1^{X_2}(X_2\setminus f(Y))\\
j_2^{\bigwedge(V_1)}\left(c_1^*((j_2^{\cl(V_1^*,g_1^*)})^{-1}(v))((j_2^{\bigwedge(V_1)})^{-1}(e))\right) & \mbox{over }i_2^{X_1}(X_1).
\end{array}\right.$$

\subsubsection{The action $\tilde{c}_{*,\cup}$ of $\cl(V_2^*\cup_{\tilde{f}^*}V_1^*,\widetilde{g^*})$ on $\bigwedge_*(V_2^*\cup_{\tilde{f}^*}V_1^*)$}

This is also a standard Clifford action; its smoothness follows, once again, from Lemma \ref{standard:clifford:action:is:smooth:lem}.

\subsubsection{The equivalence of $c^*$ to $\tilde{c}_{*,\cup}$, and that of $c_{*,\cup}$ to $\tilde{c}_{*,\cup}$}

It now remains to prove the two equivalence formulae for the covariant actions $c^*$, $c_{*,\cup}$, and $\tilde{c}_{*,\cup}$.

\paragraph{The diffeomorphism $\Phi_{\cup,*}^{\bigwedge}:\bigwedge(V_1\cup_{\tilde{f}}V_2) \to\bigwedge_*(V_2^*\cup_{\tilde{f}^*}V_1^*)$} This diffeomorphism is induced by the gluing-dual commutativity 
diffeomorphism $\Phi_{\cup,*}$ in a completely standard way and is, by definition, the map
$$\Phi_{\cup,*}^{\bigwedge}=\bigoplus_{n=0}^{\dim_{V^*}}\Phi_{\cup,*}^{\otimes n}|_{\bigwedge(V_1\cup_{\tilde{f}}V_2)}.$$ That it has, in particular, the desired range follows from the
commutativity of $\Phi_{\cup,*}^{\otimes n}$ with the relevant alternating operators.

\paragraph{The equivalence of $c^*$ to $\tilde{c}_{*,\cup}$} We now show that the action $c^*$ of $\cl((V_1\cup_{\tilde{f}}V_2)^*,\tilde{g}^*)$ on $\bigwedge(V_1\cup_{\tilde{f}}V_2)$ is equivalent to the action
$\tilde{c}_{*,\cup}$ of $\cl(V_2^*\cup_{\tilde{f}^*}V_1^*,\widetilde{g^*})$ on $\bigwedge_*(V_2^*\cup_{\tilde{f}^*}V_1^*)$, via the rule described in the following statement.

\begin{thm}
Let $v\in\cl((V_1\cup_{\tilde{f}}V_2)^*,\tilde{g}^*)$ and $e\in\bigwedge(V_1\cup_{\tilde{f}}V_2)$ be such that $\pi^{\bigwedge}(e)=\pi^{\cl}(v)$. Then
$$\Phi_{\cup,*}^{\bigwedge}(c^*(v)(e))=\tilde{c}_{*,\cup}(\Phi_{\cup,*}^{\cl}(v))(\Phi_{\cup,*}^{\bigwedge}(e)).$$
\end{thm}

\begin{proof}
The proof uses the same kind of reasoning as that of Theorem \ref{gluing:compatible:actions:commutes:with:standard:thm}, in which it suffices observe that both diffeomorphisms $\Phi^{\cl(*)}$ and
$\Phi_{\cup,*}^{\bigwedge}$ are based on the same diffeomorphism $\Phi_{\cup,*}$, and that the latter essentially commutes with the exterior product.
\end{proof}

\paragraph{The equivalence of $c_{*,\cup}$ to $\tilde{c}_{*,\cup}$} Let us now show that the action $c_{*,\cup}$ of $\cl(V_2^*,g_2^*)\cup_{(\tilde{F}^*)^{\cl}}\cl(V_1^*,g_1^*)$ on
$\bigwedge(V_2)\cup_{\tilde{f}^{\bigwedge}}\bigwedge(V_1)$ is equivalent to the action $\tilde{c}_{*,\cup}$ of $\cl(V_2^*\cup_{\tilde{f}^*}V_1^*,\widetilde{g^*})$ on
$\bigwedge_*(V_2^*\cup_{\tilde{f}^*}V_1^*)$. Specifically, we have the following statement.

\begin{thm}
Let $v\in\cl(V_2^*,g_2^*)\cup_{(\tilde{F}^*)^{\cl}}\cl(V_1^*,g_1^*)$ and $e\in\bigwedge(V_2)\cup_{\tilde{f}^{\bigwedge}}\bigwedge(V_1)$ be such that
$(\pi_2^{\bigwedge}\cup_{(\tilde{f}^{\bigwedge},f^{-1})}\pi_1^{\bigwedge})(e)=(\pi_2^{\cl}\cup_{((\tilde{F}^*)^{\cl},f^{-1})}\pi_1^{\cl})(v)$. Then
$$\left(\Phi_{\cup,*}^{\bigwedge}\circ(\Phi^{\bigwedge})^{-1}\right)(c_{*,\cup}(v)(e))=\tilde{c}_{*,\cup}\left(\Phi^{\cl(*)}(v)\right)\left((\Phi_{\cup,*}^{\bigwedge}\circ(\Phi^{\bigwedge})^{-1})(e)\right).$$
\end{thm}

\begin{proof}
This is a direct consequence of Theorem \ref{gluing:compatible:actions:commutes:with:standard:thm} applied to $(V_2^*,g_2^*)$, $(V_1^*,g_1^*)$, and $(\tilde{f}^*,f^{-1})$.
\end{proof}

\section{Examples}

The two examples that we describe in this section are chosen with the following considerations in mind. For one thing, even when we start with some usual smooth vector bundles (as in the first example
below), the gluing of them may be defined on a non-open set, producing a result which is not a smooth manifold, yet is being treated as if it were one. In the second example we consider a
non-diffeomorphic gluing of fibres.

\subsection{The wedge of two lines}

We start with the two pseudo-bundles $\pi_1:V_1\to X_1$ and $\pi_2:V_2\to X_2$. Let $V_1=V_2=\matR^2$ with the standard diffeology, let $X_1=X_2=\matR$, also standard, and let $\pi_1$ and
$\pi_2$ be the two standard projections on the $x$-axis: $\pi_1:V_1\ni(x,y)\mapsto x\in\matR=X_1$ and $\pi_2:V_2\ni(x,y)\mapsto x\in\matR=X_2$. The pseudo-bundle
structure is given by imposing on each fibre $(x,y_1)+(x,y_2)\mapsto(x,y_1+y_2)$ and $\lambda(x,y_1)\mapsto(x,\lambda y_1)$. The gluing of these two pseudo-bundles is given by $(\tilde{f},f)$, where
$f:X_1\supset\{0\}\to\{0\}\subset X_2$ and $\tilde{f}$ is determined by a non-zero constant $a\in\matR$ via the rule
$$\tilde{f}(0,1)=(0,a)\in V_2=\matR^2.$$

Let $f_1,f_2:\matR\to\matR$ be two smooth everywhere positive functions; let $g_i$ be the pseudo-metric on $V_i$ given by $$g_i(x)(v,w)=f_i(x)\cdot e^2(v)\cdot e^2(w),$$ where $e^2$ is the second 
element of the usual dual basis of the canonical basis of $\matR^2$ (relative to the notation used later on it can be written as $dy$). The compatibility condition is then $f_1(0)=a^2f_2(0)$.

\paragraph{The result of gluing} The pseudo-bundle $V_1\cup_{\tilde{f}}V_2$ can be described as the union $\{(x,0,z)\}\cup\{(0,y,z)\}$ of two planes in $\matR^3$, and, accordingly, $X_1\cup_f X_2$ as the 
union $\{(x,0,0)\}\cup\{(0,y,0)\}$ of the two axes, with the projection $\pi_1\cup_{(\tilde{f},f)}\pi_2$ acting by $(x,0,z)\mapsto(x,0,0)$, $(0,y,z)\mapsto(0,y,0)$. The pseudo-metric $\tilde{g}$ is then
$$\tilde{g}(x,y,0)=\left\{\begin{array}{ll} f_1(x)dz^2 & \mbox{if }y=0\mbox{ and }x\neq 0,\\ f_2(y)dz^2 & \mbox{if }x=0. \end{array}\right.$$

\paragraph{The pairing maps $\Psi_{g_1}$, $\Psi_{g_2}$, and $\Psi_{\tilde{g}}$} Since all fibres are standard, the characteristic sub-bundles coincide with the pseudo-bundles
themselves, so all three maps are automatically invertible. They act by:
$$\Psi_{g_1}(x,y)=f_1(x)ydy,\,\,\,\Psi_{g_2}(x,y)=f_2(x)ydy,$$ and then, using the just-mentioned presentation of $V_1\cup_{\tilde{f}}V_2$ as a subset of $\matR^3$, we have
$$\Psi_{\tilde{g}}(x,y,z)=\left\{\begin{array}{ll} f_1(x)zdz & \mbox{if }y=0,\\ f_2(y)zdz & \mbox{if }x=0 \end{array}\right.$$

\paragraph{The dual pseudo-metrics} The dual pseudo-metrics are therefore described in the same manner as $g_1$ and $g_2$, but the coefficients are inverted:
$$g_2^*(x)(v^*,w^*)=\frac{1}{f_2(x)}\cdot v^*(e_2)\cdot w^*(e_2)\,\,\,\mbox{ and }\,\,\,g_1^*(x)(v^*,w^*)=\frac{1}{f_1(x)}\cdot v^*(e_2)\cdot w^*(e_2).$$ The compatibility condition for $g_2^*$ and $g_1^*$ is
thus $\frac{1}{f_2(0)}=\frac{a^2}{f_1(0)}$, and so is equivalent to the one for $g_1$ and $g_2$. The pseudo-metric $\tilde{g}^*\equiv\widetilde{g^*}$ is therefore
$$\tilde{g}^*(x',y',0)=\left\{\begin{array}{ll}
\frac{1}{f_1(x')}\frac{\partial}{\partial z}\otimes\frac{\partial}{\partial z} & \mbox{if }y'=0,\\
\frac{1}{f_2(x')}\frac{\partial}{\partial z}\otimes\frac{\partial}{\partial z} & \mbox{if }x'=0.
\end{array}\right.$$

\paragraph{The pseudo-bundles of Clifford algebras} All fibres in our case are $1$-dimensional, so each of $\cl(V_1,g_1)$, $\cl(V_2,g_2)$ is thus a trivial fibering of $\matR^3$ over $\matR$;
the result of their gluing can be described as the subset in $\matR^4$ given by the equation $xy=0$, so that
$$\cl(V_1,g_1)\cup_{\tilde{F}^{\cl}}\cl(V_2,g_2)=\{(x,0,z,w),\mbox{ where }x\neq 0\}\cup\{(0,y,z,w)\},$$ with the Clifford multiplication being defined by
$$(x,0,z_1,w_1)\cdot_{\cl}(x,0,z_2,w_2)=(x,0,z_1w_2+z_2w_1,-f_1(x)z_1z_2+w_1w_2),$$
$$(0,y,z_1,w_1)\cdot_{\cl}(0,y,z_2,w_2)=(0,y,z_1w_2+z_2w_1,-f_2(y)z_1z_2+w_1w_2).$$

From this, it is also quite evident that the result trivially coincides with $\cl(V_1\cup_{\tilde{f}}V_2,\tilde{g})$, so much in fact, that we can only distinguish between the two by choosing two
slightly different forms of designating the same subset in $\matR^4$. Specifically, in the case of $\cl(V_1\cup_{\tilde{f}}V_2,\tilde{g})$ we describe the set of its
points as
$$\{(x,y,z,w),\mbox{ where }xy=0\}.$$ Obviously, this is the same set as we described as the set of points of $\cl(V_1,g_1)\cup_{\tilde{F}^{\cl}}\cl(V_2,g_2)$; the chosen
presentation of the latter emphasizes its structure as the result of a gluing.

\paragraph{The pseudo-bundles of covariant Clifford algebras} Consider again the subset in $\matR^4$ given by the equation $xy=0$. This is the subset that is identified with all three (shapes of) the
Clifford algebra. For all three possibilities, we identify the copy of $V_1^*$ contained in either of them with the hyperplane $\{(x,0,z,0)\}$, and the copy of $V_2^*$, with the hyperplane
$\{(0,y,z,0)\}$; the fourth coordinate $w$ corresponds to the scalar part of the Clifford algebra.

The distinction between the various shapes of the Clifford algebra is the following one. When this subset is viewed as $\cl((V_1\cup_{\tilde{f}}V_2)^*,\tilde{g}^*)$, we describe the Clifford multiplication as
$$\left\{\begin{array}{l}
(x,0,z_1,w_1)\cdot_{\cl}(x,0,z_2,w_2)=(x,0,z_1w_2+z_2w_1,-\frac{1}{f_1(x)}z_1z_2+w_1w_2)\mbox{ for }x\neq 0,\\
(0,y,z_1,w_1)\cdot_{\cl}(0,y,z_2,w_2)=(0,y,z_1w_2+z_2w_1,-\frac{1}{f_2(2)}z_1z_2+w_1w_2)\mbox{ otherwise}.
\end{array}\right.$$ On the other hand, when we view the same subset as either $\cl(V_2^*,g_2^*)\cup_{(\tilde{F}^*)^{\cl}}\cl(V_1^*,g_1^*)$ or $\cl(V_2^*\cup_{\tilde{f}^*}V_1^*,\widetilde{g^*})$, we describe the
corresponding product by
$$\left\{\begin{array}{l}
(x,0,z_1,w_1)\cdot_{\cl}(x,0,z_2,w_2)=(x,0,z_1w_2+z_2w_1,-\frac{1}{f_1(x)}z_1z_2+w_1w_2)\mbox{ for all }x,\\
(0,y,z_1,w_1)\cdot_{\cl}(0,y,z_2,w_2)=(0,y,z_1w_2+z_2w_1,-\frac{1}{f_2(2)}z_1z_2+w_1w_2)\mbox{ for }y\neq 0.
\end{array}\right.$$

\paragraph{The pseudo-bundles of exterior algebras} These can be presented in exactly the same way as those of Clifford algebras. In both the contravariant and the covariant case we have a unique
presentation, again as a subset of $\matR^4$ given by the equation $xy=0$, with the exterior product
$$\left\{\begin{array}{l}
(x,0,z_1,w_1)\wedge(x,0,z_2,w_2)=(x,0,z_1w_2+z_2w_1,w_1w_2),\\
(0,y,z_1,w_1)\wedge(0,y,z_2,w_2)=(0,y,z_1w_2+z_2w_1,w_1w_2)
\end{array}\right.$$

\paragraph{The Clifford actions} In the contravariant case, we have two exterior algebras, $\bigwedge_*(V_1\cup_{\tilde{f}}V_2)$ and $\bigwedge_*(V_1)\cup_{\tilde{f}_*^{\bigwedge}}\bigwedge_*(V_2)$,
with the actions $c$ and $\tilde{c}$ of, respectively, $\cl(V_1\cup_{\tilde{f}}V_2,\tilde{g})$ and $\cl(V_1,g_1)\cup_{\tilde{F}^{\cl}}\cl(V_2,g_2)$. In the former case,
we have
\begin{flushleft}
$c((0,0,z_2,w_2))(0,0,z,w)=(0,0,z_2,w_2)\wedge(0,0,z,w)-\left(0,0,0,\tilde{g}(0,0,0)((0,0,z_2),(0,0,z))\right)=$
\end{flushleft}
\begin{flushright}
$=(0,0,z_2w+w_2z,w_2w+f_2(0)z_2z)$;
\end{flushright}
in the latter case, the only thing that changes with respect to the formula just given, is that the term $\tilde{g}(0,0,0)((0,0,z_2),(0,0,z)$ is replaced by the term
$g_2(0,0)((0,z_2),(0,z))$, whose value however is exactly the same.

The covariant case is analogous, although we have three exterior algebras, $\bigwedge(V_1\cup_{\tilde{f}}V_2)$, $\bigwedge_*(V_2^*\cup_{\tilde{f}^*}V_1^*)$, and
$\bigwedge(V_2)\cup_{(\tilde{f}^*)^{\bigwedge}}\bigwedge(V_1)$, with the actions $c$, $c_{*,\cup}$, and $c_{\cup,*}$ of, respectively, $\cl((V_1\cup_{\tilde{f}}V_2)^*,\tilde{g}^*)$,
$\cl(V_2^*\cup_{\tilde{f}^*}V_1^*,\widetilde{g^*})$, and $\cl(V_2^*,g_2^*)\cup_{(\tilde{F}^*)^{\cl}}\cl(V_1^*,g_1^*)$. Once again, these actions have the same form everywhere except over the
point of gluing (the origin), where we would formally write the formulae for $c((0,0,z_2,w_2))(0,0,z,w)$, $c_{*,\cup}((0,0,z_2,w_2))(0,0,z,w)$, and
$c_{\cup,*}((0,0,z_2,w_2))(0,0,z,w)$ with respect to $\tilde{g}^*$, $\widetilde{g^*}$, or $g_1^*$, respectively:
\begin{flushleft}
$c((0,0,z_2,w_2))(0,0,z,w)=(0,0,z_2,w_2)\wedge(0,0,z,w)-\left(0,0,0,\tilde{g}^*(0,0,0)((0,0,z_2),(0,0,z))\right)=$
\end{flushleft}
\begin{flushright}
$=(0,0,z_2w+w_2z,w_2w+\frac{1}{f_2(0)}z_2z)$,
\end{flushright}
\begin{flushleft}
$c_{*,\cup}((0,0,z_2,w_2))(0,0,z,w)=(0,0,z_2,w_2)\wedge(0,0,z,w)-\left(0,0,0,\widetilde{g^*}(0,0,0)((0,0,z_2),(0,0,z))\right)=$
\end{flushleft}
\begin{flushright}
$=(0,0,z_2w+w_2z,w_2w-g_1^*(0,0)((0,z_2),(0,z)))=(0,0,z_2w+w_2z,w_2w+\frac{1}{f_1(0)}z_2z)$,
\end{flushright}
\begin{flushleft}
$c_{\cup,*}((0,0,z_2,w_2))(0,0,z,w)=(0,0,z_2,w_2)\wedge(0,0,z,w)-\left(0,0,0,g_1^*(0,0)((0,z_2),(0,z))\right)=$
\end{flushleft}
\begin{flushright}
$=(0,0,z_2w+w_2z,w_2w+\frac{1}{f_1(0)}z_2z)$.
\end{flushright}

\subsection{A non-diffeomorphism $\tilde{f}$ and diffeomorphism $\tilde{f}^*$}

Let $\pi_2:V_2\to X_2$ be the same as in the previous section, \emph{i.e.}, the standard projection $\matR^2\to\matR$; define $\pi_1:V_1\to X_1$ to be the projection of $V_1=\matR^3$ to
$X_1=\matR$, where $X_1$ carries the standard diffeology, and $V_1=\matR\times\matR\times\matR$ carries the product diffeology relative to the standard diffeologies on the first two factors and
the vector space diffeology generated by the plot $\matR\ni x\mapsto|x|$ on the third factor.\footnote{In fact, any non-standard vector space diffeology would be sufficient for our purposes.} The
projection $\pi_1$ is just the projection onto the first factor. The gluing map $f$ for the bases is the same, $\{0\}\to\{0\}$, and the one for the total spaces is almost the same, specifically,
$\tilde{f}(0,y,z)=(0,ay)$ with $a\neq 0$ (again, notice that zeroing out the third coordinate is necessary for $\tilde{f}$ to be smooth). The pseudo-bundle $\pi_2:V_2\to X_2$ carries the same pseudo-metric
$g_2$ as in the previous example, while the pseudo-metric $g_1$ on $\pi_1:V_1\to X_1$ extends the previous one in a trivial manner:
$$g_1(x)((x,y_1,z_1),(x,y_2,z_2))=f_1(x)y_1y_2.$$ The compatibility condition remains the same.

The entire covariant case coincides with that of the example treated in the previous section. We only consider the pseudo-bundle $V_1\cup_{\tilde{f}}V_2$ and the corresponding contravariant
constructions.

\paragraph{The pseudo-bundle $V_1\cup_{\tilde{f}}V_2$} We represent it as a subset in $\matR^4$, specifically as the union of the plane given by the equations $x=0$ and $w=0$ (the part corresponding to
$V_2$), and of the set $\{y=0\}\setminus\{x=0,y=0,w=0\}$; this is the part corresponding to $V_1$, where excising the line $\{x=0,y=0,w=0\}$ reflects how $V_1\cup_{\tilde{f}}V_2$ contains
$V_1\setminus\pi_1^{-1}(Y)$, and not the entire $V_1$. Thus, the entire set can be described as
$$\left\{\begin{array}{ll} (x,0,z,w) & \mbox{except the points }(0,0,z,0)\\ (0,y,z,0) & \mbox{for all }y,z. \end{array}\right.$$

\paragraph{The two Clifford algebras} The Clifford algebra of $V_2$ is the already seen one; relative to the presentation of $V_1\cup_{\tilde{f}}V_2$ given above, we could describe it as a
subset of $\matR^5$, adding the 5th coordinate $u_1$ for the scalar part of $\cl(V_2,g_2)\cong\matR\oplus V_2$. Thus,
$$\cl(V_2,g_2)=\{(0,y,z,0,u_1)\},$$ with the Clifford multiplication given by
$$(0,y,z',0,u_1')\cdot_{\cl}(0,y,z'',0,u_1'')=(0,y,u_1''z'+u_1'z'',0,u_1'u_1''-f_2(y)z'z'').$$

The Clifford algebra $\cl(V_1,g_1)$ is bigger; since the fibres of $V_1$ have dimension $2$, each fibre of $cl(V_1,g_1)$ has dimension $4$. Thus, we represent it as a subset in $\matR^6$, by adding the
coordinates $u_1,u_2$, where $u_1$ corresponds to the scalar part and $u_2$ corresponds to the degree $2$ vector part. Thus,
$$\cl(V_1,g_1)=\{(x,0,z,w,u_1,u_2)\},$$ with the Clifford multiplication given by
\begin{flushleft}
$(x,0,z',w',u_1',u_2')\cdot_{Cl}(x,0,z'',w'',u_1'',u_2'')=$
\end{flushleft}
\begin{flushright}
$=(x,0,z'u_1''+z''u_1',w'u_1''+w''u_1'+f_1(x)w'',u_1'u_1''-f_1(x)z'z'',u_1'u_2''+u_1''u_2')$.
\end{flushright}
Finally, $\cl(V_1\cup_{\tilde{f}}V_2,\tilde{g})$ can be described as the following subset in $\matR^6$:
$$\{(x,y,z,w,u_1,u_2)\mbox{ such that }xy=0,\,x=0\Rightarrow w=u_2=0\},$$ while $\cl(V_1,g_1)\cup_{\tilde{F}^{\cl}}\cl(V_2,g_2)$ is presented as the subset in $\matR^6$ of the following form:
$$\{(x,0,z,w,u_1,u_2)\mbox{ such that }x\neq 0\}\cup\{(0,y,z,0,u_1,0)\mbox{ for all }y,z\}.$$ The fibrewise multiplication is described by uniting the two formulae just given.

\paragraph{The contravariant exterior algebras} Likewise, the exterior algebras $\bigwedge_*(V_1)$ and $\bigwedge_*(V_2)$ are given by the same sets. Both of these we immediately represent as subsets of
$\matR^6$, with the $5$-th coordinate being the scalar part and the $6$-th coordinate being the exterior product corresponding to the exterior product relative to the $3$-rd and the $4$-th coordinates;
in the case of $V_2$, this part is obviously trivial. Thus, we have
$$\bigwedge_*(V_1)=\{(x,0,z,w,u_1,u_2)\},\,\,\,\bigwedge_*(V_2)=\{(0,y,z,0,u_1,0)\},$$ with the exterior product given by
\begin{flushleft}
$(x,0,z',w',u_1',u_2')\wedge(x,0,z'',w'',u_1'',u_2'')=$
\end{flushleft}
\begin{flushright}
$=(x,0,u_1''z'+u_1'z'',u_1''w'+u_1'w'',u_1'u_1'',u_1''u_2'+u_1'u_2''+z'w''-z''w')$,
\end{flushright}
\begin{flushleft}
$(0,y,z',0,u_1',0)\wedge(0,y,z'',0,u_1'',0)=(0,y,u_1''z'+u_1'z'',0,u_1'u_1'',0)$.
\end{flushleft}
The exterior algebras $\bigwedge_*(V_1\cup_{\tilde{f}}V_2)$ and $\bigwedge_*(V_1)\cup_{\tilde{f}^{\bigwedge_*}}\bigwedge_*(V_2)$ are then represented respectively by the sets
$$\bigwedge_*(V_1\cup_{\tilde{f}}V_2)=\{(x,y,z,w,u_1,u_2),\mbox{ where } xy=0,\,x=0\Rightarrow w=u_2=0\},$$
$$\bigwedge_*(V_1)\cup_{\tilde{f}^{\bigwedge_*}}\bigwedge_*(V_2)=\{(x,0,z,w,u_1,u_2)\mbox{ such that }x\neq 0\}\cup\{(0,y,z,0,u_1,0)\}.$$ It is obvious that the two
presentations determine the same set, with the second one possibly giving a better idea of the structure of the set, and the first one allowing for the uniform description of the exterior product, in the
following way:
\begin{flushleft}
$(x,y,z',w',u_1',u_2')\wedge(x,y,z'',w'',u_1'',u_2'')=$
\end{flushleft}
\begin{flushright}
$=(x,y,u_1''z'+u_1'z'',u_1''w'+u_1'w'',u_1'u_1'',u_1''u_2'+u_1'u_2''+z'w''-z''w')$.
\end{flushright}

\paragraph{The Clifford actions} It remains to describe the corresponding Clifford actions. As is standard, in the case of $\cl(V_1,g_1)$, it suffices to consider the action of elements of
form $(x,0,z,0,0,0)$ and $(x,0,0,w,0,0)$ on elements of form $(x,0,z,0,0,0)$, $(x,0,0,w,0,0)$, $(x,0,0,0,u_1,0)$, and $(x,0,0,0,0,u_2)$.

For these elements the multiplication is determined as follows
$$\left\{\begin{array}{l}
c_1(x,0,z,0,0,0)(x,0,z',0,0,0)=(x,0,0,0,-f_1(x)z^2,0)\\
c_1(x,0,z,0,0,0)(x,0,0,w,0,0)=(x,0,0,0,0,zw)\\
c_1(x,0,z,0,0,0)(x,0,0,0,u_1,0)=(x,0,u_1z,0,0,0)\\
c_1(x,0,z,0,0,0)(x,0,0,0,0,u_2)=(x,0,0,-u_2f_1(x)z,0,0)\\
c_1(x,0,0,w,0,0)(x,0,z,0,0,0)=(x,0,0,0,0,-zw)\\
c_1(x,0,0,w,0,0)(x,0,0,w',0,0)=(x,0,0,0,-f_1(x)ww',0)\\
c_1(x,0,0,w,0,0)(x,0,0,0,u_1,0)=(x,0,0,u_1w,0,0)\\
c_1(x,0,0,w,0,0)(x,0,0,0,0,u_2)=(x,0,0,0,0,0).
\end{array}\right.$$ In the case of $\cl(V_2,g_2)$, it suffices to consider the action of $(0,y,z,0,0,0)$ on elements of form $(0,y,z,0,0,0)$ and $(0,y,0,0,u_1,0)$, and we have
$$\left\{\begin{array}{l}
c_2(0,y,z,0,0,0)(0,y,z',0,0,0)=(0,y,0,0,-f_2(y)zz',0) \\
c_2(0,y,z,0,0,0)(0,y,0,0,u_1,0)=(0,y,u_1z,0,0,0)
\end{array}\right.$$
Finally, the Clifford action on both $\bigwedge_*(V_1\cup_{\tilde{f}}V_2)$ and $\bigwedge_*(V_1)\cup_{\tilde{f}^{\bigwedge_*}}\bigwedge_*(V_2)$ is obtained by concatenating the two lists; the difference 
between the two pseudo-bundles is not seen on the level of defining the action, but rather in how we determine the two sets of points (as already been indicated above), underlying the commutativity between 
the gluing and the exterior product.

\vspace{1cm}

\noindent University of Pisa \\
Department of Mathematics \\
Via F. Buonarroti 1C\\
56127 PISA -- Italy\\
\ \\
ekaterina.pervova@unipi.it\\

\end{document}